\documentclass[10pt]{amsart}
\usepackage{amsmath}
\usepackage{amsfonts}
\usepackage{tikz}
\usepackage{physics} 
\usepackage{dsfont} 
\usepackage{xcolor} 
\usepackage{amssymb}
\usepackage{hyperref}
\usepackage{pgfplots}
\usepackage{csquotes}
\usepackage{mathtools}
\usepackage{comment}
\usepackage{array}
\usepackage{float}
\usepackage{physics}

\newcommand{\T}{\mathbb{T}}
\newcommand{\Z}{\mathbb{Z}}
\usepackage{graphicx}
\usepackage{subcaption}
\newtheorem{theorem}{Theorem}

\newtheorem{corollary}[theorem]{Corollary}

\newtheorem{definition}[theorem]{Definition}

\newtheorem{lemma}[theorem]{Lemma}

\newtheorem{proposition}[theorem]{Proposition}
\newtheorem{remark}[theorem]{Remark}

\newcommand\cH{{\mathcal H}}

\newcommand\cL{{\mathcal L}}

\newcommand\cP{{\mathcal P}}

\newcommand\bT{{\mathbb T}}

\title[Optimal response for SDEs in $\T^d$ with perturbations on the drift term]{Optimal response for stochastic differential equations in $\T^d$ with perturbations on the drift term}

\author{Gianmarco Del Sarto}
\address{Technische Universit\"{a}t Darmstadt\\
Fachbereich Mathematik\\
	Schlossgartenstr.\ 7\\
	64289 Darmstadt\\
	Germany}
\email{delsarto@mathematik.tu-darmstadt.de}

\author{Franco Flandoli}
\address{Scuola Normale Superiore\\
	Classe di Scienze\\
	P.za dei Cavalieri 7\\
	56126 Pisa\\
	Italy}
\email{franco.flandoli@sns.it}

\author{Stefano Galatolo}
\address{University of Pisa, Centro Interdipartimentale per lo Studio dei Sistemi Complessi\\
Department of Mathematics and Statistics\\
	Via Buonattoti 1\\
	Pisa 56217\\
	Italy}
\email{stefano.galatolo@unipi.it}

\author{Sakshi Jain}
\address{School of Mathematics and Physics\\
    The University of Queensland\\
	St Lucia, QLD 4072\\
	Australia}
\email{sakshi.jain@uq.edu.au}

\author{Angxiu Ni}
\address{Department of Mathematics\\
University of California, Irvine\\
USA
\and
Yau Mathematical Sciences Center\\
Tsinghua University\\
China}
\email{angxiun@uci.edu}

\begin{document}
\begin{abstract}
    We study stochastic differential equations on the $d$-dimensional flat torus $\T^d$ with drift and perturbation coefficients in $L^{\infty}(\T^d;\mathbb{R}^d)$ and additive non-degenerate noise. For the associated transfer operators, we analyse the dependence of the stationary measure and of the expectation of a given observable on small perturbations of the drift. In this framework, we prove a linear response formula for the invariant density and for the expectation of a given observable.
    
    We then address an optimal response problem, namely the determination of admissible perturbations that maximise the first-order variation of a prescribed observable. We establish existence of optimal perturbations and, in a Hilbert space framework, prove uniqueness and provide an explicit characterisation of the optimiser. This yields a practical Fourier-based numerical method, which we implement in several numerical examples, including both low and high-dimensional settings.
\end{abstract}

\maketitle
\tableofcontents

\section{Introduction}

Understanding how the statistical properties of a dynamical system change under perturbations is a classical problem in dynamics, probability, and applications.  These statistical properties are naturally encoded by stationary measures and by the expectations of observables with respect to them; this point of view is standard in the study of diffusion processes and Markov semigroups, see \cite{Pavliotis2014,HM10,CKB22,AFG22}. A central question is whether these quantities vary differentiably under small changes of the system. When this happens, the system is said to exhibit \emph{linear response}.
For general background on linear response we refer for instance to \cite{Ruelle97,Ruelle09,Baladi14,Lucarini2008,CKB22,AFG22}, while for stochastic systems see also \cite{HM10,KLP,CKB22,GG_2019}.
 Beyond its intrinsic theoretical interest, linear response provides quantitative information on how a system reacts to external interventions and is therefore a natural starting point for problems of control, optimisation, or climate change \cite{AFG22,Schtte2001,Lucarini2016,Schtte_2023,RevModPhys.92.035002}.

A natural question arising from the existence of the linear response for a certain family of systems/perturbations is the  inverse problem of understanding how to find perturbation to apply to the system, resulting in a certain desired response, and how to do it  in some optimal way.
This is in some sense a control problem for the statistical properties of a dynamical system.

In this paper we study these kinds of questions for stochastic differential equations (SDEs) on the flat torus $\T^d$ with additive non-degenerate noise, under perturbations of the drift term. More precisely, we consider a family of SDEs obtained by adding a small perturbation $\delta \eta$ to the drift, with $\delta> 0$, and we investigate the effect of this perturbation on the invariant density and on the expectation of a given observable. 
A first goal of the paper is to establish a linear response formula in this setting and study its continuity properties up to changes in the perturbation direction. The main goal is then to address the associated \emph{optimal response problem}: among all admissible perturbation directions, determine those producing the largest first-order variation of a prescribed observable.

Our approach is based on the transfer operator associated with the time-one map of the SDE \cite{Schtte2001,Schtte_2023,AFG22}. The regularising effect of the noise allows us to work with compact integral operators on $L^2(\T^d)$, and this makes it possible to combine partial differential equation (PDE) estimates for the Fokker-Planck equation (FPE) with spectral arguments for the transfer operator \cite{Pavliotis2014,Bittracher2017,AFG22}. In this way we first obtain a linear response formula for the invariant density and, as a consequence, for the expectation of observables.
The linear response statement we prove is uniform in the perturbation, in the sense that we construct a certain bounded response operator, associating the infinitesimal perturbation with its response. 
This allows to investigate the optimization problem we want to address.
We then formulate the optimal response problem in a Hilbert-space framework, prove existence of optimal perturbations, and obtain uniqueness and an explicit characterisation of the optimiser under natural convexity assumptions. This explicit characterisation leads to a practical Fourier-based approximation method, which is implemented in several numerical examples, including higher-dimensional ones. A key analytic ingredient is a self-contained treatment of the FPE with Dirac delta initial condition. This gives the transition density for positive times and the estimates linking the FPE to the transfer-operator approach.

The control of the statistical properties of dynamical systems can be studied from different points of view. In \cite{Galatolo_2017} and \cite{Kloeckner2018LinearRequestProblem} the point of view was to find an infinitesimal perturbation that realizes a certain response uniformly for a large class of observables. This problem  was also refered to as the "Linear Request" problem. 

Toward the purposes of the present work, particularly relevant precursors are \cite{Antown2018} where the problem of finding an optimal response for a finite Markov Chain is considered and \cite{AFG22}, where the linear response and the optimal response were studied for systems whose transfer operators are given by kernels. 
We also mention the related  optimal response results obtained for deterministic expanding maps in \cite{FG23} and for hyperbolic systems in \cite{FroylandPhalempin2025OptimalLinearResponse} and \cite{GN25}.
In  \cite{Galatolo2022} the problem of optimizing an infinitesimal mean field coupling is considered.

The framework of \cite{AFG22}  includes important random systems and provides a natural operator-theoretic viewpoint on the optimisation problem, see also \cite{Carigi2025} for a related matrix-based computational approach. 
The present paper is closely related to that perspective, but differs from it in two main respects. First, we work directly with perturbations of the drift of an SDE, so that the perturbation has a direct dynamical meaning. Second, we are especially interested in multidimensional and high-dimensional computations, which motivates the Fourier-based Hilbert-space approach developed here.

It is also useful to compare the present paper with our previous work \cite{DGJ}. There, we studied optimal response for a time-discretisation of an SDE on $\mathbb{R}^d$ under local perturbations of the associated transfer kernel. While that setting is natural from the viewpoint of transfer operators, such perturbations are not, in general, induced by perturbations of the underlying drift. Here, instead, we perturb the drift directly on $\T^d$, so that the admissible directions have a direct dynamical, and hence more physical, interpretation. This naturally leads to a Hilbert-space framework. On the numerical side, the present paper also benefits from more recent efficient algorithms developed in \cite{Ni2025, Ni25}
for the linear response of diffusive systems, which allow a more effective treatment of multidimensional examples.

This paper is organised as follows. Section \ref{sec:main_results} states the main results. Section \ref{sec: SDEs} introduces the stochastic differential equation on $\T^d$ and the main properties of the associated probability densities. Section \ref{sec: transfer operator and spectral gap} studies the transfer operator and its spectral properties, while Section \ref{sec: perturbations and linear response} proves the linear response formula. Section \ref{sec: optimal response} is devoted to the optimal response problem, and Section \ref{sec: numerical examples} contains the numerical examples. Finally, Appendix \ref{appendix} develops the semigroup estimates for the FPE used in the paper, including a self-contained treatment of rough initial data (such as Dirac deltas), the corresponding density and positivity estimates for the transition kernel, and the first-order perturbation results needed.

\subsection*{Notation}
Throughout the paper, $\T^d:=\mathbb{R}^d/\mathbb{Z}^d$ denotes the $d$-dimensional flat torus, endowed with the normalised Lebesgue measure. For $1\leq p\leq \infty$, $L^p(\T^d)$ denotes the usual Lebesgue space of real-valued functions on $\T^d$, and $L^p(\T^d;\mathbb{R}^d)$ its vector-valued analogue. For $m\in\mathbb{N}$ and $1\leq p\leq \infty$, we write $W^{m,p}(\T^d)$ and $W^{m,p}(\T^d;\mathbb{R}^d)$ for the corresponding Sobolev spaces. In particular, $H^m(\T^d):=W^{m,2}(\T^d)$ and $H^m(\T^d;\mathbb{R}^d):=W^{m,2}(\T^d;\mathbb{R}^d)$, their norms are denoted by $\|\cdot\|_{p}$ and $\|\cdot\|_{{m,p}}$ respectively. We use the notation $X\hookrightarrow Y$ to denote a continuous embedding of $X$ into $Y$, and we denote by $V_{L^2}$ the zero average subspace of $L^2(\T^d)$, i.e. $V_{L^2}:= \{ f \in L^2(\T^d) \; : \; \int_{\T^d} f \, dx  = 0\}$. Given Banach spaces $X$ and $Y$, we denote by $\mathcal{L}(X,Y)$ the space of bounded linear operators from $X$ to $Y$, and we write $\mathcal{L}(X):=\mathcal{L}(X,X)$. Lastly, given a Banach space $X$, we denote by $C(X):= \{ f \colon X \to \mathbb{R} \, : f \text{ is continuous } \}$ the space of real-valued functions on $X.$ We refer to \cite{Tri78,Brezis2010} for more details on the function spaces.

We write $(\Omega,\mathcal{F},(\mathcal{F}_t)_{t\geq 0},\mathbb{P})$ for a filtered probability space, where $\Omega$ is the sample space, $\mathcal{F}$ is the $\sigma$-algebra, $(\mathcal{F}_t)_{t\geq 0}$ is a filtration, and $\mathbb{P}$ is the probability measure. The symbol $\mathbb{E}$ denotes expectation with respect to $\mathbb{P}$. We denote by $\mathcal{M}(\T^d)$ the space of finite signed Borel measures on $\T^d$, endowed with the total variation norm. We use $\mathcal{B}_b(\T^d)$ for the space of bounded Borel-measurable real-valued functions on $\T^d$. See \cite{Baldi2017,Stroock2006} for more information on stochastic processes.

\section{Main results}
\label{sec:main_results}
We consider the stochastic differential equation on the flat torus $\T^d$
\begin{equation*}
 dX_t^{\delta,x}=(b(X_t^{\delta,x})+\delta\eta(X_t^{\delta,x}))\,dt+dW_t,
 \qquad X_0^{\delta,x}=x,
\end{equation*}
where $b\in L^{\infty}(\T^d;\mathbb{R}^d)$ is a given drift, $\eta\in L^{\infty}(\T^d;\mathbb{R}^d)$ is a perturbation direction, and the noise is additive and non-degenerate. For each $\delta \geq 0$, let $p^\delta(1,x,y)$ be the transition density at time $1$, and define the associated transfer operator
\begin{equation*}
\cL_\delta\colon L^2(\T^d)\to L^2(\T^d),
\qquad (\cL_\delta f)(y):=\int_{\T^d} p^\delta(1,x,y)f(x)\,dx.
\end{equation*}
Thus $\cL_\delta f$ is the density at time $1$ of the perturbed dynamics started from an initial density $f$.

Let $f_0$ be the unique normalised invariant density of $\cL_0$, and let 
$$
V_{L^2}:=\{g\in L^2(\T^d) : \int_{\T^d} g\,dx=0 \}.
$$
A key point of our work is that the first-order effect of the perturbation on the one-step evolution of $f_0$ is described by a bounded linear operator on the space of perturbation fields. More precisely, for each $\eta\in L^{\infty}(\T^d;\mathbb{R}^d)$ we define
$$
D_{f_0,1}^{L^{\infty}}(\eta) := \lim_{\delta\to 0}
\frac{(\cL_\delta-\cL_0)f_0}{\delta} \qquad\text{in }L^2(\T^d),
$$
and prove that this limit exists, belongs to $V_{L^2}$, and depends linearly and continuously on $\eta$, see Proposition~\ref{cont}. Our first result shows that the linear response operator is continuous on $L^\infty(\T^d; \mathbb{R}^d).$
\begin{theorem}[Bounded linear response operator]
\label{thm: observables linear response}
For every observable $\phi\in L^2(\T^d)$ and every perturbation direction
$\eta\in L^{\infty}(\T^d;\mathbb{R}^d)$, the response
\begin{equation*}
R(\phi,\eta)
:=
\lim_{\delta\to 0}
\frac{\int_{\T^d}\phi f_\delta\,dx-\int_{\T^d}\phi f_0\,dx}{\delta}
\end{equation*}
is well defined and
\begin{equation*}
R(\phi,\eta)
=
\int_{\T^d}
\phi\,(Id-\cL_0)^{-1}D_{f_0,1}^{L^{\infty}}(\eta)\,dx.
\end{equation*}
Moreover, for every fixed $\phi\in L^2(\T^d)$, the map
\begin{equation*}
\eta\mapsto R(\phi,\eta)
\end{equation*}
is a continuous linear functional on $L^{\infty}(\T^d;\mathbb{R}^d)$.
\end{theorem}
\noindent
The linear response formula itself is classical in noisy settings. But the point that is important for the optimisation problem considered in our work is that the response is obtained through the bounded operator
\begin{equation*}
D_{f_0,1}^{L^{\infty}}\colon L^{\infty}(\T^d;\mathbb{R}^d)\to V_{L^2},
\end{equation*}
which allows us to optimise over general classes of perturbation fields.

To formulate the optimisation problem, let $\cH$ be a separable Hilbert space such that
$$
\cH \hookrightarrow L^{\infty}(\T^d;\mathbb{R}^d).
$$
By restriction of $D_{f_0,1}^{L^{\infty}}$, we obtain a bounded linear map $
D_{f_0,1}^{\cH}\colon \cH\to V_{L^2}.$ For $\phi\in L^2(\T^d)$ and $\eta\in\cH$, we then define
\begin{equation}
R_{\cH}(\phi,\eta):=
\int_{\T^d}
\phi\,(Id-\cL_0)^{-1}D_{f_0,1}^{\cH}(\eta)\,dx.
\label{eq: definition of R_cH}
\end{equation}
Since $\cH\hookrightarrow L^{\infty}(\T^d;\mathbb{R}^d)$, this is exactly the response functional restricted to perturbations in $\cH$.

We can now state the optimal response problem: given an observable $\phi$ and an admissible class $P\subset \cH$, find the perturbation fields in $P$ that maximise $R_{\cH}(\phi,\cdot)$. Our next result gives existence and uniqueness on strictly convex admissible classes.

\begin{theorem}[Existence and uniqueness of an optimal perturbation]
\label{thm:gen-cnvx}
Let $\cH$ be a separable Hilbert space such that
\begin{equation*}
\cH \hookrightarrow L^\infty(\T^d;\mathbb{R}^d).
\end{equation*}
Let $P\subset \cH$ be a closed, bounded, and strictly convex set whose relative interior contains the origin, and let $\phi\in L^2(\T^d)$. If the functional
\begin{equation*}
R_{\cH}(\phi,\cdot)\colon \cH\to\mathbb{R}
\end{equation*}
is not identically zero on $P$, then there exists a unique
$\eta_{\mathrm{opt}}\in P$ such that
\begin{equation*}
R_{\cH}(\phi,\eta_{\mathrm{opt}})
=
\max_{\eta\in P}R_{\cH}(\phi,\eta).
\end{equation*}
\end{theorem}

An especially important case is when the admissible set is the closed unit ball of $\cH$,
\begin{equation}
\label{e:define optr}
P:=\{\eta\in\cH:\|\eta\|_{\cH}\le 1\}.
\end{equation}
In this case the optimiser can be written explicitly.

\begin{theorem}[Optimal perturbation on the unit ball of $\cH$]
\label{gen:algo}
Let $\cH$ be a separable Hilbert space such that
\begin{equation*}
\cH \hookrightarrow L^\infty(\T^d;\mathbb{R}^d),
\end{equation*}
and let $\{b_i\}_{i\in\mathbb{N}}$ be an orthonormal basis of $\cH$. Assume that $
R_{\cH}(\phi,\cdot)\colon \cH\to\mathbb{R}$ is not identically zero. Then the optimisation problem
\begin{equation*}
R_{\cH}(\phi,\eta_{\mathrm{opt}})
=
\max_{\eta\in P}R_{\cH}(\phi,\eta),
\qquad P=\{\eta\in\cH:\|\eta\|_{\cH}\le 1\},
\end{equation*}
has the unique solution
\begin{equation*}
\eta_{\mathrm{opt}}=\frac{v}{\|v\|_{\cH}},
\end{equation*}
where $v\in\cH$ is the Riesz representative of $R_{\cH}(\phi,\cdot)$, characterised by
\begin{equation*}
\langle v,b_i\rangle_{\cH}=R_{\cH}(\phi,b_i),
\qquad i\in\mathbb{N}.
\end{equation*}
\end{theorem}
Theorem \ref{gen:algo} provides a practical route to computation: once the response is evaluated on a convenient orthonormal basis of $\cH$, the optimal perturbation can be approximated from its coefficients. In the Sobolev setting one may take
\begin{equation*}
\cH=H^p(\T^d;\mathbb{R}^d),
\qquad
p>\frac d2,
\end{equation*}
so that the Sobolev embedding theorem yields
\begin{equation*}
H^p(\T^d;\mathbb{R}^d)\hookrightarrow L^{\infty}(\T^d;\mathbb{R}^d).
\end{equation*}
More generally, the same approach applies to Hilbert subspaces continuously embedded in $L^{\infty}(\T^d;\mathbb{R}^d)$, which is useful in some of the high-dimensional examples.

The final part of the paper is devoted to numerical experiments. There we implement the above strategy in several examples, including a two-dimensional Kuramoto system, a high-dimensional solenoid-type model, and a three-dimensional Lorenz system. These computations show that the optimal response problem can be solved effectively in practice and that the Hilbert-space characterisation above leads to a concrete and flexible numerical method.

\section{Stochastic differential equations on \texorpdfstring{$\T^d$}{T d}}
\label{sec: SDEs}

\subsection{Preliminaries}
We work on the $d$-dimensional flat torus $\T^d:=\mathbb{R}^d/\Z^d$, endowed with the normalised Lebesgue measure. Let $b, \eta \colon \T^d \to \mathbb{R}^d$ and assume $b, \eta \in L^{\infty}(\mathbb{T}^d; \mathbb{R}^d).$ For any $\delta \in [0,1)$, we consider the SDE on $\T^d$ given by
\begin{equation}
\left \lbrace 
    \begin{aligned}
        dX_t^{\delta,x} &= (b(X_t^{\delta,x}) + \delta \eta (X_t^{\delta,x}) ) dt + dW_t, \qquad t \in (0,T),\\
        X_0^{\delta,x} &= x.
    \end{aligned}
    \right.
    \label{eq: SDE in Td}
\end{equation}
Here $x \in \mathbb{T}^d$ denotes the initial condition and $(W_t)_t$ is a Brownian motion defined on the filtered probability space $(\Omega, \mathcal{F}, (\mathcal{F}_t)_t, \mathbb{P})$. 
\begin{remark}\label{rem:stationary_measure}
We recall that the assumptions $b,\eta \in L^\infty(\T^d; \mathbb{R}^d)$ and the non-degenerate diffusion coefficient are sufficient for the strong existence and uniqueness for the solution of \eqref{eq: SDE in Td}. This can be deduced, for instance, by the Zvonkin-Veretennikov Theorem, see \cite[Theorem 1]{Zvonkin}. Regarding the existence of a probability density $p^{\delta}(t,x, \cdot )$ associated to $X_t^{\delta,x}$, for $t>0$, it is again a consequence of the boundedness of the drift and the non degenerate noise, see \cite[Section 9.2]{Stroock2006}. Moreover, since $\T^d$ is compact, the Krylov-Bogolyubov Theorem guarantees the existence of at least one invariant probability measure for \eqref{eq: SDE in Td}, see \cite[Section 7.1]{DaPrato2006}. On the other hand, the non-degenerate additive noise implies that the associated Markov semigroup is strong Feller and irreducible; thus, this invariant probability measure is unique, see \cite[Section 7.2]{DaPrato2006}.
\end{remark}
We recall, for the sake of clarity, that the solution of the SDE \eqref{eq: SDE in Td} can be constructed by lifting the problem on $\mathbb{R}^d$ as follows. Let $\pi := \mathbb{R}^d \to \mathbb{R}^d / \mathbb{Z}^d$ be the canonical $1$-Lipschitz projection. Define the $\mathbb{Z}^d$-periodic lifts $\tilde b, \tilde \eta \colon \mathbb{R}^d \to  \mathbb{R}^d$ by
$$
\tilde b(x) := b (\pi (x)), \qquad \tilde \eta (x) := \eta (\pi (x)), \qquad x \in \mathbb{R}^d,
$$
and note that $\tilde b, \tilde  \eta \in L^{\infty}(\mathbb{R}^d; \mathbb{R}^d)$ since $b, \eta \in L^{\infty} (\mathbb{T}^d; \mathbb{R}^d)$ by assumption and the projection $\pi $ is $1$-Lipschitz. For any representative $\tilde x \in \mathbb{R}^d$ of $x = \pi (\tilde x) \in \mathbb{T}^d$, consider the lifted SDE on $\mathbb{R}^d$
\begin{equation}
\left \lbrace 
    \begin{aligned}
        d \tilde X_t^{\delta, \tilde x} &= (\tilde b(\tilde X_t^{\delta, \tilde x}) + \delta \tilde  \eta (\tilde X_t^{\delta, \tilde x}) ) dt + dW_t, \qquad t \in (0,T),\\
        \tilde X_0^{\delta, \tilde x} &= \tilde x.
    \end{aligned}
    \right.
    \label{eq: SDE on Rd}
\end{equation}
Since the drift of the previous equation, i.e. $\tilde b + \delta \tilde \eta $, is in $L^\infty(\mathbb{R}^d; \mathbb{R}^d)$, then the strong existence and uniqueness for the SDE \eqref{eq: SDE on Rd} follows from the Zvonkin-Veretennikov Theorem, as recalled in Remark \ref{rem:stationary_measure}.

Moreover, since the noise of the SDE is non-degenerate, the law of $\tilde X_t^{\delta, \tilde x}$ is absolutely continuous with respect to Lebesgue measure on $\mathbb{R}^d$ for $t>0$. We denote by $\tilde p^\delta(t,x,\cdot)$ its density, i.e.
\[
\mathbb{P}(\tilde X^{\delta,\tilde x}_t\in A)=\int_A \tilde p^\delta(t,x,y)\,dy,\qquad A\subset\mathbb{R}^d\ \text{Borel}.
\]
We refer to \cite[Section 9.2]{Stroock2006} for more details on densities associated to SDEs. Hence, the solution of \eqref{eq: SDE in Td} is defined as
$$
X_t^{\delta,x}:=\pi(\tilde X_t^{\delta,\tilde x}).
$$
This definition does not depend on the choice of the representative $\tilde  x$ since $\tilde b, \tilde \eta$ are $\mathbb{Z}^d$-periodic. Moreover, for every $t>0$ also the law of $X_t^{\delta,x}$ is absolutely continuous with respect to the
(normalised) Lebesgue measure on $\T^d$; we denote its density by $p^\delta(t,x,\cdot)$. It can be expressed in terms of the density $\tilde p^\delta$ of the lifted process by the periodisation formula
\begin{equation}
    p^\delta(t,x,y)=\sum_{k\in\Z^d}\tilde p^\delta\big(t,\tilde x,\tilde y+k\big),
    \label{eq: density on Td}
\end{equation}
where $\tilde x,\tilde y\in\mathbb{R}^d$ are any representatives of $x,y\in\T^d$. Informally, \eqref{eq: density on Td} means that reaching $y$ on the torus corresponds to reaching any translate $\tilde y + k \in \mathbb{R}^d$.\footnote{To justify \eqref{eq: density on Td}, identify $\T^d$ with the representative domain $[0,1)^d$ and let $A\subset [0,1)^d$ be a Borel set. Then
$$
\mathbb{P}(X_t^{\delta,x}\in A)=\mathbb{P}(\pi(\tilde X_t^{\delta,\tilde x})\in A)
=\sum_{k\in\Z^d}\mathbb{P}(\tilde X_t^{\delta,\tilde x}\in A+k)
=\sum_{k\in\Z^d}\int_{A+k}\tilde p^\delta(t,\tilde x,z)\,dz.
$$
Changing variables $z=y+k$ in each term gives
$$
\mathbb{P}(X_t^{\delta,x}\in A)=\int_A\left(\sum_{k\in\Z^d}\tilde p^\delta(t,\tilde x,y+k)\right)\,dy,
$$
which proves that the expression in \eqref{eq: density on Td} is the density of $X_t^{\delta,x}$. The series is well-defined at least in $L^1([0,1)^d)$, and hence finite a.e., since it is non-negative and
$$
\int_{[0,1)^d} p^\delta(t,x,y)\,dy
=\sum_{k\in\Z^d}\int_{[0,1)^d}\tilde p^\delta(t,\tilde x,y+k)\,dy
=\sum_{k\in\Z^d}\int_{[0,1)^d+k}\tilde p^\delta(t,\tilde x,z)\,dz
=\int_{\mathbb{R}^d}\tilde p^\delta(t,\tilde x,z)\,dz
=1.
$$}

Furthermore, the density $p^\delta $ satisfies the Fokker-Planck equation (FPE)
\begin{equation}
\left \lbrace
    \begin{aligned}
        \partial_t  p^\delta &= \frac{1}{2}\Delta_y  p^\delta - \mathrm{div}_y(  p^\delta (b + \delta \eta )), \qquad  &\text{ in } (0,T) \times \T^d, \\
         p^\delta|_{t = 0} &=  \delta_x, \qquad & \text{ in } \T^d,
    \end{aligned}
    \right.
    \label{eq: FPE in Td}
\end{equation}
More precisely, a family of probability measures
$(\mu_t^\delta)_{t\in[0,T]}$ is a solution of \eqref{eq: FPE in Td} if
for every \(\varphi\in C_c^\infty([0,T)\times\T^d)\),
$$
\int_0^T \int_{\T^d}
\left( \partial_t\varphi(t,y) +\frac{1}{2}\Delta_y\varphi(t,y) +(b(y)+\delta\eta(y))\cdot\nabla_y\varphi(t,y)
\right) \,d\mu_t^\delta(y)\,dt +\varphi(0,x)=0.
$$
For $t>0$, $\mu_t^\delta$ is absolutely continuous with respect to Lebesgue measure and we denote its density by $p^\delta(t,x,\cdot)$; in this case, the density $p^\delta(t) $, for $t >0$, coincides with the mild/variational solution introduced in Appendix \ref{appendix}. For more details on the measure-valued solutions for FPEs and associated densities we refer to \cite[Chapter 6]{Bogachev2022}.

The next proposition collects the basic properties of the transition kernel associated with the perturbed SDE that will be used repeatedly in the following sections. In particular, it provides existence of the density, conservation of mass, quantitative smoothing bounds, and a uniform positivity estimate, which are the ingredients needed later for the transfer-operator approach.

\begin{proposition}[Kernel estimates]\label{prop:kernel}
Assume $b,\eta\in L^{\infty}(\T^d;\mathbb{R}^d)$. Let $x\in\T^d$. There exists a unique solution
$(\mu_t^{\delta})_{t\in[0,T]}$ of \eqref{eq: FPE in Td}. For every $t>0$, $\mu_t^{\delta}$ admits a density $p^\delta(t,x,\cdot)$,
and this density satisfies
\begin{equation}
\| p^\delta (t,x, \cdot ) \|_1 = 1, \qquad t \in (0, T].
\label{eq: mass conservation}
\end{equation}
Furthermore, there exists a constant  $C = C( \| b \|_{{\infty}},  \| \eta  \|_{{\infty}}, d, T)>0$ independent of $\delta$ such that 
\begin{equation}
\| p^\delta (t,x, \cdot ) \|_2 \leq C t^{-d/4}, \qquad \delta \in [0,1).
\label{eq: kernel estimate H1}
\end{equation}
Moreover
\begin{equation}
\| p^\delta (t,x,\cdot )\|_\infty \leq  C t^{-d/2}, \quad 0 < t < T, \quad \delta \in [0,1),
\label{eq: kernel estimarte Linfty}
\end{equation}
and there exists $c = c( t,\| b\|_\infty, \| \eta \|_\infty ,d)>0$ such that
\begin{equation}
\inf_{(x,y ) \in \T^d \times \T^d} p^\delta (t, x,y) \geq c, \qquad t \in (0,T], \quad \delta \in [0,1). 
    \label{eq: bound below positive kernel}
\end{equation}
\end{proposition}
\begin{proof}
    In the Appendix \ref{appendix}, see Corollary \ref{cor: semigroup estimates for dirac delta initial condition} for the proof of \eqref{eq: kernel estimate H1}. See Proposition \ref{prop:Linfty_bound_density} for the proof of \eqref{eq: kernel estimarte Linfty}. See Proposition \ref{prop: kernel positivity} for the proof of \eqref{eq: bound below positive kernel}.
\end{proof}

\subsection{Convergence of densities}\label{sec:L^2-initial condition}

In the rest of this section we work with an $L^2$ density for the initial condition of the FPE. Indeed this is the natural space for the operator-theoretic framework below: the transfer operator acts on $L^2(\T^d)$, and this choice is also compatible with the compactness and perturbative estimates proved later.

 Let $p_0\in L^2(\T^d)$ and $\delta\in[0,1)$. We denote by $p^{\delta,p_0}$ the (unique) mild/variational solution of
\begin{equation}
\left \lbrace
\begin{aligned}
\partial_t p^{\delta,p_0} &= \frac{1}{2}\Delta p^{\delta,p_0} - \mathrm{div}(p^{\delta,p_0}( b + \delta \eta)), \qquad &\text{ in } (0,T) \times \T^d,\\
p^{\delta,p_0}|_{ t = 0} &= p_0, \qquad &\text{ in }   \T^d.
\end{aligned}
\right.
    \label{eq: FPE with regular initial condition}
\end{equation}
We write $p^{0,p_0}$ for the corresponding solution when $\delta=0$. For $\delta\in(0,1)$ we introduce the difference quotient
\begin{equation}\label{eq:def_rdelta_p0}
r^{\delta,p_0}(t,\cdot):=\frac{p^{\delta,p_0}(t,\cdot)-p^{0,p_0}(t,\cdot)}{\delta},
\qquad t\in(0,T].
\end{equation}
Formally differentiating the Fokker-Planck equation \eqref{eq: FPE with regular initial condition} at $\delta=0$ suggests that the candidate limit $\dot r^{p_0}$
should solve a linear forced parabolic equation, where the forcing comes from differentiating the drift term
$$
-\mathrm{div}((b+\delta\eta)p^{\delta,p_0})|_{\delta=0}'  = -\mathrm{div}(b\dot r^{p_0}) -\mathrm{div}(\eta p^{0,p_0}).
$$
This motivates the definition of $\dot r^{p_0, \eta}:=\dot r^{p_0}$ via the following parabolic problem
\begin{equation}
\left\lbrace 
\begin{aligned}
\partial_t \dot { {r}}^{p_0}  &= \frac{1}{2}\Delta \dot{ r }^{p_0}  - \mathrm{div} (b \dot{ r}^{p_0} ) - \mathrm{div} ( \eta  p ^{0,p_0} ), \qquad &\text{ in } (0,T) \times \T^d, \\
 \dot r^{p_0} |_{t = 0} &= 0, \qquad & \text{ in } \T^d.
     \label{eq: eq for dot r regular initial condition}
\end{aligned}
\right.
\end{equation}

\begin{remark}\label{rem:density-relation}
The relation between the solution $p^{\delta,p_0}$ of the FPE \eqref{eq: FPE with regular initial condition} with regular initial condition $p_0$, and the solution $p^\delta$ of \eqref{eq: FPE in Td} with initial condition $\delta_x$ is
\begin{equation}
p^{\delta, p_0} (t,y) = \int_{\T^d}p^\delta (t,x,y) p_0(x)dx .
\label{eq: relation betwee p p_0 and kernel}
\end{equation}
Indeed, let $(X^{\delta,p_0}_t)_t$ be the unique solution of the SDE \eqref{eq: SDE in Td} starting from a random initial condition $X_0$ with density $p_0$. Then, for any Borel set $A \subset \T^d$, by the tower property of the conditional expectation,
$$
\mathbb{P}(X_t^{\delta,p_0} \in A) = \mathbb{E}\left[\mathds{1}_A (X_t^{\delta, p_0}) \right] = \mathbb{E}\left[\mathbb{E}\left[\mathds{1}_A (X_t^{\delta, p_0}) \mid X_0 \right] \right]= \mathbb{E}\left[\mathbb{P}(X_t^{\delta, p_0} \in A \mid X_0 )\right].
$$
By the Markov property, $\mathbb{P}$-a.s. it holds
$$
\mathbb{P}(X_t^{\delta, p_0} \in A \mid X_0 )  = g(X_0), \qquad \text{where } g(x):= \mathbb{P}(X_t^{\delta, x} \in A).
$$
Since $X_0$ has density $p_0$, taking the expectation yields
$$
\mathbb{P}(X_t^{\delta, p_0} \in A) = \mathbb{E}\left[ g(X_0)\right]= \int_{\T^d} \mathbb{P}(X_t^{\delta, x} \in A) \, p_0 (x) \, dx.
$$
Using that $\mathbb{P}(X_t^{\delta, x} \in A) = \int_A p^{\delta}(t,x,y) \, dy$ and applying Fubini's theorem, we obtain
$$
\mathbb{P}(X_t^{\delta ,p_0} \in A) = \int_{\T^d} \left( \int_A p^{\delta}(t,x,y)\, dy \right) p_0(x) \, dx = \int_A \left( \int_{\T^d}p^{\delta}(t,x,y)\, p_0(x) \, dx \right) dy,
$$
which implies \eqref{eq: relation betwee p p_0 and kernel} by the arbitrariness of $A$.
\end{remark}
    The previous discussion identifies the natural candidate for the derivative of the density with respect to the perturbation parameter. The next proposition shows that this formal computation is indeed correct: the difference quotient $r^{\delta,p_0}$ converges in $L^2$ to $\dot r^{p_0}$, and the convergence comes with an explicit quantitative error bound. In this sense, $\dot r^{p_0}$ is the first-order response of the solution of the Fokker-Planck equation.
\begin{proposition}[First-order expansion for $L^2$-initial data]
\label{prop: convergence density pdelta regular p0}
    Assume $b, \eta \in L^{\infty} (\T^d; \mathbb{R}^d)$, $p_0 \in L^2(\T^d)$ and $T>0$. There exists $
    C = C(  \|  b \|_{\infty }, \| \eta \|_{\infty } ,T,d) >0 
    $
    such that for all $\tau \in (0,T)$ and $\delta \in (0,1)$,
    $$
     \norm{ r^{\delta, p_0}(\tau) - \dot r^{p_0}(\tau) }_2 = \norm{ \frac{p^{\delta,p_0} (\tau) - p^{0,p_0} (\tau)}{\delta} - \dot r^{p_0}(\tau) }_2 \leq C \delta  \tau \| p_0 \|_2.
    $$
    In particular
    $$
    \sup_{0 \leq \tau  \leq T} \|r^{\delta, p_0}(\tau) - \dot r^{p_0}(\tau) \|_2 \leq  C  \delta   T \| p_0 \|_{2}
    $$
\end{proposition}
\begin{proof}
    See Section \ref{proof: main result convergence p delta regular ic} in the Appendix.
\end{proof}
Once the first-order expansion is established, the next step is to understand how the derivative depends on the perturbation field $\eta$. The following proposition shows that, for each fixed time $\tau$, the map sending $\eta$ to the corresponding first-order variation $\dot r^{p_0,\eta}(\tau)$ is linear and continuous from the perturbation space into $L^2(\T^d)$. Moreover, its values have zero average, a property that will be important later when applying the resolvent on $V_{L^2}$.
\begin{proposition}[Linearity and continuity of the derivative map]
    \label{cont}
    Assume $b, \eta \in L^{\infty} (\T^d; \mathbb{R}^d)$, $p_0 \in L^2(\T^d)$ and $T>0$. Fix $\tau \in (0,T)$. Define  $$
    {D}_{p_0,\tau }^{L^\infty} :L^{\infty}(\mathbb{T}^d; \mathbb{R}^d)\to L^2(\mathbb{T}^d), \qquad \eta\mapsto \dot{r}^{p_0, \eta } (\tau):=\dot{r}^{p_0} (\tau),
    $$
    where $\dot{r}^{p_0}$ solves \eqref{eq: eq for dot r regular initial condition}. Then ${D}_{p_0,\tau }^{L^\infty}$ is linear and continuous.  Moreover, for every $\eta \in L^\infty(\T^d;\mathbb{R}^d)$,
    $$
    \int_{\T^d} D_{p_0,\tau}^{L^\infty}(\eta)\,dx=0,
    $$
\end{proposition}
\begin{proof}
    See Section \ref{proof: derivative linear and bounded} in the Appendix.
\end{proof}
Consider now a Banach space $\cH$ such that $\cH \hookrightarrow L^\infty (\T^d; \mathbb{R}^d),$ and denote by $\iota_\cH$ the continuous embedding. It is straightforward to obtain the following result.
\begin{corollary}
    Let $\cH$ be a Banach space such that $\cH \hookrightarrow L^\infty (\T^d; \mathbb{R}^d)$. Under the assumptions of Proposition \ref{cont}, for every $\tau \in (0,T]$ the map
    $$
    D_{p_0, \tau}^\cH :  \cH \to L^2(\T^d), \qquad D_{p_0, \tau }^\cH (\eta) := (D^{L^\infty }_{p_0, \tau} \circ \iota_{\cH} ) (\eta)
    $$
    is linear and continuous. Moreover, $D_{p_0, \tau}^\cH (\eta)$ has zero average on $\T^d$ for any $\eta \in \cH$.
    \label{cor: extension continuity derivative in cV}
\end{corollary}

\section{Transfer operator and spectral gap}
\label{sec: transfer operator and spectral gap}
\begin{remark}\label{rem:no_t}
    As stated above, the SDE that we consider in this article has a unique stationary measure $\mu$.
   In the following we fix the time $t=1$ and consider a time discretization of this system.
   In the notations thus drop the subscript \(t\). 
   Hence, in the following, the kernel $p^\delta(1,x,y)$ will be denoted as $p^\delta(x,y)$ (and as \(p(x,y)\) for \(\delta=0\)).  
\end{remark}

In this section, we define the transfer operator associated with the evolution of the SDE and state/prove some basic properties of these operators. We will extensively use the density \(p(x,y)\) and its properties proved in the previous section.
\begin{definition}
\label{def:koldef} The Kolmogorov operator $\cP:L^{\infty }(\mathbb{T}%
^{d})\rightarrow C(\mathbb{T}^{d})$ associated with the system %
\eqref{eq: SDE in Td}, for \(\delta=0\), is defined as follows. Let $\phi \in L^{\infty }(%
\mathbb{T}^{d})$, then for all $ x\in \mathbb{T}^{d}$ we set
\begin{equation*}
(\cP\phi )(x):=\mathbb{E}[\phi (X_1^x)],
\end{equation*}
where $X_1^x$ is the solution at time $1$ of the SDE \eqref{eq: SDE in Td} with initial condition $x$.
\end{definition}

The solution $X_1^x$, thanks to Proposition~\ref{prop:kernel}, admits a (unique) density $p = p(x,y),$ and thus we have
\begin{equation*}
(\cP\phi )(x)=\int_{\mathbb{T}^{d}}\phi (y)p(x,y)dy.
\end{equation*}%
By using H\"older inequality as follows
$$
\| \mathcal{P} \phi \|_\infty \leq \| \phi \|_\infty \| p (x, \cdot) \|_1.
$$
and the fact that $\| p (x, \cdot) \|_1 $ is bounded (by \eqref{eq: mass conservation}), we get
\begin{equation}
\norm{ \cP\phi }_{\infty }\leq \norm{ \phi}_{\infty }.  \label{inft}
\end{equation}

 If $\nu $ is a Borel signed measure on $%
\mathbb{T}^{d}$ 
\begin{equation*}
\int_{\mathbb{T}^{d}}(\cP\phi )(x) d\nu(x) = \int_{\mathbb{T}^{d}}\int_{\mathbb{T%
}^{d}}\phi (y)p(x,y)dyd\nu (x)
\end{equation*}%
supposing that $\nu $ has a density with respect to the Lebesgue measure {%
 $f\in L^{1}(\mathbb{T}^{d})$ i.e. $d\nu =f(x)dx$. We can thus
write 
\begin{equation}
\int_{\mathbb{T}^{d}}(\cP\phi )(x)d\nu (x)=\int_{\mathbb{T}^{d}}\phi (y)\left(
\int_{\mathbb{T}^{d}}p(x,y)f(x)dx\right) dy.  \label{duality}
\end{equation}%
}
Now we define the transfer operator $\cL:L^{1}(\mathbb{T}^{d})\rightarrow
L^{1}(\mathbb{T}^{d})$ associated with the evolution of the system \eqref{eq: SDE in Td} for \(\delta=0\) at time $1$.

\begin{definition}[Transfer operator]\label{def:to}
Given $f\in L^{1}(\mathbb{T}^{d})$ we define the measurable function $\cL f$ as follows. For almost each $y\in 
\mathbb{T}^{d}$ let 
\begin{equation}
\lbrack \cL f](y):=\int p(x,y)f(x)dx.  \label{eq:stoctransfer}
\end{equation}
\end{definition}

By \eqref{duality} we now get the duality relation between the Kolmogorov
and the transfer operator 
\begin{equation}
\int (\cP\phi )(x)d\nu (x)=\int \phi (y)[\cL f](y)dy,  \label{eq:stocduality}
\end{equation}
whenever \(d\nu = f(x)\ dx\). The following are some well-known and basic facts about integral operators with
kernels $p(x,y)$ in $L^{p}$, which will be useful:\\

\begin{itemize}
\item If $p\in L^2$ then using \eqref{eq: kernel estimate H1}, the operator $\cL:L^{2}\rightarrow L^{2}$ is bounded and  
\begin{equation}
\norm{\cL f}_{2}\leq \norm{p}_{2}\norm{f}_{2}  \label{KF}
\end{equation}%
(see Proposition 4.7 in II.\S 4 \cite{C}).

\item If $p\in L^{\infty}$, then using \eqref{eq: kernel estimarte Linfty}, 
\begin{equation}
\norm{\cL f}_{\infty }\leq \norm{p}_{\infty }\norm{f}_{1}  \label{KF2}
\end{equation}
and the operator $\cL:L^1\rightarrow L^{\infty }$ is bounded.
\end{itemize}
Now, we discuss the properties of the transfer operator on the space \(L^1\).

\begin{lemma}\label{lem:contraction}
The operator $\cL$ preserves the integral and is a weak contraction with
respect to the $L^1$ norm.\label{7}
\end{lemma}

\begin{proof}
The first statement directly follows from \eqref{eq:stocduality} setting $%
\phi =1$. For the second part of the statement, using that the kernel $p(x,y) $ is positive thanks to Proposition \ref{prop:kernel}, we have
$$
\vert (\mathcal{L}f) (y) \vert \leq \int_{\T^d} p(x,y) \vert f (x) \vert \, dx.
$$
Integrating in space with respect to the variable $y$, applying Fubini-Tonelli, and using that $\int_{\T^d} p(x,y) \, d y  =1,$ we conclude
$$
\| \mathcal{ L} f \|_{1} \leq \int_{\T^d} \vert f (x) \vert \int_{\T^d} p(x,y) \, dy\, dx = \int_{\T^d} \vert f(x) \vert \, dx = \| f \|_{_1}.
$$
Hence $\mathcal{L}$ is a weak contraction in $L^1(\T^d).$
\end{proof}

\begin{remark}\label{rem:Markov_op}
     \(p(x,y)\) being positive (see Proposition~\ref{prop:kernel}) implies that $\cL$ is a positive operator. Thus, we also get that $\cL$ is a Markov operator having kernel $p$.
\end{remark}

Finally, in the following section, we define/construct the strong space that we want to study the operator \(\cL\) on.

Consider the space \(L^2(\bT^d)\), consisting of the square integrable functions on \(\bT^d\), as our strong space. Then our transfer operator \(\cL:L^2(\bT^d)\to L^2(\bT^d)\) is a compact operator being a kernel operator on a compact domain (see Proposition 4.7 in II.\S 4 \cite{C}).

\begin{lemma}\label{lem:simple_ev}
The transfer operator \(\cL\) has spectral gap on $L^2$ and it has 1 as its simple and only eigenvalue on the unit circle. Additionally, \(\cL\) has a unique normalised invariant function \(f_0\in L^2\) such that \(\int_{\T^d} f_0\ dx=1\).
\end{lemma}
\begin{proof}

\emph{Step $1$: $1$ is an eigenvalue for $\cL$}.
    By Lemma~\ref{lem:contraction}, the operator \(\cL\) is integral preserving, which implies that \(1\) lies in the spectrum of \(\cL\). The operator being compact implies that $1$ is an eigenvalue of $\mathcal{L}$. Indeed, consider the adjoint operator \(\cL^*:L^2\to L^2\) defined by the duality relation \(\langle \cL f, g\rangle = \langle f, \cL^*, g\rangle\). \(\cL\) being integral preserving, we have, for all \(f\in L^2\) and the constant function \(\bf{1}\), \(\langle f, \cL^*{\bf{1}}\rangle = \langle \cL f, {\bf{1}}\rangle = \int \cL f \ dx = \int f \ dx = \langle f, {\bf{1}}\rangle\). This implies \(\cL^*{\bf{1}}= {\bf{1}}\), thus 1 is in the spectrum of both \(\cL^*\) and \(\cL\). Since \(\cL\) is compact, every non-zero eigenvalue is an eigenvalue with finite multiplicity. Since $1 \in \sigma (\mathcal{L})$, it follows that $1$ is an eigenvalue of $\mathcal{L}.$
    
   \emph{Step $2$: if $\theta \in \mathbb{C} \setminus \{ 0,1 \}$ is an eigenvalue for $\cL$, then $\vert \theta \vert \leq \rho <1$}. Let $c:= \inf_{x,y} p(x,y) \in (0,1)$ be given by Proposition \ref{prop:kernel}. Then, consider the decomposition
   $$
   p(x,y) = c +  q (x,y),  \qquad  q(x,y) \geq 0 .
   $$
   Note that, since $\int_{\T^d} p (x,y) \, dy =1$, then $\int_{\T^d} q(x,y)\, dy = 1 - c$. Now, let $g \in L^1(\T^d)$ be such that $\int_{\T^d} g (x) \, dx  =0$, then
   $$
   (\mathcal{L} g)(y) = \int_{\T^d} p (x,y) g(x) \, dx = c \int_{\T^d} g(x) \, dx + \int_{\T^d} q(x,y) g(x) \, dx = \int_{\T^d} q(x,y) g(x) \, dx.
   $$
   Therefore, integrating in space with respect to the variable $y$, and applying Fubini-Tonelli, we have
   $$
   \| \mathcal{L} g \|_{1} \leq \int_{\T^d} \int_{\T^d} q(x,y) \vert g (x) \vert \, dx \, dy = \int_{\T^d} \vert g(x) \vert \left[\int_{\T^d} q(x,y) \, dy \right] \, dx = (1-c) \int_{\T^d} \vert g(x) \vert \, dx  = (1-c) \| g \|_{1}.
   $$
   So we have proved that
   \begin{equation}
\| \mathcal{L} g \|_1 \leq \rho \| g \|_1, \qquad \rho \in (0,1), \qquad  \text{for any }g \in L^1(\T^d), \ \int_{T^d}g  =0.
       \label{eq: strict contraction on zero average}
   \end{equation}
   Now, assume by contradiction that $ \mathcal{L} f_\theta =  \theta f_\theta$, for some $\vert \theta \vert =1$, $\theta \neq 1$, and $f_\theta \in L^2(\T^d)$. Since $\mathcal{L}$ is integral preserving, we have
   $$
   \int f_\theta = \int \mathcal{L} f_\theta = \theta  \int f_\theta, 
   $$
   but hence $(1- \theta ) \int f_\theta  = 0$. Since we are assuming that $\theta \neq 1$, we conclude $\int f_\theta  = 0,$ and thus $f_\theta$ belongs to the zero average space. Hence, by \eqref{eq: strict contraction on zero average}, we have
   $$
   \vert \theta \vert \| f_\theta \|_1 = \| \mathcal{L} f_\theta \|_1 \leq \rho \| f_\theta \|_1, \qquad \rho \in (0,1).
   $$
   Thus, since $f_\theta \neq 0$, we conclude $\vert \theta \vert \leq \rho < 1$. This is a contradiction, and hence $1$ is the only eigenvalue on the unit circle.

   \emph{Step $3$: $1$ is a simple eigenvalue.} Denote by $u,v \in L^2(\T^d)$ two eigenvectors associated to the eigenvalues $1$. Assume moreover that they are normalised in such a way that
   $$
   \int_{\T^d} u \, dx = \int_{\T^d} v \, dx  = 1.
   $$
   Then
   $$
   \int_{\T^d} u-v \, dx = \int_{\T^d} u \, dx - \int_{\T^d} v \, dx = 0.
   $$
   Thus, we can apply the contraction inequality \eqref{eq: strict contraction on zero average} to the zero-average function $g:= u-v$, we get
   $$
    \| g \|_1 =  \| \mathcal{L} g\|_1 \leq \rho \| g \|_1, \qquad \rho \in (0,1).
   $$
   Hence $g = 0$, that is $u=v$. Hence $u$ and $v$ belongs to the same one-dimensional eigenspace.  

   Lastly, the invariant function $f_0 \in L^2$ can be obtained by considering an eigenfunction of $\mathcal{L}$ corresponding to the eigenvalue $1$ on $L^2(\T^d),$ and then normalising it to have $\int_{\T^d} f_0 \, dx  =1.$
\end{proof}
\begin{remark}\label{rem:justification}
    Since Remark~\ref{rem:stationary_measure} asserts that the SDE under consideration has a unique stationary measure, and the above lemma asserts the uniqueness of the invariant measure. We remark that the unique stationary measure $\mu$ of the SDE \eqref{eq: SDE in Td}
     must be also invariant for $\cL$, and this must be the unique invariant probability measure found above. This also implies that $\mu$ has a density in the strong space $L^2$. We conclude that the unique stationary measure of the SDE is indeed equal to the invariant measure of the transfer operator.
\end{remark}

\begin{definition}
We say that an operator $\cL:L^{2}\rightarrow L^{2}$ has  \emph{exponential contraction of the zero average space} $V$ if there are $C\geq 0$ and $%
\lambda <0$ such that $\forall g\in V$
\begin{equation}
\Vert \cL^{n}g\Vert _{2}\leq Ce^{\lambda n}\Vert g\Vert _{2}  \label{equil}
\end{equation}%
for all $n\geq 0$.
\end{definition}

\begin{proposition}\label{prop:conv_eq}
 The operator \(\cL\) has exponential contraction of the zero average space.
\end{proposition}

\begin{proof}
  By Lemma~\ref{lem:simple_ev}, we get that \(\cL:L^2(\bT^d)\to L^2(\bT^d)\) has spectral gap, which implies it can be decomposed as
\[
\cL=\Pi+{\mathcal Q}
\]
where \(\Pi\) is the projection operator onto the one-dimensional eigenspace corresponding to the eigenvalue 1, and the spectral radius of \({\mathcal Q}\) is strictly less than 1, that is \(\norm{{\mathcal Q}}_{2}<1\). Furthermore, for \(f\in V_{L^2}\), we have \(\cL^n f={\mathcal Q}^n f\), and
\[
\norm{{\mathcal Q}^n f}_2\leq \norm{{\mathcal Q}^n}_2\norm{f}_2 \leq Ce^{-\gamma n}\norm{f}_2
\]
for some $C\geq0 $ and $\gamma \in (0,1)$ and hence the claim.
\end{proof}
\vspace{1cm}

\section{Perturbations and linear response}
\label{sec: perturbations and linear response}

For $\bar{\delta}>0$ and $\delta\in [0,\bar{\delta})$, we consider a family of integral-preserving, compact operators $\cL_{\delta }:L^{2}\rightarrow L^{2}$  associated to the SDE \eqref{eq: SDE in Td} ; we think of $\cL_{\delta }$ as perturbations of $\cL_{0}$, defined as
\[
\cL_\delta f(y) := \int p^\delta(x,y)f(x)\ dx.
\]
By definition, for every $\delta\in [0,\bar{\delta})$, the operators \(\cL_\delta\) are defined and compact as operators on \(L^2(\bT^d)\).
We say that $f_\delta\in L^2$ is an \emph{invariant density} of $\cL_\delta$ if $\cL_\delta f_\delta=f_\delta$.
We will see that under natural assumptions, the operators $\cL_\delta$, $\delta\in[0,\bar{\delta})$, have a family of normalized invariant functions $f_{\delta } \in L^{2}$ (see Proposition~\ref{prop:linearresponse}).
Furthermore, for suitable perturbations the invariant densities vary smoothly in $L^{2}$ and we get an explicit formula for the resulting derivative $\frac{df_\delta}{d \delta}$.

    Using Remark~\ref{rem:density-relation}, one can see that, for any \(\delta\in [0,1)\), \(f\in L^2(\bT^d)\), the relation between the transfer operator \(\cL_\delta\) and the density  \(p^{\delta,f}(t,\cdot)\) (see Section~\ref{sec:L^2-initial condition}), is given as follows:
    \begin{equation}\label{eq:to-density_relation}
        \cL_\delta f(y) = p^{\delta,f}(1,y).
    \end{equation}

Now that we have detailed information on the structure of the transfer operator and on the perturbative framework we intend to implement, we turn to stationary measures and to the linear response of the stationary measure under perturbations of the underlying system. Since the operators under consideration are compact Markov operators, these questions can be treated in a rather direct way within the classical spectral-perturbative approach. 
In particular, we shall rely on the following statement proved in \cite{AFG22}, which is especially well suited to the applications developed in the remainder of this work. We also note that in \cite{AFG22} the result is formulated and proved for the space $L^2{([0,1]})$, but the argument carries over verbatim to the setting relevant for us, namely $L^2({\T^d})$.

\begin{lemma}\label{lem:L_delta_cty}
Assume $b, \eta \in L^{\infty}(\T^d; \mathbb{R}^d).$ Let $\cL_{\delta }$ be the family of transfer operators associated with the kernels $p^{\delta }$ as above. Then there exist constants $C>0$ such that for each $\delta \in \lbrack 0,\overline{\delta })$, 
\begin{equation}
\norm{\cL_{\delta }-\cL_{0}}_2\leq \delta C.
\end{equation}
Furthermore, let \(f_0\) be the normalised invariant function of $\mathcal{L}_0$ and denote by $\dot r^{f_0}$ be the solution of \eqref{eq: eq for dot r regular initial condition}. Then $\dot r^{f_0} \in V_{L^2}$ and
\[
\underset{\delta \rightarrow 0}{\lim }  \frac{(\cL_{\delta}-\cL_{0})}{%
\delta }f_0 = \dot r^{f_0}, \qquad \text{ in } L^2(\T^d).
\]
\end{lemma}

\begin{proof}
Let \(f\in L^2(\bT^d)\), then using \eqref{KF}, and Proposition \ref{prop:new_prop}, there exists $C=C(\|b\|_{{\infty}}, \| \eta \|_{{\infty}},d,T)>0$ such that for any \(\delta\in [0,1)\), we have the following
\begin{eqnarray*}
\norm{\cL_{\delta }f-\cL_{0}f}_2  &\leq \norm{p^\delta-p}_2\norm{f}_2 \leq \delta C \norm{f}_2
\end{eqnarray*}%

 We now prove the second claim. Let $f_0$ be the normalized invariant function of $\cL_0$. By \eqref{eq:to-density_relation}
 $$
 \cL_\delta f_0 = p^{\delta, f_0} (1, \cdot), \qquad \cL_0 f_0 = p^{0, f_0}(1, \cdot).
 $$
 Hence
 $$
    \frac{\cL_\delta - \cL_0}{\delta} f_0 =
    \frac{p^{\delta, f_0} (1, \cdot) - p^{0, f_0}(1, \cdot )}{\delta }.
 $$
 Therefore, using Proposition~\ref{prop: convergence density pdelta regular p0}
 $$
 \lim_{\delta \to 0} \left\|  \frac{\cL_{\delta} - \cL_0}{ \delta } f_0 - \dot{r}^{f_0}\right\|_2 = \lim_{\delta \to 0} \left\|  \frac{p^{\delta, f_0} (1, \cdot) - p^{0, f_0}(1, \cdot )}{\delta } -  \dot{r}^{f_0}\right\|_2 = 0.
 $$
It remains to show that $ r^{f_0}\in V_{L^2}.$ Since both $\cL_{\delta}$ and $\cL_0$ preserve the integral, for every $\delta >0$ we have
 $$
 \int_{\T^d} \frac{ \cL_{\delta } - \cL_0}{\delta }f_0 \, dx = \frac{1}{\delta } \left(  \int_{\T^d} \cL_{\delta} f_0 \, dx - \int_{\T^d} \cL_0 f_0 \, dx\right) = 0.
 $$
 Passing to the limit in $L^2(\T^d),$ we obtain $\int_{\T^d} \dot r^{f_0} \, dx  = 0$, and in conclusion $\dot r^{f_0} \in V_{L^2}.$
\end{proof}
As a preliminary step towards Theorem~\ref{thm: observables linear response}, we state the
following proposition, which is adapted from \cite[Theorem~1]{AFG22} to our
family of transfer operators. The results established above verify the
perturbative assumptions needed in that framework, while compactness and the
convergence properties of the operators yield the spectral information required
to control the invariant densities. This gives existence and uniqueness of
normalized invariant densities for sufficiently small perturbations, together
with the corresponding linear response formula for the invariant density.
\begin{proposition}
\label{prop:linearresponse}
Assume that $b,\eta \in L^{\infty}(\T^d;\mathbb{R}^d)$, fix $\bar\delta>0$, and
let $(\cL_\delta)_{\delta\in[0,\bar\delta)}$ be the family of transfer operators
introduced above. The following holds.
\begin{enumerate}
    \item[(i)] Then the unperturbed operator $\cL_0$ admits a unique normalized invariant density
$f_0\in L^2(\T^d)$, with $
\int_{\T^d} f_0\,dx=1, $
and $\cL_0$ is exponentially contracting on $V_{L^2}$.
\item[(ii)] There exists
$\delta_2\in(0,\bar\delta]$ such that, for every $\delta\in[0,\delta_2)$, the
operator $\cL_\delta$ admits a unique normalised invariant density
$f_\delta\in L^2(\T^d)$ satisfying $
\int_{\T^d} f_\delta\,dx=1,$ and $\cL_\delta$ is exponentially contracting on $V_{L^2}$.
\item[(iii)]  The resolvent $
(Id-\cL_0)^{-1}\colon V_{L^2}\to V_{L^2}$ is continuous, and if $\dot r^{f_0}\in V_{L^2}$ denotes the solution of
\eqref{eq: eq for dot r regular initial condition}, then
\[
\lim_{\delta\to0}
\left\|
\frac{f_\delta-f_0}{\delta}-(Id-\cL_0)^{-1}\dot r^{f_0}
\right\|_2=0.
\]
In particular, $(Id-\cL_0)^{-1}\dot r^{f_0}$ gives the first-order response of the
invariant density.
\end{enumerate}
\end{proposition}
\begin{proof}
(i). The existence and uniqueness of a normalised invariant density
$f_0\in L^2(\T^d)$ for $\cL_0$, together with the normalisation $
\int_{\T^d} f_0\, dx=1,$ follow from Lemma \ref{lem:simple_ev}. Moreover,
Proposition \ref{prop:conv_eq} shows that $\cL_0$ is exponentially
contracting on $V_{L^2}$.

(ii). By Lemma \ref{lem:contraction}, each operator $
\cL_\delta:L^2(\T^d)\to L^2(\T^d) $ is compact. Since $\cL_\delta$ is a transfer operator, it preserves the
integral, which means $
\int_{\T^d}\cL_\delta g\,dx=\int_{\T^d}g\,dx,$ for any $g\in L^2(\T^d).$ This is equivalent to $\cL_\delta^*1=1$. Since $\cL_\delta^*$ is also compact and the non-zero spectra of $\cL_\delta$ and $\cL_\delta^*$ coincide, it follows
that $1$ is an eigenvalue of $\cL_\delta$. Hence $\cL_\delta$ admits a
nontrivial fixed point.

Next, by (i), there exists $n_0 \in \mathbb{N}$ such that $
\|\cL_0^{n_0}|_{V_{L^2}}\|_{L^2\to L^2}\leq \frac12.$ On the other hand, Lemma \ref{lem:L_delta_cty} implies that
$$
\|\cL_\delta-\cL_0\|_{L^2\to L^2}\to 0
\qquad\text{as }\delta\to0.
$$
Therefore
$$
\|\cL_\delta^{n_0}-\cL_0^{n_0}\|_{L^2\to L^2}\to 0,
\qquad\text{as }\delta\to0,
$$
and so there exists $\delta_2\in(0,\bar\delta]$ such that, for every
$\delta\in[0,\delta_2)$,
$$
\|\cL_\delta^{n_0}|_{V_{L^2}}\|_{L^2\to L^2}\leq \frac23.
$$
It follows that $\cL_\delta$ is exponentially contracting on $V_{L^2}$ for
every $\delta\in[0,\delta_2)$.

Fix now $\delta\in[0,\delta_2)$, and let $f_\delta$ be a non-trivial fixed
point of $\cL_\delta$. Since $\cL_\delta$ preserves the integral, we can
normalise $f_\delta$ so that $
\int_{\T^d}f_\delta\,dx=1. $ This normalised invariant density is unique. Indeed, if $f_\delta'$ is another
invariant density with the same normalization, then
$f_\delta-f_\delta'\in V_{L^2}$ and
\[
\cL_\delta^n(f_\delta-f_\delta')=f_\delta-f_\delta',
\qquad n\in\mathbb{N}.
\]
Since $\cL_\delta$ is exponentially contracting on $V_{L^2}$, we obtain
$f_\delta=f_\delta'$.

(iii) We first show that the resolvent $
(Id-\cL_0)^{-1}\colon V_{L^2}\to V_{L^2} $ is continuous. Let $g\in V_{L^2}$. Since $\cL_0$ is exponentially
contracting on $V_{L^2}$ by (i), the Neumann series converges in
$L^2(\T^d)$ and gives
\[
(Id-\cL_0)^{-1}g=\sum_{n=0}^\infty \cL_0^n g.
\]
Hence there exist constants $C_0,\lambda_0>0$ such that
\[
\|(Id-\cL_0)^{-1}g\|_2
\leq
\sum_{n=0}^\infty \|\cL_0^n g\|_2
\leq
C_0\sum_{n=0}^\infty e^{-\lambda_0 n}\|g\|_2
\leq C_1\|g\|_2
\]
for some $C_1>0$. Thus the resolvent is continuous on $V_{L^2}$.

We next show that $
f_\delta\to f_0$ in $L^2(\T^d).$ Let $
g_\delta:=f_\delta-f_0.$ Since both $f_\delta$ and $f_0$ have integral $1$, we have
$g_\delta\in V_{L^2}$. Using the invariance relations
\[
\cL_\delta f_\delta=f_\delta,
\qquad
\cL_0f_0=f_0,
\]
we obtain
\[
g_\delta
=
\cL_\delta^{n_0}f_\delta-\cL_0^{n_0}f_0
=
(\cL_\delta^{n_0}-\cL_0^{n_0})f_0+\cL_\delta^{n_0}g_\delta.
\]
Therefore
\[
\|g_\delta\|_2
\leq
\|(\cL_\delta^{n_0}-\cL_0^{n_0})f_0\|_2
+\frac23\|g_\delta\|_2,
\]
and hence
\[
\|f_\delta-f_0\|_2
=
\|g_\delta\|_2
\leq
3\|(\cL_\delta^{n_0}-\cL_0^{n_0})f_0\|_2.
\]
Since $\cL_\delta^{n_0}\to\cL_0^{n_0}$ in operator norm, we conclude that
$f_\delta\to f_0$ in $L^2(\T^d)$.

We now prove the linear response formula. Using again the invariance
relations,
\[
(Id-\cL_0)(f_\delta-f_0)=(\cL_\delta-\cL_0)f_\delta.
\]
Dividing by $\delta$, we obtain
\[
(Id-\cL_0)\frac{f_\delta-f_0}{\delta}
=
\frac{\cL_\delta-\cL_0}{\delta}f_\delta.
\]
Since $f_\delta-f_0\in V_{L^2}$ and both $\cL_\delta$ and $\cL_0$ preserve
the integral, the right-hand side also belongs to $V_{L^2}$. Therefore,
applying the bounded resolvent on $V_{L^2}$, we get
\[
\frac{f_\delta-f_0}{\delta}
=
(Id-\cL_0)^{-1}\frac{\cL_\delta-\cL_0}{\delta}f_\delta.
\]
We decompose the right-hand side as
\[
\frac{f_\delta-f_0}{\delta}
=
(Id-\cL_0)^{-1}\frac{\cL_\delta-\cL_0}{\delta}f_0
+
(Id-\cL_0)^{-1}\frac{\cL_\delta-\cL_0}{\delta}(f_\delta-f_0).
\]

By Lemma~\ref{lem:L_delta_cty},
\[
\frac{\cL_\delta-\cL_0}{\delta}f_0\to \dot r^{f_0}
\qquad\text{in }L^2(\T^d),
\]
and therefore
\[
(Id-\cL_0)^{-1}\frac{\cL_\delta-\cL_0}{\delta}f_0
\to
(Id-\cL_0)^{-1}\dot r^{f_0}
\qquad\text{in }L^2(\T^d).
\]

For the remainder term, Lemma~\ref{lem:L_delta_cty} yields a constant $C_2>0$
such that
\[
\left\|\frac{\cL_\delta-\cL_0}{\delta}\right\|_{L^2\to L^2}\leq C_2
\]
for all sufficiently small $\delta$. Hence
\[
\left\|
(Id-\cL_0)^{-1}\frac{\cL_\delta-\cL_0}{\delta}(f_\delta-f_0)
\right\|_2
\leq
\|(Id-\cL_0)^{-1}\|_{V_{L^2}\to V_{L^2}}
\left\|\frac{\cL_\delta-\cL_0}{\delta}\right\|_{L^2\to L^2}
\|f_\delta-f_0\|_2,
\]
and the right-hand side tends to \(0\) because \(f_\delta\to f_0\) in
\(L^2(\T^d)\).

Combining the previous limits, we conclude that
\[
\lim_{\delta\to0}
\left\|
\frac{f_\delta-f_0}{\delta}-(Id-\cL_0)^{-1}\dot r^{f_0}
\right\|_2=0.
\]
This concludes the proof.
\end{proof}

We now pass from the linear response of the invariant density to the linear response of expectations of observables. The argument is immediate once Proposition \ref{prop:linearresponse}(iii) is paired with the observable $\phi$.
\begin{proof}[Proof of Theorem \ref{thm: observables linear response}]
Let $\phi\in L^2(\T^d)$ and $\eta\in L^{\infty}(\T^d;\mathbb{R}^d)$. By Proposition \ref{prop:linearresponse}(iii)
$$
\frac{f_\delta-f_0}{\delta}
\to
(Id-\cL_0)^{-1}D_{f_0,1}^{L^{\infty}}(\eta) \qquad\text{in }L^2(\T^d).
$$
Testing against $\phi$ and using the $L^2$ duality, we obtain
$$
\lim_{\delta\to0}
\int_{\T^d} \phi \frac{f_\delta-f_0}{\delta}\,dx
=
\int_{\T^d} \phi \left[(Id-\cL_0)^{-1}D_{f_0,1}^{L^{\infty}}(\eta)\right]\, dx.
$$
This proves that the response
$$
R(\phi,\eta)
:=
\lim_{\delta\to 0}
\frac{\int_{\T^d}\phi f_{\delta}\,dx-\int_{\T^d}\phi f_0\,dx}{\delta}
$$
exists and is given by
$$
R(\phi,\eta)
=
\int_{\T^d}
\phi\,\left[(Id-\cL_0)^{-1}D_{f_0,1}^{L^{\infty}}(\eta)\right]dx.
$$
Moreover, Proposition \ref{cont} shows that $D_{f_0,1}^{L^{\infty}}$ is linear and continuous from $L^{\infty}(\T^d;\mathbb{R}^d)$ to $L^2(\T^d)$, and its values have zero average. Since $(Id-\cL_0)^{-1}\colon V_{L^2}\to V_{L^2}$ is continuous, it follows that
$$
|R(\phi,\eta)| \leq \|\phi\|_2 \|(Id- \cL_0)^{-1}\|_{V_{L^2}\to V_{L^2}} \|D_{f_0,1}^{L^{\infty}}(\eta)\|_2
\leq C\|\phi\|_2\|\eta\|_{\infty}
$$
for some constant $C>0$. Hence $\eta\mapsto R(\phi,\eta)$ is a continuous linear functional on $L^{\infty}(\T^d;\mathbb{R}^d)$.
\end{proof}

For the optimal response problem it is convenient to work on a Hilbert space of admissible perturbations rather than directly on $L^\infty$. We therefore consider a Hilbert space $\cH$ continuously embedded in $L^\infty(\T^d;\mathbb{R}^d)$. The next lemma shows that the response functional extends naturally to this setting: the extension is well defined, linear, and continuous, and it agrees with the original response functional whenever both are defined.

\begin{lemma}[Continuous extension of the response functional to $\cH$]
    Assume $b, \eta \in L^\infty (\T^d; \mathbb{R}^d)$. Let $\cH$ be any Hilbert space such that $\cH \hookrightarrow L^\infty (\T^d; \mathbb{R}^d).$ For $\phi \in L^2(\T^d)$ and $\eta \in \cH$, define
    \begin{equation}
 R_{\cH} (\phi, \eta) := \int_{\T^d} \phi \left[ (Id - \cL_0)^{-1} D_{f_0, 1}^\cH (\eta) \right] \, dx.
        \label{eq: response on cV}
    \end{equation}
    Then $R_{\cH}(\phi, \cdot) : \cH \to \mathbb{R}$ is linear and continuous. Moreover
    $$
R_{\cH}(\phi,\eta)=R(\phi,\eta),
$$
for any $\phi \in L^2(\T^d)$ and $\eta \in \cH \cap L^\infty (\T^d, \mathbb{R}^d).
    $
    \label{lemma: response continuous on cV}
\end{lemma}
\begin{proof}
    By Corollary \ref{cor: extension continuity derivative in cV}, the map $D^{\cH}_{f_0, 1}: \cH \to L^2(\T^d)$ is linear and continuous. By Corollary \ref{cor: extension continuity derivative in cV}, for every $\eta \in \cH$, $D_{f_0,1}^{\cH}(\eta)$ has zero average, hence $D_{f_0,1}^\cH (\eta ) \in V_{L^2}.$ Therefore the resolvent operator
    $$
    (Id- \mathcal{L}_0)^{-1} : V_{L^2} \to V_{L^2}
    $$
    can be applied to $D_{f_0,1}^\cH (\eta )$. Using Cauchy-Schwarz and the boundedness of the resolvent on $V_{L^2}$, we get
    $$
    \vert R_{\cH} (\phi , \eta ) \vert \leq \| \phi \|_2  \| (Id- \cL_0)^{-1} D_{f_0, 1}^\cH (\eta) \|_2 \leq \| \phi \|_2 \| (Id- \cL_0)^{-1} \|_{V_{L^2} \to V_{L^2}} \| D_{f_0,1}^{\cH} (\eta) \|_2  \leq C \| \phi \|_2 \| \eta \|_{\cH},
    $$
    for some constant $C>0$. This proves the continuity of $R_{\cH}(\phi, \cdot)$. Finally, for $\eta \in \cH \cap L^\infty (\T^d; \mathbb{R}^d)$ Proposition \ref{prop:linearresponse}(iii) gives
    $$
    R(\phi, \eta) = \int_{\T^d} \phi \left[ (Id- \cL_0)^{-1} D_{f_0, 1}^{L^\infty}(\eta) \right] \, dx.
    $$
    Hence, $R$ coincides with the functional $R_\cH$ defined in \eqref{eq: response on cV} since $ D_{f_0, 1}^{L^\infty}(\eta) = D_{f_0,1}^{\cH} (\eta)$ for any $\eta \in \cH \cap L^\infty (\T^d; \mathbb{R}^d).$
\end{proof}

\section{Optimal response}
\label{sec: optimal response}

\subsection{Preliminaries on convex optimization}

Here we recall some general results on convex optimization on Hilbert spaces we will use in the following. For details and the proofs of the results stated in this form we refer to \cite{AFG22}.

Let $\cH$ be a Hilbert space continuously embedded in $L^\infty (\T^d;\mathbb{R}^d)$. Denote by \(P \subset \cH\) be a bounded and convex subset of $\cH$.
\begin{definition}
\label{stconv}We say that a convex closed set $P\subseteq \cH$ is \emph{strictly convex} if for each pair $x,y\in P$, with $x \neq y$, and for all $\gamma\in(0,1)$, the points $\gamma x+(1-\gamma)y\in \mathrm{int}(P)$, where \(\mathrm{int}(P)\) is the relative interior\footnote{The relative interior of a closed convex set $C$ is the interior of $C$ relative to the closed affine hull of $C$.} of \(P\).
\end{definition}

We now recall the abstract convex optimisation facts that will be used to prove existence and uniqueness of optimal perturbations. The point is that, once the response is viewed as a continuous linear functional on a Hilbert space, the geometry of the admissible set $P$ determines whether the maximisation problem admits a solution and whether that solution is unique. Suppose in addition that $\cH$ is separable and let $\mathcal{J}:\cH\to\mathbb{R}$ be a continuous linear functional. We consider the abstract problem of finding $p_*\in P$ such that
\begin{equation}
\mathcal{J}(p_*)=\max_{p\in P}\mathcal{J}(p) .  \label{gen-func-opt-prob}
\end{equation}%

The existence and uniqueness of an optimal perturbation follows from properties of $P$ as stated in the following two propositions.
(See \cite{AFG22}  Proposition 4.1 and 4.3 )
\begin{proposition}[Existence of the optimal solution]
\label{prop:exist} Let $P$ be bounded, convex, and closed in $\cH$.
Then problem~\eqref{gen-func-opt-prob} has at least one solution.
\end{proposition}

Upgrading convexity of the feasible set $P$ to strict convexity provides
uniqueness of the optimal solution.

\begin{proposition}[Uniqueness of the optimal solution]
\label{prop:uniqe} Suppose $P$ is closed, bounded, and strictly convex
subset of $\cH$, and that $P$ contains the zero vector in its
relative interior. If $\mathcal{J}$ is not uniformly vanishing on $P$
then the optimal solution to \eqref{gen-func-opt-prob} is unique.
\end{proposition}

Note that in the case when $\mathcal{J}$ is uniformly vanishing, all the elements of $P$ are solutions of the problem $(\ref{gen-func-opt-prob}).$ 

We next prove Theorem \ref{thm:gen-cnvx}, which establishes existence and uniqueness of an optimal perturbation in the general convex setting announced in Section \ref{sec:main_results}. The key point is that, once the response functional is known to be linear and continuous on $\cH$, the optimal response problem fits directly into the abstract convex optimisation framework recalled above.
\begin{proof}[Proof of Theorem \ref{thm:gen-cnvx}]
Note that $R_{\cH}(\phi,\cdot)$ is continuous by Lemma \ref{lemma: response continuous on cV}. Since $R_{\cH}(\phi,\cdot)$ is not identically zero on $P$ by assumption, Proposition \ref{prop:exist} and Proposition \ref{prop:uniqe} can be applied, and thus we obtain the statement.
\end{proof}

{We now turn to the case where the admissible set $P$ is the closed unit ball of $\cH$, namely
\begin{equation*}
R_{\cH}(\phi,\eta_{\mathrm{opt}}) = \max_{\eta\in P} R_{\cH}(\phi,\eta), \qquad \text{where}\qquad
P := \{\eta \in \cH : \|\eta\|_{\cH} \leq 1 \}.
\end{equation*}
We now prove Theorem \ref{gen:algo}, which provides an explicit characterisation of the optimiser \(\eta_{\mathrm{opt}}\) in this setting. The proof combines the continuity of the response functional $R_{\cH}(\phi,\cdot)$ with the Riesz representation theorem. This characterisation will also be the starting point for the numerical approximation developed later, since it allows to approximate the optimal perturbation through its coefficients with respect to a chosen orthonormal basis of $\cH$.
\begin{proof}[Proof of Theorem \ref{gen:algo}]
By Lemma~\ref{lemma: response continuous on cV}, the map \(R_{\cH}(\phi,\cdot):\cH\to\mathbb{R}\) is linear and continuous. Hence, by the Riesz Representation Theorem, there exists a unique \(v\in \cH\) such that
$$
R_{\cH}(\phi,w)=\langle w,v\rangle_{\cH} \qquad\text{for all } w\in \cH.
$$
Therefore, for every $w\in \cH$ with $\|w\|_{\cH} \leq 1 $,
$$
R_{\cH}(\phi,w)=\langle w,v\rangle_{\cH} \leq \|w\|_{\cH}\|v\|_{\cH}\leq \|v\|_{\cH},
$$
and equality is attained at $w=v/\|v\|_{\cH}$. Hence the maximum in \eqref{e:define optr} is realized at
$$
\eta_{\mathrm{opt}}=\frac{v}{\|v\|_{\cH}}.
$$
Uniqueness also follows from Theorem \ref{thm:gen-cnvx}, since the unit ball of a Hilbert space is strictly convex and \(R_{\cH}(\phi,\cdot)\) is not identically zero. Finally,
$$
R_{\cH}(\phi,b_i)=\langle b_i,v\rangle_{\cH},
$$
which characterizes the coefficients of $v$ in the basis $\{b_i\}_{i\in\mathbb{N}} $.
\end{proof}

We remark that the last proposition shows that the optimal perturbation can be approximated by computing the linear response $ R(\phi,b_i)$ of the expectation of $\phi$ on the elements of the basis. 

This approach was also used in \cite{GN25} and \cite{DGJ} and is related to the method used in \cite{FG23} where the Fourier approximation is used together with Lagrange multipliers to compute the optimal perturbation for  expanding maps of the circle.

In the following subsection, we see how to choose a suitable orthonormal basis to compute $v$ by the above strategy.
}

\section{Numerical examples}
\label{sec: numerical examples}

In this section we illustrate the optimal-response framework on several examples. For each system, once the observable $\phi$ and the admissible perturbation space $\cH$ are fixed, we approximate the response functional in a truncated Fourier basis, reconstruct its Riesz representative $v$, and then compute the corresponding optimal perturbation. A practical ingredient in this procedure is a systematic enumeration of the multi-indices $\vec n$ appearing in the Fourier expansion. We use an ordering strategy that works in arbitrary dimension $d$, following \cite{GN25}.

\begin{remark}
Throughout this section, the observable $\phi$ is fixed in each example.
Hence, for $\eta\in\mathcal H$, we write
$
R(\eta): = R( \phi , \eta).$ Moreover, by Lemma \ref{lemma: response continuous on cV}, the abstract functional
$R_{\mathcal H}(\phi,\cdot)$ agrees with $R(\phi,\cdot)$ on $\mathcal H$.
Therefore, in the numerical examples we simply write $R(\eta)$.
If $v\in\mathcal H$ denotes the Riesz representative of $R(\phi,\cdot)$, then
$$
R(\eta)=\langle v,\eta\rangle_{\mathcal H}
\qquad \text{for all }\eta\in\mathcal H.
$$
\label{rem: optimal response numerical example}
\end{remark}

We begin by describing the Fourier representation of $v$ on $\T^d$ and the associated enumeration of the multi-indices used in the computations.

\subsection{Fourier representation of the Riesz representative on $\T^d$}
\hfill\vspace{0.1in}
\label{s:fourier}

All the numerical examples considered below are posed on flat tori. In order to remain consistent with the notation of the previous sections, we write $\T^d$ for the $d$-dimensional torus. On $\T^d$, we use a natural orthogonal basis of $H^p(\T^d;\mathbb{R}^d)$, given by products of trigonometric functions. We view the $i$th coordinate of $\T^d$ as the interval
$[c_i-r_{box},\, c_i+r_{box}]$, where $r_{box}$ denotes the radius of the torus..
We adopt the multi-index notation
\[\vec{n}:=(n_1, \ldots n_d),\] 
and consider
\begin{equation}\begin{split}
\label{e:basis}
B^j_{\vec{n}}
=
e_j \prod_{i=1}^d b_{n_i}(x_i),\quad
1 \le j\le d ,\quad
n_i \ge 0,
\quad \textnormal{where} \quad \\
e_j = (0,0,\ldots,0,1,0,\ldots,0)\in \mathbb{R}^d;
\\
b_m(x_i) = 
\begin{cases}
  1 / \sqrt{r_{box}}, \quad m=0;
    \\
    \sqrt{\frac{2}{r_{box}}} \sin(\lfloor \frac{m+1}{2} \rfloor \frac{\pi}{r_{box}} (x_i-c_i)) \quad &m \textnormal{ odd}
    ;
    \\
    \sqrt{\frac{2}{r_{box}}} \cos(\lfloor \frac{m+1}{2} \rfloor \frac{\pi}{r_{box}} (x_i - c_i)) \quad &m \textnormal{ even}.
\end{cases}
\end{split}\end{equation}
Here, $\lfloor\cdot\rfloor$ denotes the floor function. Note that $\int b_m^2\,dx = 1$.
We define the normalized basis functions by
\[
\tilde B^j_{\vec{n}}
:=
\frac{B^j_{\vec{n}}}{||B^j_{\vec{n}}||_{H^p}}.
\]

Here the $H^p$ norm has the following expression (note that it in fact does not depend on $j$),  
\begin{equation}\begin{split}
\label{e:Hpnorm of Fourier}
\langle B^j_{\vec{n}}, B^j_{\vec{n}}\rangle_{H^p}
=
\sum_{l=0}^p C_l \sum_{\vec{k}=(1,\ldots,1)\in \mathbb{R}^l }^{(d,\ldots,d)} 
\int \left(\partial_{\vec{k}}
\prod_{i=1}^d b_{n_i}(x_i)\right)^2 dx   
\\
=
\sum_{l=0}^p C_l \sum_{\vec{k}=(1,\ldots,1)\in \mathbb{R}^l }^{(d,\ldots,d)} 
(\lfloor \frac{n_{k_1}+1}{2} \rfloor 2 \pi)^2
\ldots
(\lfloor \frac{n_{k_l}+1}{2} \rfloor 2 \pi)^2.
\end{split}\end{equation}
Note that if $m=0$ then $\lfloor \frac{m+1}{2} \rfloor=0$, so the above formula still applies if one of the directions being differentiated is $b_0\equiv 1$.
In this section, the weight coefficients $C_l$ in the $H^p$ norm are chosen so that
\[\begin{split}
\langle B^j_{\vec{n}}, B^j_{\vec{n}}\rangle_{H^p}
=
\sum_{l=0}^p \sum_{\vec{k}=(1,\ldots,1)\in \mathbb{R}^l }^{(d,\ldots,d)} 
(\lfloor \frac{n_{k_1}+1}{2} \rfloor )^2
\ldots
(\lfloor \frac{n_{k_l}+1}{2} \rfloor )^2.
\end{split}\]
Let $v\in\mathcal H$ be the Riesz representative of the linear functional
$\eta\mapsto R(\eta)$ on $\mathcal H$, that is,
$$
R(\eta)=\langle v,\eta\rangle_{\mathcal H}
\qquad \text{for all }\eta\in\mathcal H.
$$
Then the Fourier expansion of $v$ is
\begin{equation} \begin{split}
\label{e:vFourier}
v = \sum_{j=1}^d \sum_{{\vec{n}} \ge 0 }
C^j_{\vec{n}} \tilde B^j_{\vec{n}},
\quad \textnormal{where} \quad
C^j_{\vec{n}} 
= \langle v, \tilde B^j_{\vec{n}}\rangle_{H^p}
= R(  \tilde B^j_{\vec{n}}).
\end{split} \end{equation}
Combining the theoretical result with the Fourier expression of $v$, we can see that the optimal perturbation achieving the optimal response is

\begin{equation} \begin{split} \label{e:Y}
\eta_{\mathrm{opt}} = \frac{v}{||v||_{H^p}}
= \sum_{j=1}^d \sum_{{\vec{n}} \ge 0 }
\frac{R( \tilde B^j_{\vec{n}})}{||v||_{H^p}} \tilde B^j_{\vec{n}},
\quad \textnormal{where} \quad
||v||_{H^p} = \left( \sum_{j=1}^d \sum_{{\vec{n}} \ge 0 }
(C^j_{\vec{n}})^2 \right)^{\frac12}.
\end{split} \end{equation}

\subsection{Two-dimensional Kuramoto system}
\label{s:2d}
\hfill\vspace{0.1in}

This subsection illustrates our algorithm on a two-dimensional Kuramoto system
on $\T^2=[0,2\pi]^2$. Kuramoto systems describe coupled oscillators and, depending on the parameters, can exhibit a wide range of behaviors, including synchronization and other emergent collective phenomena. They also appear in several applications and have connections with other high-dimensional dynamical models, including some recent viewpoints related to machine learning and
transformer-type architectures. In the present two-dimensional setting, the
drift, the response coefficients, and the optimal perturbation can all be
visualized explicitly, making this a convenient first benchmark for the method.

The base dynamical system is
\begin{equation}
\begin{split} \label{e:kuramoto}
  d x^i = F^i(x)\, dt + dW^i,
  \quad \textnormal{where} \quad 
  F^i(x):= \omega^i + \frac 1d \sum_{j=1}^d \sin (x^j-x^i),
  \quad \textnormal{and} \quad 
  \omega = [1,3].
\end{split}
\end{equation}
Here, the superscript $i\in \{1,\ldots, d\}$ labels the coordinates. 
The observable function is 
\[ \begin{split}
  \phi(x) := \frac 1d \sum_{i=1}^d \sin (x^i)
\end{split} \]
Figure \ref{f:orbit} shows a typical orbit of this nonlinear system of time length $1000$.

\begin{figure}[ht]
    \centering
    \includegraphics[width=0.45\linewidth]{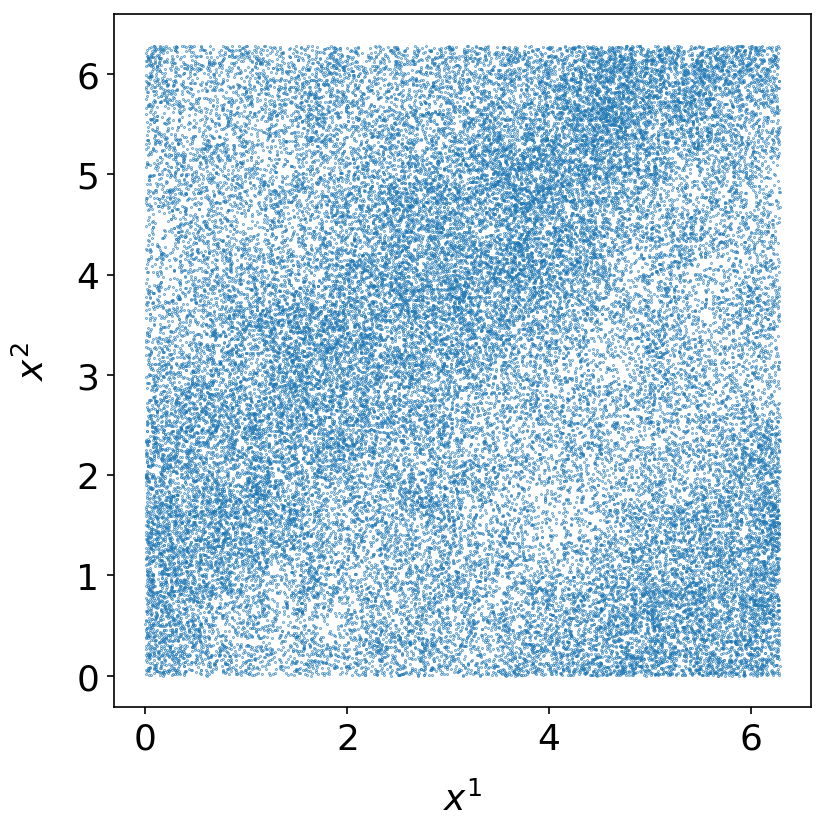}
    \caption{Typical orbit of the two-dimensional Kuramoto system \eqref{e:kuramoto}, shown as a scatter plot in the phase space $\T^2$.}
    \label{f:orbit}
\end{figure}

We define the perturbed dynamics as 
\begin{equation}
\label{e:f ga}
   F^\gamma:= F + \gamma \eta ,
   \quad \textnormal{where} \quad
   \gamma \ge 0.
\end{equation}
The space of feasible infinitesimal perturbations is
\[
P:=\{|| \eta ||_{H^5} \le 1\}\subset 
\cH:
= H^5(\T^2;\mathbb{R}^2) \subset C^{3}(\T^2 ; \mathbb{R}^2).
\]
Moreover, the Sobolev embedding theorem gives
$
\cH :=H^5(\T^2;\mathbb{R}^2)\hookrightarrow L^{\infty}(\T^2;\mathbb{R}^2)
$
continuously.

In the numerical computation, we consider only $N=11$ basis in each direction of $x$, that is, $ 0\le n_i \le N-1=10$ for each $n_i$ in $\vec{n}$.
The contour plot of the norms of all non-normalized basis are in Figure \ref{f:B2contour}.
With this, we can normalize all the basis elements to get $\tilde B^j_{\vec{n}}$.

\begin{figure}[ht]
    \centering
    \includegraphics[width=0.45\linewidth]{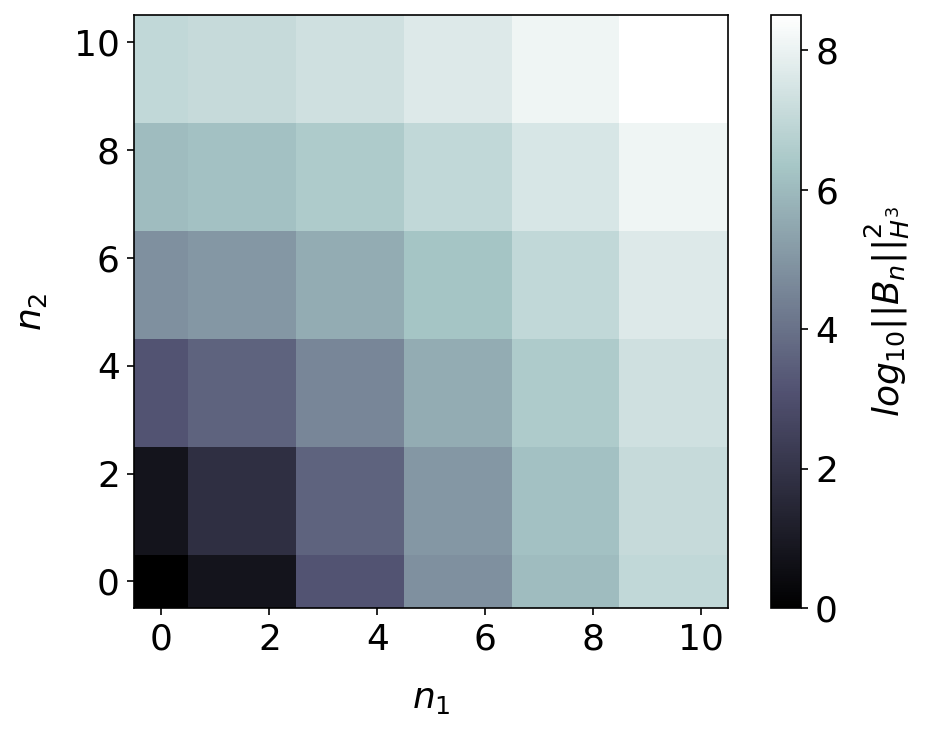}
    \caption{Squared $H^5$ norms of the non-normalized Fourier basis functions $B^j_{\vec n}$ for the two-dimensional Kuramoto example in \eqref{e:kuramoto}, shown in logarithmic scale. }
    \label{f:B2contour}
\end{figure}

We use the ergodic kernel-differentiation algorithm to compute the linear response of orthonormal basis, $R(\tilde B^j_{\vec{n}})$, which is also the coefficient for $v$, the $H^p$ representative of the linear response operator $R$.
The results are plotted in Figure \ref{f:RBcontour}.

In the ergodic kernel-differentiation algorithm, we set the total time $T=10^5$ and the decorrelation time $W = 4$.
The code is at \url{https://github.com/niangxiu/optrKD}.
On a 3GHz 8-core computer, the computation time for computing the linear responses of all $2\times11\times11=242$ basis is 164 seconds.

\begin{figure}[ht]
    \centering
    \includegraphics[width=0.45\linewidth]{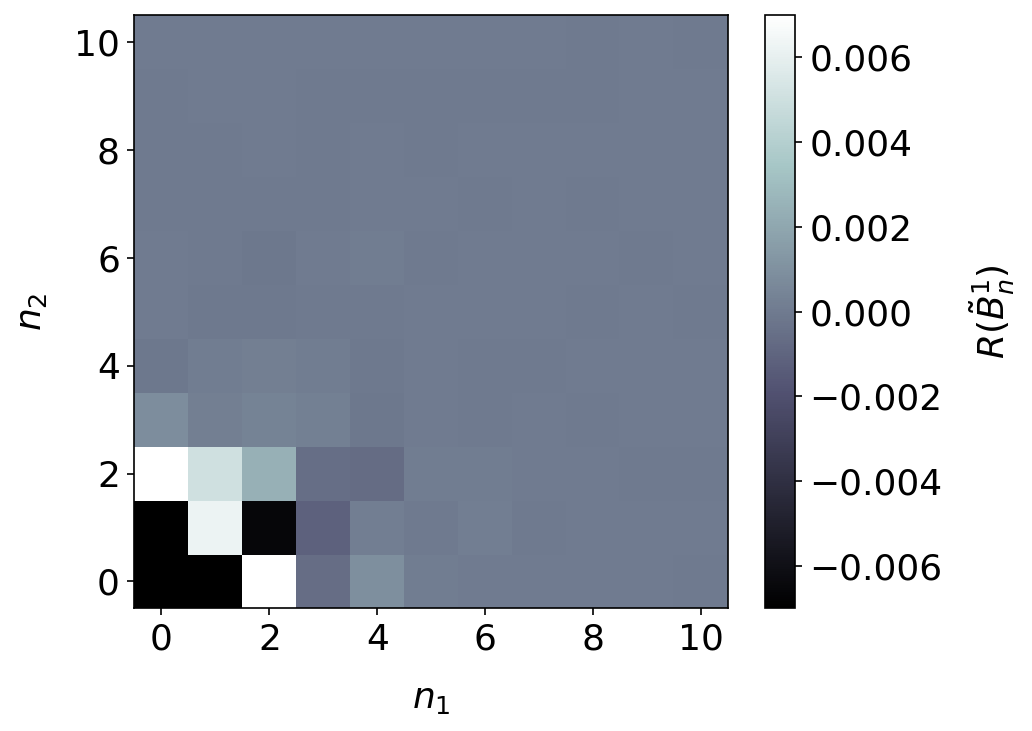}
    \includegraphics[width=0.45\linewidth]{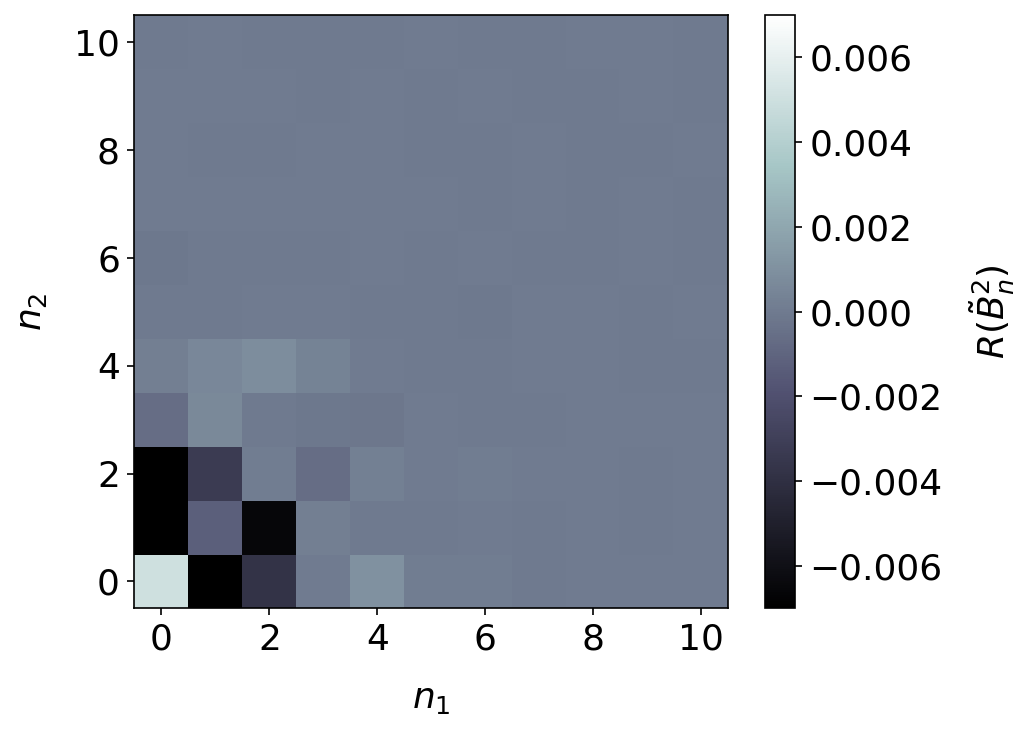}
    \caption{Computed response coefficients $C^j_{\vec n}=R(\tilde B^j_{\vec n})$ for the normalized Fourier basis functions in the two-dimensional Kuramoto example. Left: coefficients for the first component ($j=1$). Right: coefficients for the second component ($j=2$).}
    \label{f:RBcontour}
\end{figure}

With $R(\tilde B^j_{\vec{n}})$, we can compute $v$, then compute $\eta_{\mathrm{opt}}$, the optimal perturbation achieving the optimal response.
The vector field plot of $\eta_{\mathrm{opt}}$ and $F$ is in Figure \ref{f:Yvec}.
The first six Fourier coefficients of $\eta_{\mathrm{opt}}$, for the basis functions $\tilde B^1_{(0,0)},\ldots, \tilde B^1_{(0,5)}$, rounded off to two-digits, are:
\begin{equation*}
\begin{split}
[c_0,...,c_{5}]& =
[\texttt{-0.05, -0.09,  0.19, 0.01, -0.,    0. }].
\end{split}
\end{equation*}
The largest (in absolute value) Fourier coefficient is \texttt{-0.87}, attained for $\tilde B^1_{(1,0)}$, which generates the largest linear response among all the basis.

\begin{figure}[ht]
    \centering
\includegraphics[width=0.45\linewidth]{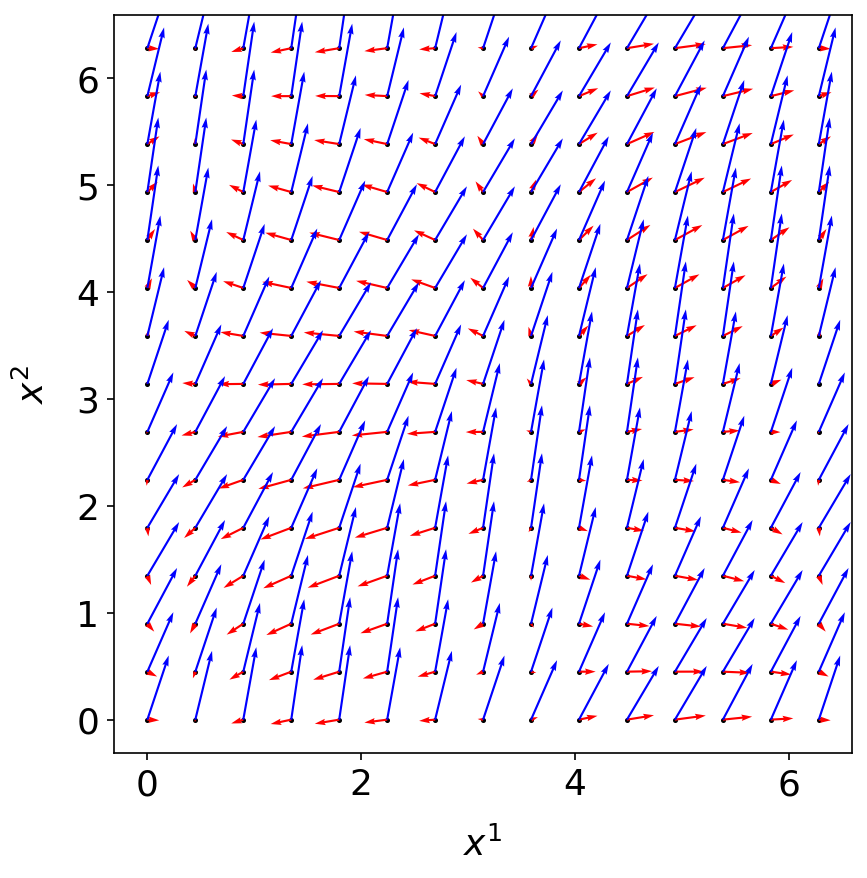}
    \caption{Vector field of the drift $F$ and of the optimal perturbation $\eta_{\mathrm{opt}}$ for the two-dimensional Kuramoto example. Blue arrows represent $\frac15 F$, and red arrows represent $\frac12 \eta_{\mathrm{opt}}$. The figure shows the spatial structure of the perturbation that maximizes the linear response under the $H^5$ constraint.}
    \label{f:Yvec}
\end{figure}

Figure \ref{f:prm obj} verifies that the optimal perturbation $\eta_{\mathrm{opt}}$ computed by our method indeed generates the optimal response.
The optimal response has an absolute value larger than the response of any individual basis function.
In particular, it is larger than the linear response of $\tilde B^1_{(1,0)}$, which, by Figure \ref{f:RBcontour}, generates the largest linear response among all the basis.
We also plot the linear response of $\tilde B^2_{(11,11)}$.
This figure also verifies that the linear response we compute is correct in reflecting the trend between $\gamma$ and $\mu^\gamma(\phi)$.
To do this, we compute $\mu^\gamma(\phi)$ for different values of $\gamma$ around 0.
The parameter-observable relation matches the linear response we computed.

\begin{figure}[ht]
    \centering
    \includegraphics[width=0.45\linewidth]{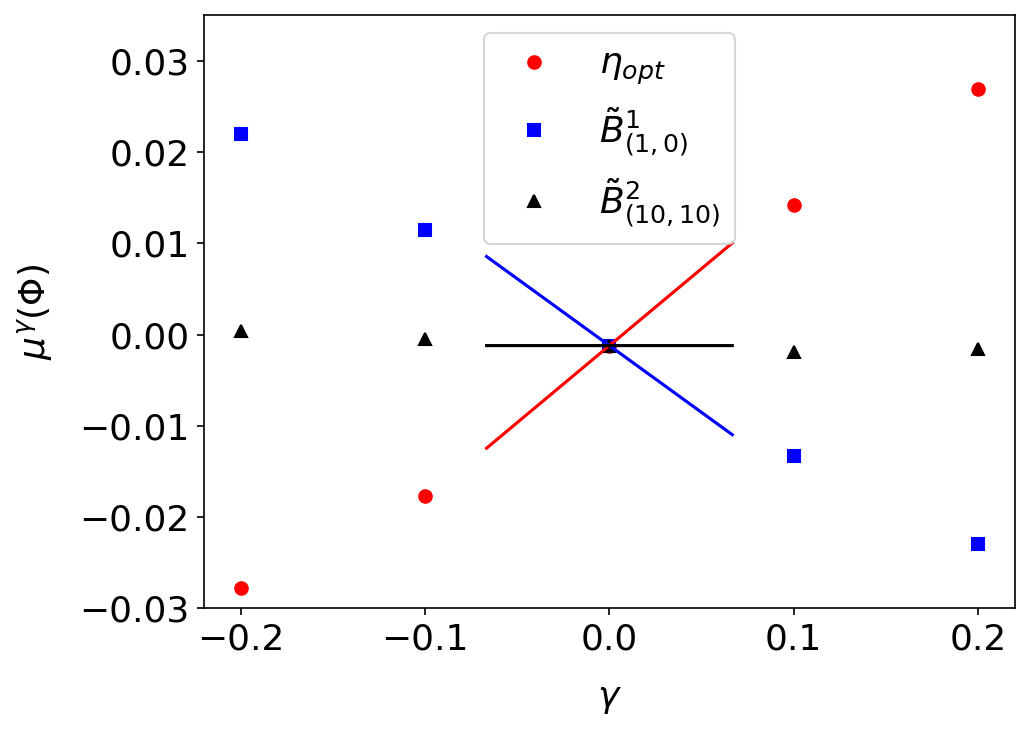}
    \caption{Computed linear response in the two-dimensional Kuramoto example. The short line segments show the linear responses at $\gamma=0$, while the dots show the values of $\mu^\gamma(\phi)$ for the perturbed systems $F+\gamma\eta_{\mathrm{opt}}$ (red circles), $F+\gamma \tilde B^1_{(1,0)}$ (blue squares), and $F+\gamma \tilde B^2_{(10,10)}$ (black triangles). All linear responses and $\mu^\gamma(\phi)$ are computed on orbits of same length.
    }
    \label{f:prm obj}
\end{figure}

\subsection{Twenty-dimensional Kuramoto system with a reduced perturbation space}
\hfill\vspace{0.1in}
\label{s:lowH}

This example shows that the method remains effective even in a high-dimensional phase space, provided that the admissible perturbations are restricted to a smaller space. We consider the same Kuramoto-type dynamics as in \eqref{e:kuramoto}, now on $\T^{20}$, with the same observable $\phi$. The perturbations are chosen to depend only on the first coordinate and to act only on the first two components. In this way, the dynamical system is twenty-dimensional, while the perturbation space $\cH$ is one-dimensional, so the number of basis elements remains small and computationally tractable.

We choose the frequencies in \eqref{e:kuramoto} to be 
\begin{equation*}\begin{split}
  [\omega^i ] = [1, 1.2, 1.4, \ldots, 4.8].
\end{split}\end{equation*}
The perturbed dynamics is
$$
F^\gamma := F + \gamma \eta, \qquad \eta \in  \cH,
$$
where
\[
\cH:=
\{\eta(x):\eta^1(x)=\eta^2(x)=g(x^1),\,
\eta^3(x)=\cdots=\eta^{20}(x)=0,\,
g\in H^4(\T)\}.
\]
We endow $\cH$ with the norm $\| \eta\|_{\cH} = \| g \|_{H^4 (\T)},$ so that $\cH$ is isomorphic to the Sobolev space $H^4(\T)$. The feasible set $P$ is the unit ball of $\cH$.

Since every $\eta \in \cH$ has the form $ \eta (x) = (g(x^1), g(x^1), 0, \dots, 0)$, we have
$$
\| \eta \|_{L^\infty (\T^{20})} \leq \sqrt{2} \| g \|_{L^\infty (\T)} \leq C \| g \|_{H^4(\T)} = C \| \eta \|_{\cH},
$$
where we have used the embedding $H^4(\T) \hookrightarrow L^\infty (\T).$ Hence $\cH \hookrightarrow L^\infty (\T^{20}, \mathbb{R}^{20})$, and thus $R_{\cH} (\eta) = R(\eta)$ for all $\eta \in \cH.$

The unnormalized Fourier basis for $\cH$ is
\[
B_n(x) = [b_n(x^1),b_n(x^1),0,\ldots,0] \,,
\]
where $b_n$ is the trigonometric function given in \eqref{e:basis}.
In numerical computations, we truncate the Fourier basis to
\[
0\le n\le N = 21.
\]
In the ergodic kernel-differentiation algorithm, we set $T=5\times 10^5$ and $W=6$.
The wall-clock time to compute the gradients is about 278 seconds.

\begin{figure}[ht]
\centering
\includegraphics[width=0.45\linewidth]{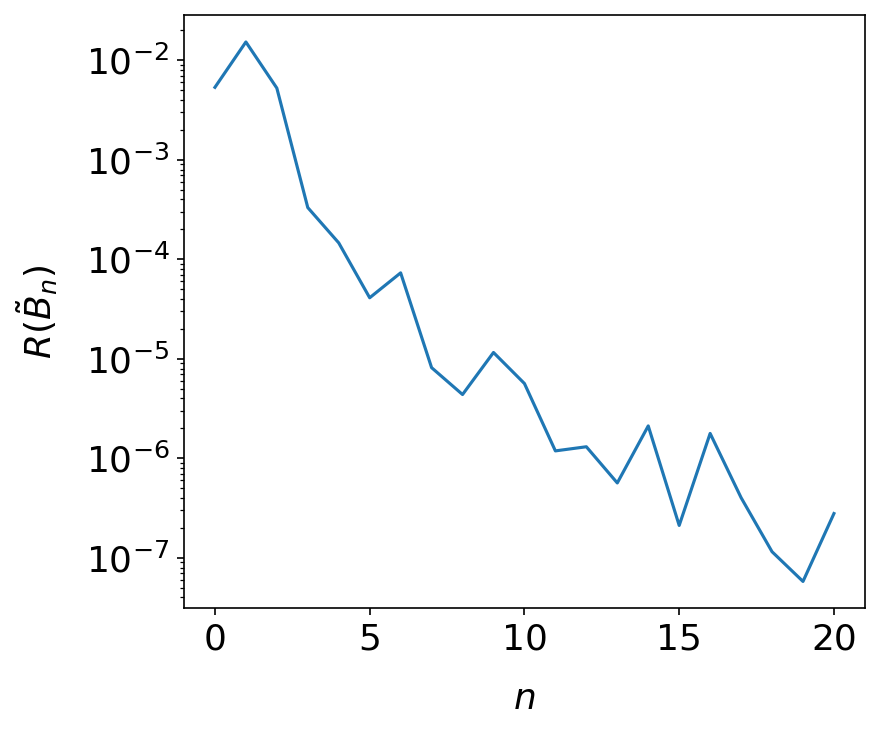}
\caption{Computed linear responses $R(\tilde B_n)$ for the normalized basis functions of the reduced perturbation space $\cH$ in the twenty-dimensional example. The rapid decay as $n$ increases indicates that the truncation error from the finite Fourier basis is small.}
\label{f:RB lowH}
\end{figure}

The linear response for each basis function is plotted in Figure \ref{f:RB lowH}, which decays very fast as $n$ increases, so the error caused by using a finite basis is also very small.
The first several Fourier coefficients of $\eta_{\mathrm{opt}}$ with respect to the basis $\tilde B_n$ are:
\begin{equation*}
\begin{split}
[c_0,...,c_{5}]& =
[\texttt{0.31, -0.9, 0.31, -0.02, 0.01, -0. }].
\end{split}
\end{equation*}

\begin{figure}[ht]
    \centering
\includegraphics[width=0.45\linewidth]{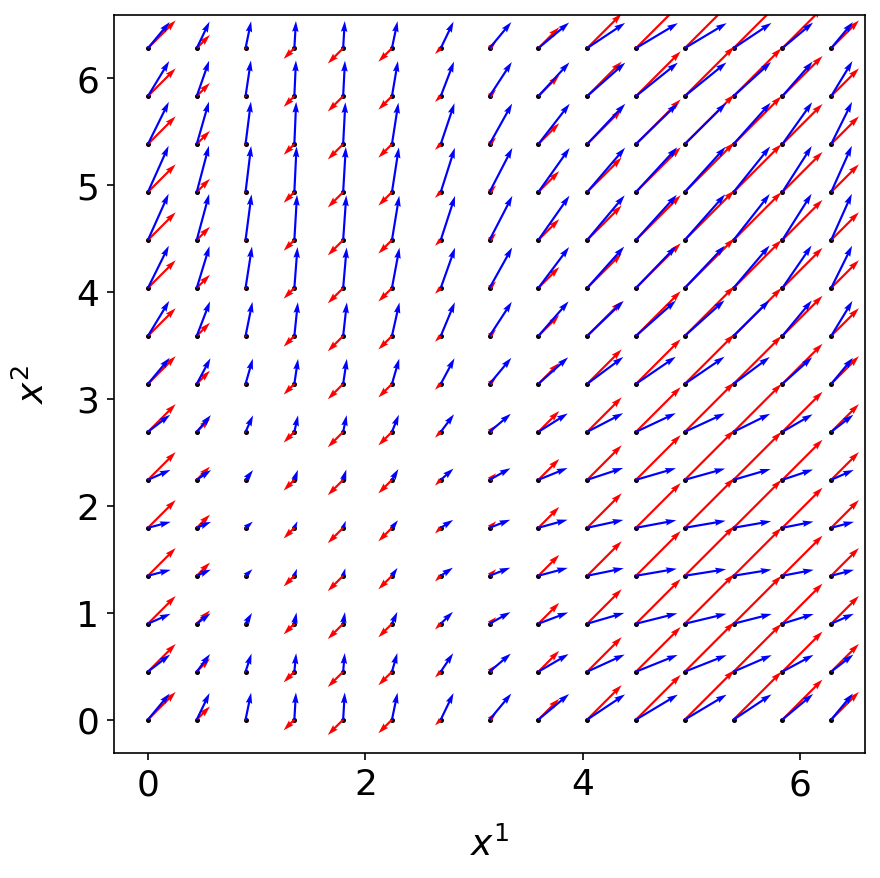}
\includegraphics[width=0.45\linewidth]{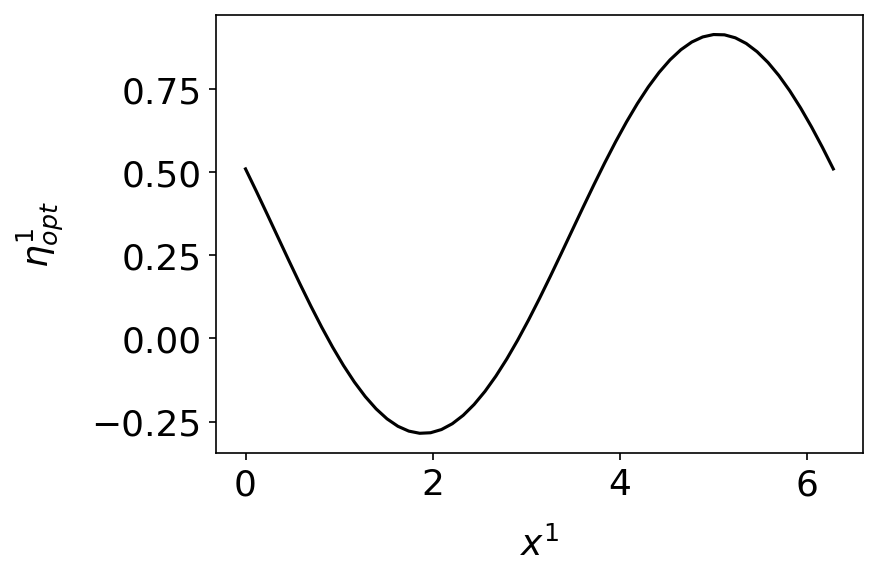}
    \caption{
      Optimal perturbation for the twenty-dimensional example with reduced perturbation space $\cH$. Left: vector field of the drift $\frac15 F$ (blue arrows) and of the optimal perturbation $\frac12 \eta_{\mathrm{opt}}$ (red arrows), projected onto the first two coordinates; all remaining components vanish. Right: graph of the first component $\eta^1_{\mathrm{opt}}(x^1)$.
    }
    \label{f:YveclowH}
\end{figure}

\begin{figure}
    \centering
\includegraphics[width=0.5\linewidth]{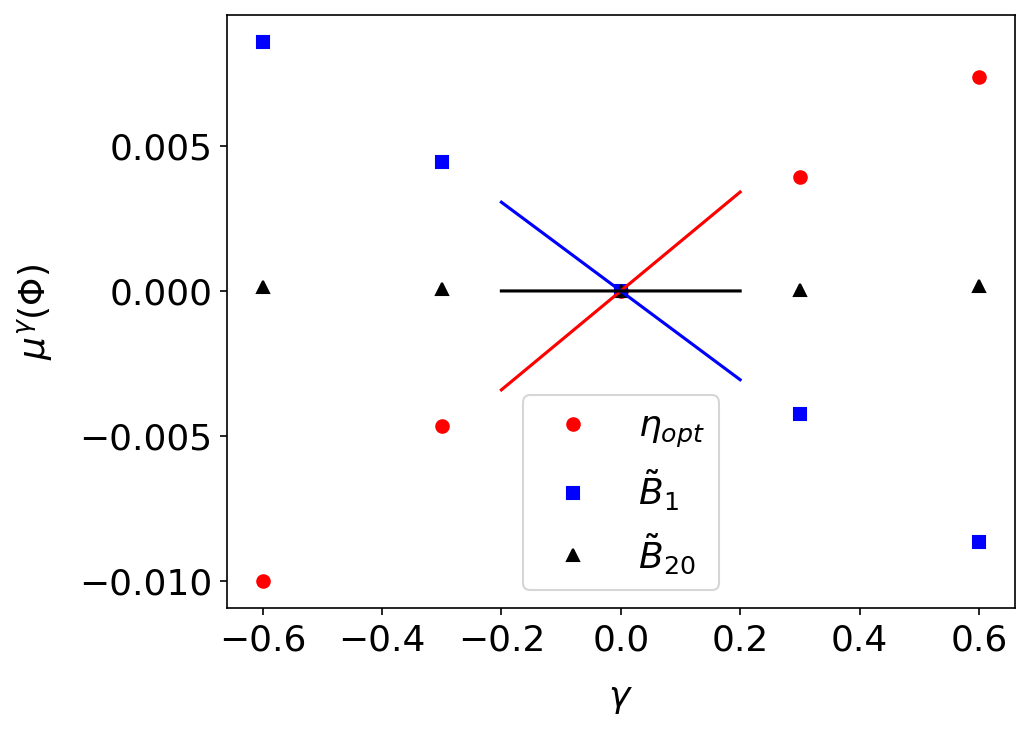}
    \caption{
    Linear response in the twenty-dimensional example. As in Figure \ref{f:prm obj}, the short line segments represent the linear responses at $\gamma=0$, while the dots show the averaged observable $\mu^\gamma(\phi)$ for several perturbations. The figure confirms that the optimal perturbation produces a larger linear response than any individual basis element.}
    \label{f:prm obj lowH}
\end{figure}

The vector field plot of the optimal perturbation $\eta_{\mathrm{opt}}$ is in Figure \ref{f:YveclowH}.
We can see that $\eta^1_{\mathrm{opt}}=\eta^2_{\mathrm{opt}}$ and it depends only on $x^1$; this is due to the selection of $\cH$.
Then we plot $\mu^\gamma(\phi)$ versus $\gamma$ and the linear responses at $\gamma=0$ computed for different perturbations in Figure \ref{f:prm obj lowH}.
Here, the largest linear response among all the basis elements is generated by $\tilde B_{1}$.
This shows that the optimal perturbation indeed generates a linear response larger than all elements in the basis.
This also verifies that the linear response we compute are correct.
Due to the relatively high dimension of $\T^{20}$, it would be difficult to compute the optimal perturbation for this example using a finite-element reduction of the transfer operator associated to the system.

\subsection{The three-dimensional Lorenz system} \label{sec:3d_lorenz}
\hfill\vspace{0.1in}

This final example tests the method on a three-dimensional system with richer dynamics. We consider a modified version of the classical Lorenz 63 system, adapted to the torus by multiplying the Lorenz vector field by a cutoff near the boundary. This preserves the characteristic Lorenz structure in the interior while making the drift continuous on $\T^3$, so that the optimal perturbation can be computed and visualised in the three-dimensional setting.

More specifically, let
\begin{equation*}\begin{split}
\ \T ^3 = [-r_{box}, r_{box}]^2 \times [0,2r_{box}],
\quad \textnormal{where} \quad 
r_{box} = 40.
\end{split}\end{equation*}
The base dynamical system is 
\begin{equation} \begin{split} \label{e:lorenz}
  d x^i = F^i(x)\,dt + 5 \,dW^i,
  \quad \textnormal{where} \quad 
  F(x) := b(|x-c_{enter}|_{\infty}) F'(x)
  \\
  F'(x) = \begin{bmatrix} 
    10 (x^2-x^1) \\
    x^1 (28-x^3) - x^2\\
    x^1 x^2- 8 x^3/3
  \end{bmatrix},
  \quad \textnormal{} 
  c_{enter} = \begin{bmatrix} 
    0 \\
    0 \\
    r_{box}
  \end{bmatrix},
  \\
  b(\rho) = 
  \begin{cases}
    1, \quad\text{if}\quad 0\le \rho \le r_{box} - r_{bezel} ,
    \\
    (r_{box} - \rho)/ r_{bezel}, \quad\text{if} \quad r_{box} - r_{bezel} < \rho \le r_{box}
  \end{cases},
  \quad \textnormal{} \quad 
  r_{bezel} = 2.
\end{split} \end{equation}
Here, the superscript $i\in \{1,2,3\}$ labels the coordinates;  $|\cdot|_{\infty}$ is the $l^\infty$ norm which takes the maximum of all components.
$b(\rho)$ is a piecewise linear cutoff function which decays from 1 to 0 in a thin bezel near the boundary; $r_{bezel}$ is the width of the bezel.
$F'$ is the original Lorenz 63 system, and its product with the cutoff function $F=bF'$ is Lipschitz continuous on $\T^3$.
In Figure \ref{f:Xopt 3d}, we can see that $F$ decays to zero at the boundary, so it is continuous throughout the boundary and hence continuous throughout $\T^3$.
The observable function is 
\[ \begin{split}
  \phi(x) := \sum_{i=1}^3 \sin \left(
  \frac{2\pi x^i}{2 r_{box}}
  \right).
\end{split} \]

The feasible set is the unit ball in $\cH = H^5(\T^3; \mathbb{R}^3)$, and 
$
\cH \hookrightarrow L^{\infty}(\T^3;\mathbb{R}^3)$ by the Sobolev embedding. Due to rapidly growing (with respect to $d$ and $N$) computational cost, we reduce the number of basis in each direction to
\[  
N = 9.
\]
So we now have $3\times 9^3=2187$ functions in the Fourier basis.
The wall-clock time to run the algorithm on a single core CPU is about 27 minutes.

\begin{figure}[ht]
    \centering
\includegraphics[width=0.49\linewidth]{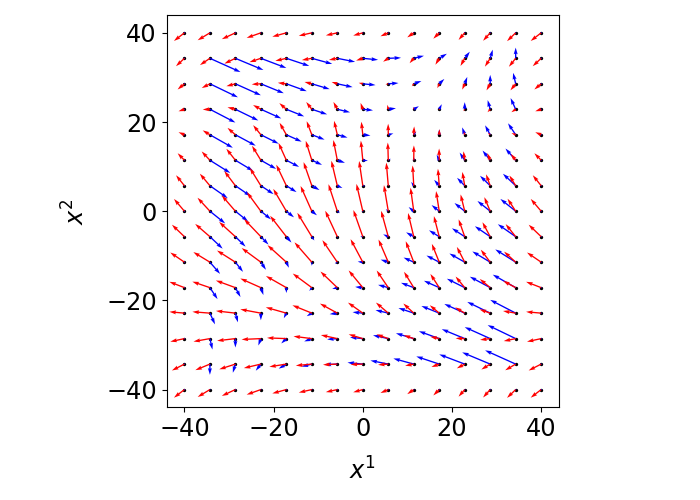}
\includegraphics[width=0.49\linewidth]{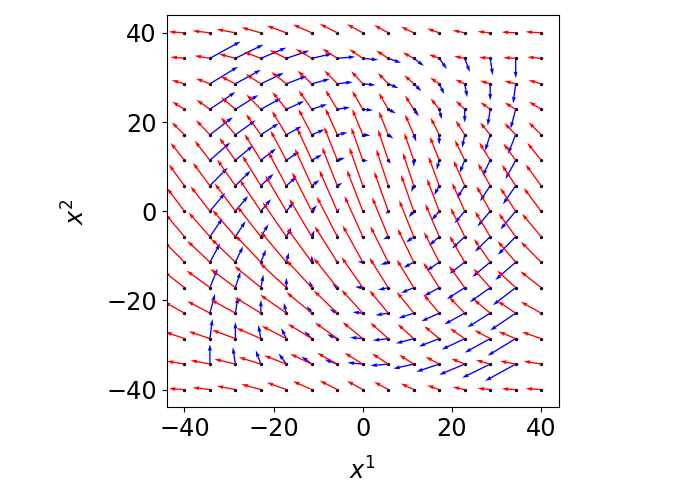}
\caption{Two-dimensional slices of the drift and of the optimal perturbation for the three-dimensional Lorenz example. Blue arrows represent $0.01[F^1,F^2]$, and red arrows represent $10[\eta^1_{\mathrm{opt}},\eta^2_{\mathrm{opt}}]$. Left: slice at $x^3=20$. Right: slice at $x^3=40$. 
    }
    \label{f:Xopt 3d}
\end{figure}

\begin{figure}[ht] \centering
  \includegraphics[width=0.5\textwidth]{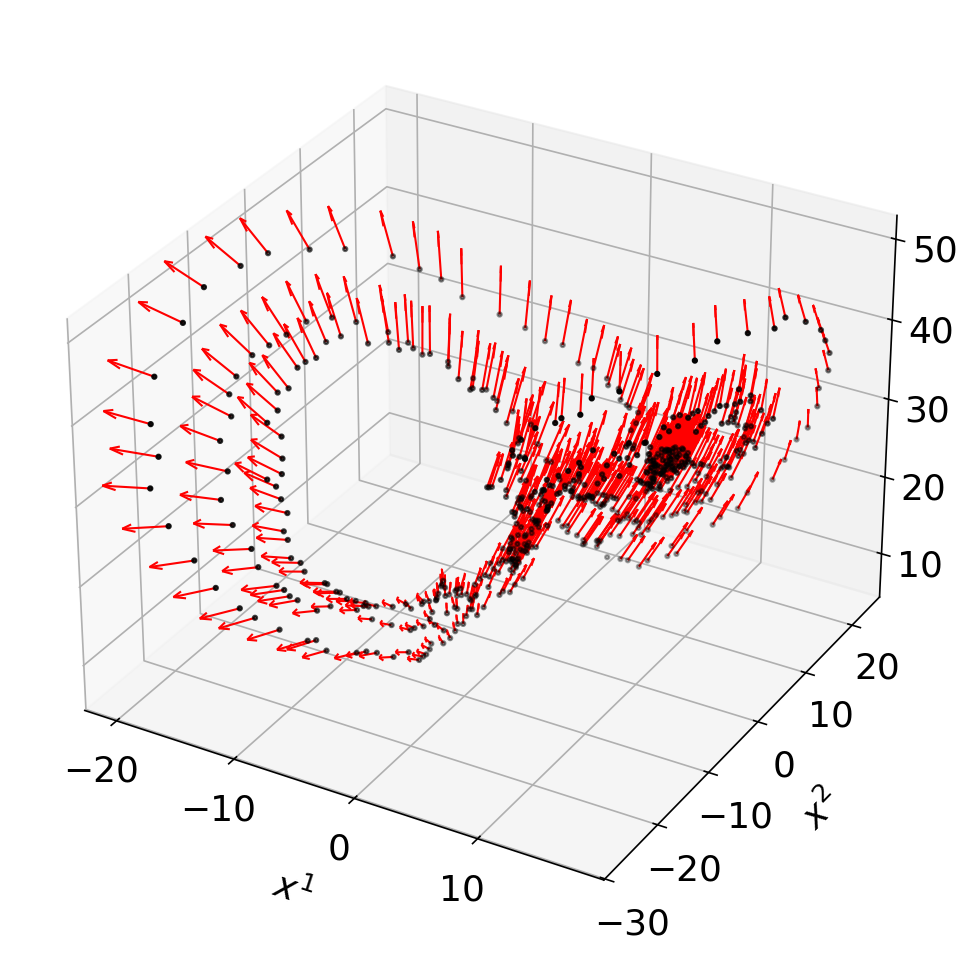}
  \caption{Three-dimensional visualization of the optimal perturbation for the Lorenz example. Red arrows represent $5\eta_{\mathrm{opt}}$, plotted along a typical orbit of the system shown in black.}
  \label{f:attractorY}
\end{figure}

We plot two-dimensional slices of $\eta_{\mathrm{opt}}$ in Figure \ref{f:Xopt 3d} and the three-dimensional plot on a typical orbit in Figure \ref{f:attractorY}.
The first six Fourier coefficients of $\eta _{\mathrm{opt}}$, relative to the basis functions $\tilde B^1_{(0,0,0)}, \ldots, \tilde B^1_{(0,0,5)}$, are:
\begin{equation*}
\begin{split}
[c_0,...,c_{5}]& =
[\texttt{-0.26, 0.09, 0.17, 0.,   -0.03, 0.}].
\end{split}
\end{equation*}
The largest (in absolute value) Fourier coefficient is attained at $\tilde B^2_{(0,2,0)}$, which generates the largest linear response among all the basis.

\begin{figure}
    \centering
    \includegraphics[width=0.5\linewidth]{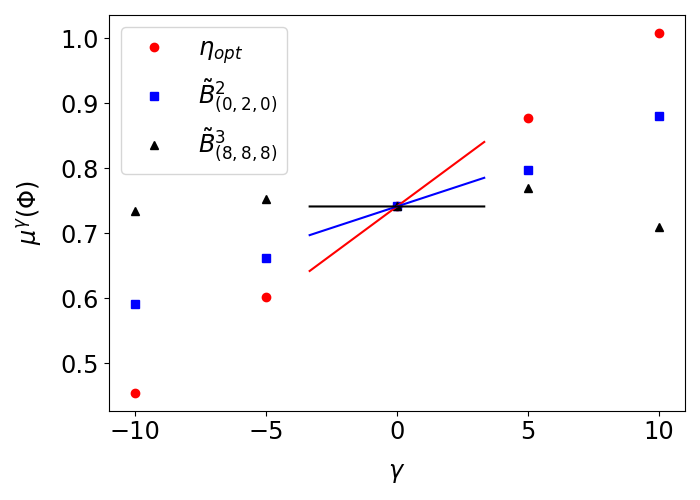}
    \caption{Computed linear response for the three-dimensional Lorenz example. The figure compares the linear responses at $\gamma=0$ with the values of the averaged observable $\mu^\gamma(\phi)$ for several perturbations, showing that the optimal perturbation produces the largest response and that the linear approximation is accurate near $\gamma=0$.
    }
    \label{f:prm obj 3d}
\end{figure}

Then we plot $\mu^\gamma(\phi)$ versus $\gamma$ and the linear responses at $\gamma=0$ computed for different perturbations in Figure \ref{f:prm obj 3d}.
This shows that the optimal perturbation indeed generates a linear response larger than all elements in the basis.
It also verifies that the linear response that we compute is correct.

\appendix

\section{Semigroup and density estimates for the Fokker-Planck equation}
\label{appendix}

This appendix collects the analytic estimates for the Fokker-Planck equation used throughout the paper. In the following, we always assume that
$$
b,\eta \colon \T^d\to\mathbb{R}^d
\qquad b, \eta \in L^\infty(\T^d;\mathbb{R}^d),
$$
and for $\delta\in[0,1)$ we set
$
g^\delta:=b+\delta\eta$. As discussed in Section \ref{sec: SDEs}, we consider the SDE on $\T^d$
$$
dX_t^{\delta,x} =g^{\delta}(X_t^{\delta,x}) dt + dW_t, \qquad X_0^{\delta,x} = x,
$$
and we introduce the Markov semigroup
\begin{equation}
    (P_t^{\delta} \phi) (x):= \mathbb{E}[ \phi (X_t^{\delta,x})], \qquad t \geq 0, \ \phi \in \mathcal{B}_b(\T^d).
    \label{eq: Markov operator appendix}
\end{equation}
Its dual operator on finite signed measure $\mathcal{M}(\T^d)$ is defined by duality
\begin{equation}
    \int_{\T^d} \phi \, d(T_t^{\delta} \mu) := \int_{\T^d} P_t^{\delta} \phi \, d \mu , \qquad \phi \in C(\T^d),\ \mu \in \mathcal{M}(\T^d).
    \label{eq: FPE semigroup appendix}
\end{equation}
If $\mu=f\,dx$ with $f\in L^1(\T^d)$, we write again $T_t^\delta f$ for the density of $T_t^\delta\mu$ with respect to Lebesgue measure. If $\mu=\delta_x$, this density is precisely the transition density $p^\delta(t,x,\cdot)$, and our assumption $b, \eta \in L^\infty (\T^d; \mathbb{R}^d)$, together with the non-degeneracy of the noise, are sufficient for the density $p^{\delta}$ to exist, see Remark \ref{rem:stationary_measure} for more details.

Our main object is the forward Fokker-Planck evolution associated with $T_t^\delta$. For $L^2$ initial data, this evolution is understood in the standard variational sense. For rough initial data, such as Dirac masses, it is understood through the forward Fokker-Planck semigroup $T_t^{\delta}$ defined in \eqref{eq: FPE semigroup appendix}. The estimates proved below are uniform in $\delta\in[0,1)$.

\subsection{Semigroup preliminaries, fractional Gr\"onwall estimates, and heat-kernel bounds}

We start by recalling some standard terminology from semigroup theory, for which we refer to the monographs \cite{Pazy,LunardiAnalytic}.

Given a Banach space $X$, a family $(T(t))_{t\ge0}\subset\mathcal L(X)$ is a \emph{strongly continuous semigroup} (a \emph{$C_0$-semigroup}) if
$$
T(0)=\mathrm{Id},\qquad T(t+s)=T(t)T(s)\quad \forall\ t,s\geq0,\qquad 
\lim_{t\downarrow0}\|T(t)f-f\|_X=0 \quad  \forall\,f\in X.
$$
Its \emph{generator} $(\mathcal{A},D(\mathcal A))$ is defined by
$$
D(\mathcal A):=\left\{f\in X:\ \lim_{t\downarrow0}\frac{T(t)f-f}{t}\ \text{exists in }X\right\},\qquad 
\mathcal Af:=\lim_{t\downarrow0}\frac{T(t)f-f}{t},
$$
and we write $T(t)=e^{t\mathcal A}$ when $\mathcal A$ generates $(T(t))_{t\geq0}$. The heat semigroup
$S(t):=e^{\frac12 t\Delta_{\T^d}}$ is a $C_0$-semigroup on each $L^p(\T^d)$, for any $1 \leq p < \infty,$ and satisfies the smoothing bounds
recalled below, see Lemma \ref{lemma: smoothing estimates heat semigroup}. Given $F\in L^1_{\mathrm{loc}}((0,T);X)$, a function $u\in C([0,T];X)$ is a \emph{mild solution} of
$$\partial_t u=Au+F, \qquad u(0)=u_0\in X,
$$
if it satisfies Duhamel's formula
$$
u(t)=T(t)u_0+\int_0^t T(t-s)F(s)\,ds,\qquad t\in[0,T].
$$
In this appendix we use mild formulations both for $L^2$ initial data and, after approximation, for rough initial data such as Dirac masses. In the present $L^\infty$ setting, the drift term is treated in divergence form and the divergence is understood distributionally. More precisely, if $G\in L^q(\T^d;\mathbb{R}^d)$ for some $1\le q <\infty$, then $\mathrm{div} (G)$ is understood as an element of $\mathcal{D}'(\T^d) = (C^\infty (\T^d))'$, and for every $t>0$ one has
$$
S(t)\mathrm{div}  (G)=\mathrm{div} (S(t)G)=\nabla S(t)\cdot G
\qquad\text{in }\mathcal{D}'(\T^d).
$$
This identity is the basic tool that allows us to combine the heat semigroup with bounded drifts without requiring derivatives of $b$ or $\eta$.

By It\^{o} formula, it is possible to show that if $f \in L^1(\T^d) $, then the density $u^{\delta} (t) = T_t^\delta f$ is a distributional solution of
$$
\partial_t u^{\delta} = \frac{1}{2}\Delta u^{\delta} - \mathrm{div } (u^\delta g^{\delta}), \qquad u^{\delta}|_{t = 0} = f.
$$
More precisely, for every $\phi \in C^\infty (\T^d)$ and every $t>0$
\begin{equation}
\int_{\T^d} u^\delta(t,y)  \phi(y)\,dy =\int_{\T^d} f(y) \phi(y)\,dy + \int_0^t\int_{\T^d}
u^\delta(s,y) (\frac12 \Delta \phi(y)+g^\delta(y)\cdot\nabla \phi(y) ) \,dy \,ds.
\label{eq:weak-formulation-FPE-appendix}
\end{equation}
If in addition $f \in L^2(\T^d),$ this weak solution coincides with the variational solution from the Lions-Magens theory, see \cite[Chapter 3]{Lions_MagensI}.

Moreover, the corresponding mild representation is
\begin{equation}
\label{eq:mild-representation-appendix}
u^{\delta}(t)= S(t)f-\int_0^t S(t-s)\mathrm{div} (g^\delta u^{\delta}(s) ) \, ds
= S(t)f - \int_0^t \nabla S (t-s)\cdot   (g^\delta u^{\delta}(s))  \,  ds,
\end{equation}
where the previous identities are understood in the sense of distributions since $g^{\delta} \in L^\infty$ only, see the next Lemma \ref{lemma: heat semigroup divergence}. Whenever $g^\delta u(s)\in L^q(\T^d;\mathbb{R}^d)$, the heat-kernel bounds recalled in Lemma \ref{lemma: heat semigroup divergence} imply that the last term belongs to the relevant $L^p$ space, and \eqref{eq:mild-representation-appendix} can then be used as the starting point for the estimates proved in the rest of the appendix.

We then recall the following generalisation of Gr\"onwall's lemma, see \cite[Theorem 3.3.1]{Amann1995}.
\begin{lemma}
    Let $\alpha, \beta \in [0,1), $ $\varepsilon>0$. There exists $c = c(\alpha, \beta, \varepsilon)>0$ such that the following is true. If $A,B>0$ and $u \colon [0,T) \to [0, \infty)$ satisfies $(t\mapsto t^\beta u(t)) \in L^\infty_{loc}([0,T))$ and
    $$
    u(t) \leq A t^{-\beta} + B \int_0^t (t-s)^{-\alpha} u(s) ds, \qquad \text{ for a.a. }t \in (0,T),
    $$
    then
    $$
    u(t) \leq A t^{-\beta} (1 + c B t^{1-\alpha} e^{(1+\varepsilon)\mu t}), \qquad \text{ for a.a. } t \in (0,T),
    $$
    where $\mu := \{ \Gamma (1-\alpha) B\}^{1/(1-\alpha)}$
    \label{lemma: fractional gronwall lemma}.
\end{lemma}
Further, we recall the smoothing estimates for the heat semigroup, see \cite[Proposition 3.5.7]{Cazenave1998} and \cite[Lemma 3]{Rothe1984}. These estimates will often lead to integral inequalities with a singular kernel of the form
$$
v(t)\le A t^{-\beta}+B\int_0^t (t-s)^{-\alpha}v(s)\,ds,
$$
where $\alpha=\tfrac12$ typically comes from the gradient bound for the heat semigroup. Lemma \ref{lemma: fractional gronwall lemma} is the tool that turns these convolution-type inequalities into explicit time bounds.

We next recall the classical $L^q$-$L^p$ smoothing estimates for the heat semigroup on $\T^d$. These bounds will be repeatedly combined with Lemma \ref{lemma: heat semigroup divergence}, which is the main tool allowing us to treat bounded drifts in divergence form.
\begin{lemma}
Let $1\le q\le p\le\infty$ and $S(t):=e^{\frac12 t\Delta_{\mathbb{T}^d}}$ be the heat semigroup on $\mathbb{T}^d$.
Then for every $T>0$ there exists $C=C(d,p,q,T)>0$ such that for all $t\in(0,T]$,
\begin{enumerate}
    \item[(i)] $\|S(t)f\|_{L^p(\T^d)} \le C\, t^{-\frac d2(\frac1q-\frac1p)}\|f\|_{L^q(\T^d)}$
    \item[(ii)] $\|\nabla S(t)f\|_{L^p(\T^d)} \le C\, t^{-\frac12-\frac d2(\frac1q-\frac1p)}\|f\|_{L^q(\T^d)}$
\end{enumerate}
\label{lemma: smoothing estimates heat semigroup}
\end{lemma}
\begin{proof}
    Assume for simplicity $D =[-\tfrac{1}{2}, \tfrac{1}{2})^d$ to be the representative domain for the torus. The heat semigroup has the form $S(t) f= K_t * f$, where $$
    K_t (x):= \sum_{k \in \mathbb{Z}^d} G_t(x+k), \qquad x \in \mathbb{T}^d,
    $$
    denotes the heat kernel on the torus, obtained by the periodization of the heat kernel $G_t$ on $\mathbb{R}^d$ given by
    $$
    G_t (z) := (2 \pi t)^{-d/2} e^{- \frac{\vert z \vert ^2}{2t}}, \qquad z \in \mathbb{R}^d.
    $$

    \textit{Step $1$: $L^1$ and $L^\infty $ bound for $K_t$.} 
    The fact that
    \begin{equation}
    \| K_t \|_{L^1 (\T^d)} = \int_{\T^d} K_t (x) \, dx = 1
    \label{eq: L1 bound heat kernel}
    \end{equation}
    is a consequence of the fact that $K_t$ is the periodization of the heat kernel on the whole space $\mathbb{R}^d$, and can be checked by an explicit computation. For the $L^\infty $ bound, it is possible to check that $\vert x+k \vert \geq \tfrac{1}{2} \vert k \vert$ for any $k \in \mathbb{Z}^d$ and any $x \in D = [- \tfrac{1}{2}, \frac{1}{2})^d$. Thus
    \begin{equation}
    K_t(x) = (2\pi t)^{-d/2}\sum_{k \in \mathbb{Z}^d}  e^{-\frac{\vert x+k \vert^2 }{2t}} \leq (2\pi t)^{-d/2} \sum_{k \in \mathbb{Z}^d}e^{- \frac{\vert k \vert ^2}{4t}} = C(d,T) t^{-d/2}
    ,
    \label{eq: Linf bound heat kernel}
    \end{equation}
    uniformly for any $t \in (0,T].$
    
    \textit{Step $2$: $L^1$ and $L^\infty $ bound for $\nabla K_t$.}
    To bound the gradient, first differentiate the Gaussian kernel
    $$
    \nabla G_ t(x) = - \frac{x}{t} G_t (x) = - (2 \pi )^{-d/2} t^{-d/2 -1} x e^{- \frac{\vert x \vert ^2}{2t}}.
    $$
    Considering the absolute value, we obtain
    \begin{equation}
    \vert \nabla G_t (x) \vert  = ( 2 \pi )^{-d/2} t^{-d/2 -1} \vert x \vert e^{- \frac{ \vert x \vert^2}{2t}}.
    \label{eq: aux bound abs G_t}
    \end{equation}
    Making the substitution $r = \vert x \vert / \sqrt{t},$ we see that
    $$
    \vert x \vert  e^{- \frac{\vert x \vert ^2}{2t}} = \sqrt{t}  r  e ^{-\frac{r^2}{2}}.
    $$
    Since the function $ r \mapsto r e^{- \frac{r^2}{4}}$ is bounded on $[0,+\infty),$ we get
    $$
    r e^{- \frac{r^2}{2}} = ( r e^{- \frac{r^2}{4}}) e^{- \frac{r^2}{4}} \leq C e^{- \frac{r^2}{4}}.
    $$
    Hence, going back to the variable $x$, we have verified that
    $$
    \vert x \vert e^{- \frac{ \vert x \vert ^2}{2t}} \leq C \sqrt{t} e^{- \frac{ \vert x \vert ^2}{4t}}.
    $$
 Substituting this into \eqref{eq: aux bound abs G_t}, we obtain
    $$
    \vert \nabla G_ t(x) \vert \leq C_d t^{-d/2 -1} \sqrt{t}e^{- \frac{\vert x \vert ^2}{4t}} = C_d t^{-(d+1)/2}e^{- \frac{\vert x \vert ^2}{4t}}.
    $$
    Therefore, we obtain the $L^\infty $ bound for $\nabla K$ as follows
    \begin{equation}
\vert \nabla K_t (x) \vert \leq \sum_{k \in \mathbb{Z}^d} \vert  \nabla G_t (x+k) \vert \leq C_d t^{-(d+1)/2} \sum_{k \in \mathbb{Z}^d} e^{- \frac{\vert x+k \vert ^2}{4t}} \leq C (d,T) t^{- (d+1) /2}.
        \label{eq: Linfty bound nabla K}
    \end{equation}

    For the $L^1$ bound, using Fubini-Tonelli, we have
    $$
    \| \nabla K_t  \|_{L^1 (\T^d)} \leq \sum_{k \in \mathbb{Z}^d} \int_{[-\tfrac{1}{2}, \tfrac{1}{2})^d} \vert \nabla G_t (x+k) \vert \, dx = \int_{\mathbb{R}^d} \vert \nabla G_t (y) \vert \, dy.
    $$
    Making the substitution $ y = \sqrt{t}z$, we have
    $$
    \int_{\mathbb{R}^d} \vert \nabla G_t (y) \vert \, dy = t^{-1/2} \int_{\mathbb{R}^d} \vert z \vert (2\pi )^{-d/2} e^{- \frac{\vert z \vert ^2}{2}}\, dz \leq C_d t^{-1/2}
    $$
    Hence, combining the last two inequalities we have proved that
    \begin{equation}
        \| \nabla K_t \|_{L^1(\mathbb{T}^d)} \leq C t^{-1/2}
        \label{eq: L1 bound nabla K}
    \end{equation}
    
    \textit{Step $3$: $L^r$ bound for $K_t$ and $\nabla K_t$, for any $r \in [1,\infty]$.}
    Let $r \in [1,\infty]$. The $L^r$ bounds for $K_t$ and $\nabla K_t$ follow from interpolation. Indeed, using \eqref{eq: L1 bound heat kernel}, \eqref{eq: Linf bound heat kernel} and interpolation, we obtain
    \begin{equation}
        \| K_t \|_r \leq  \| K_t \|_{1} ^{1/r} \| K_t \|_{\infty }^{1- 1/r} \leq C t^{-\frac{d}{2}\left( 1- \frac{1}{r}\right)}.
        \label{eq: Lr bound K}
    \end{equation}
    On the other hand, \eqref{eq: L1 bound nabla K}, \eqref{eq: Linfty bound nabla K} and again interpolation inequality, yields
    \begin{equation}
    \| \nabla K_t \|_r \leq \| \nabla K_t \|_1^{1/r} \| \nabla K_t \|_\infty ^{1-1/r} \leq C t^{-\frac{1}{2}- \frac{d}{2}\left( 1-  \frac{1}{r}\right)}
        \label{eq: Lr bound nabla K}.
    \end{equation}
    
    \textit{Step $4$: proof of (i) and (ii).} Fix $r \in [1,\infty]$ be given by
    $$
    1+ \frac{1}{p} = \frac{1}{q}+ \frac{1}{r}.
    $$
    Since $S(t) f = K_t * f$, applying Young convolution inequality, we get
    $$
    \| S(t) f \|_{L^p (\T^d)} \leq \| K_t \|_{L^r (\T^d)} \| f \|_{L^q (\T^d)}.
    $$
    By \eqref{eq: Lr bound K}, we obtain
    $$
    \| K_t \|_{L^r (\T^d)} \leq C t^{- \frac{d}{2}\left( 1- \frac{1}{r} \right)} = C t^{- \frac{d}{2}\left( \frac{1}{q}- \frac{1}{p} \right)},
    $$
    which proves (i). The proof of (ii) follows in a similar way by using that $\nabla S(t) f = ( \nabla K_t) * f$, using Young inequality for convolution, and lastly applying the $L^r$ bound for $\nabla K_t$ from \eqref{eq: Lr bound nabla K}.
\end{proof}
The next lemma states the distributional identity that allows one to move the divergence onto the heat semigroup. This is the key tool used below to handle drift terms under the only assumption $b,\eta\in L^\infty$.
\begin{lemma}
\label{lemma: heat semigroup divergence}
Let $1\leq q\leq p\leq\infty$, let $T>0$, and let $G\in L^q(\T^d;\mathbb{R}^d).$ Then, for every $t\in(0,T]$, the distribution $S(t) \mathrm{div} (G)$ is well defined and satisfies
$$
S(t) \mathrm{div} (G)=\mathrm{div} (S(t) G)=\nabla S(t)\cdot G
\qquad\text{in }\mathcal D'(\T^d).
$$
Moreover, there exists a constant $C=C(d,p,q,T)>0$ such that
$$
\|S(t) \mathrm{div} (G)\|_{p} \leq C t^{-\frac{1}{2}-\frac{d}{2}(\frac{1}{q}-\frac{1}{p})} \| G \|_{q}, \qquad t\in(0,T].
$$
\end{lemma}
\begin{proof}
We recall that, for $G\in L^q(\T^d;\mathbb{R}^d)$, the divergence $\mathrm{div} (G)$ is understood in the sense of distributions
$$
\langle \mathrm{div} (G),\varphi\rangle := -\int_{\T^d} G(x)\cdot \nabla\varphi(x) \, dx, \qquad \varphi\in C^\infty(\T^d).
$$
Since $S(t)$ maps distributions into smooth functions, $S(t) \mathrm{div} (G)$ is well defined in $\mathcal D'(\T^d)$ by duality. Let $\varphi\in C^\infty(\T^d)$. Then
\begin{equation*}
\begin{split}
\langle S(t)\mathrm{div} (G),\varphi\rangle &=\langle \mathrm{div} (G),S(t) \varphi \rangle =-\int_{\T^d} G(x)\cdot \nabla(S(t)\varphi)(x)\,dx =-\int_{\T^d} G(x)\cdot S(t)(\nabla\varphi)(x)\,dx \\
&=-\int_{\T^d} S(t)G(x)\cdot \nabla\varphi(x)\,dx =\langle \mathrm{div} (S(t)G),\varphi\rangle.
\end{split}
\end{equation*}
This implies that $
S(t) \mathrm{div} ( G)=\mathrm{div} (S(t)G)$ in $\mathcal D'(\T^d).$ Since $S(t)G=K_t*G$ componentwise, where $K_t$ is the heat kernel on $\T^d$, we further have
$$
\mathrm{div} (S(t)G) =\sum_{i=1}^d \partial_i(K_t* G_i) =\sum_{i=1}^d (\partial_i K_t)* G_i =\nabla S(t)\cdot G \qquad\text{in }\mathcal D'(\T^d).
$$

It remains to prove the $L^p$ bound. By the previous identity and the triangle inequality,
$$
\|S(t)\mathrm{div}( G)\|_{p} = \|\nabla S(t)\cdot G\|_{p}
\leq \sum_{i=1}^d \|\partial_i S(t)G_i\|_{p}.
$$
Applying Lemma \ref{lemma: smoothing estimates heat semigroup}(ii) to each component $G_i$, we obtain
$$
\|\partial_i S(t)   G_i   \|_{p}
\leq C t^{-\frac{1}{2}-\frac{d}{2}(\frac{1}{q}-\frac{1}{p})} \| G_i \|_{q}.
$$
Summing over $i=1,\dots,d$ and using
$$
\sum_{i=1}^d \| G_i\|_{q}\leq C(d)\|G\|_{q},
$$
we conclude that
$$
\|S(t)\mathrm{div} (G)\|_{p} \leq C t^{-\frac{1}{2}-\frac {d}{2}(\frac{1}{q}-\frac{1}{p})} \| G \|_{q},
$$
for a constant $C=C(d,p,q,T)>0$.
\end{proof}

\subsection{Dirac initial data for the Fokker-Planck equation}

\subsubsection{Uniform semigroup estimates}

The following result establishes a contraction on $L^1(\mathbb{T}^d)$ for the semigroup associated to the forward Fokker-Planck equation. It can be made sharper by proving that the $L^1$-norm is preserved, a fact that however is not needed for our analysis.
\begin{lemma}
    Let $b, \eta  \colon \mathbb{T}^d \to \mathbb{R}^d$ and $T>0$. Assume $b,\eta \in L^{\infty}(\mathbb{T}^d; \mathbb{R}^d)$.  Consider, for $\delta \geq 0$, the Fokker-Planck semigroup $(T_t^{\delta})_t$ introduced in \eqref{eq: FPE semigroup appendix}. Then
    $$
    \| T_t^{\delta} f \|_1 \leq \|f \|_1, \qquad t >0.
    $$
    \label{lemma: contraction on L1}
\end{lemma}
\begin{proof}
Fix $\delta\in[0,1)$ and set $g^\delta:=b+\delta\eta$. Consider the SDE
$$
dX_t^{\delta,x}=g^\delta(X_t^{\delta,x})\,dt+dW_t,\qquad X_0^{\delta,x}=x.
$$
By the assumptions made in the paper, this equation admits a unique global strong solution and, for every $t>0$, the law of $X_t^{\delta,x}$ has a density with respect to Lebesgue measure on $\T^d$, see Remark \ref{rem:stationary_measure} for more details. We denote this density by
$$
p^\delta(t,x,y), \qquad t>0,    x,y\in\T^d.
$$
Since $p^\delta(t,x,\cdot)$ is the density of a probability measure, it satisfies
$$
p^\delta(t,x,y)\geq 0, \qquad \int_{\T^d} p^\delta (t,x,y) \, dy =1, \qquad t>0,\ x\in\T^d.
$$
For $f\in L^1(\T^d)$, the semigroup $(T_t^{\delta})_t$ admits the representation
$$
(T_t^\delta f)(y)=\int_{\T^d} p^\delta(t,x,y) f(x) \, dx, \qquad t>0, y\in\T^d.
$$
We now prove that $(T_t^\delta)_{t>0}$ is contractive on $L^1(\T^d)$. Using Tonelli's theorem, the positivity of $p^\delta$, and the fact that $\int_{\T^d}p^\delta(t,x,y)\,dy=1$, we obtain
\begin{align*}
\|T_t^\delta f\|_1
&= \int_{\T^d} \left| \int_{\T^d}  p^\delta(t,x,y)  f(x)\,  dx  \right| dy \leq \int_{\T^d}\int_{\T^d}  p^\delta(t,x,y)  |f(x)| \, dx \,dy \\
&= \int_{\T^d}|f(x)|\left(\int_{\T^d}   p^\delta(t,x,y) \ ,dy \right) dx = \int_{\T^d} | f (x )| \, dx = \|f\|_1.
\end{align*}
Therefore $ \|T_t^\delta f\|_1\leq \|f\|_1$. 
\end{proof}

The next result gives an $L^1 \to   L^\infty$ smoothing estimate for the forward Fokker-Planck semigroup, uniform in $\delta$. By approximation, it also yields a corresponding $L^\infty$ bound for the density.
\begin{proposition}
    Let $b, \eta \colon \mathbb{T}^d \to \mathbb{R}^d$ and $T>0$. Assume $b, \eta \in L^{\infty}(\mathbb{T}^d; \mathbb{R}^d)$ and consider the Fokker-Planck semigroup $(T_t^{\delta})_t$ introduced in \eqref{eq: FPE semigroup appendix}. Then there exists $ C'= C'( \| b \|_{\infty },  \| \eta  \|_{\infty }, d, T)>0 $ such that for any $t \in (0,T]$ and $\delta \in [0,1)$
    \begin{equation}
            \| T_t^{\delta}f \|_\infty \leq C' t^{-d/2} \| f \|_1
            \label{eq:Linfty_bound}
    \end{equation}
    for any $f \in L^1(\T^d)$. Further, for any $x \in \T^d$,
    $$
     \| T_t^{\delta} \delta _x \|_\infty \leq C' t^{-d/2}.
    $$
    \label{prop:Linfty_bound_density}
\end{proposition}
\begin{proof}
 Denote by $S(t)= e^{\frac{t}{2}\Delta}$ the heat semigroup and let $g^\delta := b+ \delta \eta .$ Further, set
 $
 K:= \|  b \|_\infty + \| \eta \|_\infty,$ in such a way that $\|g^{\delta}\|_{\infty} \leq K$ for any $\delta \in [0,1).$ Moreover, let $f_n \in C^\infty(\T^d)$ such that $\|f_n -f \|_1 \to 0$ as $n \to \infty$ and $\| f_n \|_1 \leq \| f \|_1.$ Recall that, since $f_n\in C^\infty(\T^d)\subset C(\T^d)$, the mild solution $u_n (t):= T_t^{\delta} f_n$ is given by
 $$
 u_n (t) = S(t) f_n - \int_0^t \nabla S (t-s) \cdot (g^{\delta} u_n (s)) \, ds,
 $$
 and it is well posed in $C([0,T]; C (\T^d))$ (as can be checked, for instance, by a fixed point argument).

\textit{Step 1: $ \|  T_t^{\delta}f_n \|_\infty \leq C' t^{-d/2} \| f \|_1$, with $C'$ independent of $n$.} Fix $ n \in \mathbb{N}$ and $\delta \in [0,1).$ Let $u_n(t) :=T_t^{\delta}f_n$ as above. Then
    $$
    u_n(t)= S(t) f_n - \int_0^t \nabla S(t-s) \cdot (g^\delta u_n(s)) \, ds,
    $$
    where the previous identity is understood in the sense of distributions (see Lemma \ref{lemma: heat semigroup divergence}), since $u_n(s) \in L^1$ thanks to Lemma \ref{lemma: contraction on L1}. Considering the $L^\infty $-norm in the previous equation, applying the triangle inequality, Lemma \ref{lemma: smoothing estimates heat semigroup}(i), Lemma \ref{lemma: heat semigroup divergence}, the fact that $\|f_n \|_1 \leq \| f \|_1,$ and H\"older inequality,
    $$
    \| u_n(t) \|_\infty \leq \| S(t) f_n \|_\infty + \int_0^t \| \nabla S (t-s) \cdot (g^\delta u_n(s)) \|_\infty \, ds \leq c_d t^{-d/2} \| f \|_1+ I_1(t) + I_2(t),
    $$
    with
    $$
    I_1(t) := \int_0^{t/2} \| \nabla S(t-s) \cdot (g^{\delta} u_n(s)) \|_\infty \, ds, \qquad I_2(t):= \int_{t/2}^t \| \nabla S (t-s) \cdot (g^{\delta} u_n(s)) \|_\infty \, ds.
    $$
    For the term $I_1(t)$, by Lemma \ref{lemma: heat semigroup divergence}, H\"older inequality, the $L^1$ contraction from Lemma \ref{lemma: contraction on L1}, and the fact that $\|f_n \|_1 \leq \| f \|_1,$ we have
    \begin{equation}
        \begin{split}
            I_1 (t) &\leq c_d \int_0^{t/2} (t-s)^{-1/2 - d/2} \| g^{\delta} u_n(s) \|_1 \, ds \leq c_d K \int_0^{t/2} (t-s)^{-1/2-d/2} \| u_n(s) \|_1\, ds \\
            & \leq c_d K \| f_n \|_1 \int_0^{t/2} (t-s)^{-1/2-d/2}\, ds = C K \| f_n \|_1 t^{1/2-d/2} \leq C K \| f \|_1 t^{1/2-d/2}  , 
        \end{split}
        \label{eq: estimate term I_1 lemma Linfty_bound}
    \end{equation}
    where $C= C(d)>0$. For the term $I_2(t)$, using Lemma \ref{lemma: heat semigroup divergence}, we have
    \begin{equation}
 I_2(t) \leq C \int_{t/2}^t (t-s)^{-1/2} \| g^{\delta} u_n(s) \|_\infty \, ds \leq C K \int_{t/2}^t (t-s)^{-1/2} \| u_n(s) \|_\infty \, ds.
        \label{eq: estimate term I_2 lemma Linfty_bound}
    \end{equation}
    Consider now 
    $$
    M_n(t) = \sup_{0 < r \leq t} r^{d/2} \| u_n (r) \|_\infty, \qquad t \in (0,T]. 
    $$
    and note that $M_n (t)<\infty$ by the reasoning before the start of Step $1$. If $s \in [t/2, t]$, then
    $$
    \| u_n(s) \|_\infty \leq s^{-d/2} M_n(t) \leq 2^{d/2} t^{-d/2} M_n(t).
    $$
    Replacing this into \eqref{eq: estimate term I_2 lemma Linfty_bound}, we obtain
    $$
    I_2 (t) \leq C K 2^{d/2} t^{-d/2} M_n(t) \int_{t/2}^t (t-s) ^{-1/2}ds \leq C K t^{1/2 -d/2} M_n(t).
    $$
     Combining the previous estimates for $I_1(t)$ and $I_2(t)$, we obtain
    \begin{equation}
    \| u_n (t) \|_\infty \leq c_d t^{-d/2} \| f \|_1 + CK \| f\|_1 t^{1/2-d/2} +CK t^{1/2-d/2} M_n(t).
        \label{eq: un infty estimate integral}
    \end{equation}
    Multiplying \eqref{eq: un infty estimate integral} by $t^{d/2}$ and considering the supremum for $0\leq r \leq t$, we have
    \begin{equation}
    M_n(t) \leq C \| f\|_1 + C K \| f \|_1 t^{1/2} + C K t^{1/2} M_n(t).
        \label{eq: inequality for M(t)}
    \end{equation}
    Choose now $t_0 \in (0,T]$ small enough such that $C K t_0^{1/2}< 1/2$. Then, for any $t \in (0,t_0]$, from \eqref{eq: inequality for M(t)} it follows
    $$
    M_n(t) \leq C' \| f \|_1, \qquad t \in (0,t_0],
    $$
    for a new constant $C' = C'(\|b\|_\infty, \| \eta \|_\infty, d,T)>0.$ This, by definition of $M_n$, implies
    \begin{equation}
        \| u_n(t) \|_\infty \leq C' t^{-d/2} \| f \|_1, \qquad t \in (0,t_0].
        \label{eq: Linfty bound up to t0}
    \end{equation}
    It remains to deal with the case $t \in (t_0,T].$ So, fix $t \in (t_0,T].$ Using the semigroup property, it holds
    $$
    T_t^{\delta} f_n = T_{t_0 /2}^\delta  T_{t- t_0/2}^\delta    f_n.
    $$
    Considering the norm, using \eqref{eq: Linfty bound up to t0} and the $L^1$ contraction from Lemma \ref{lemma: contraction on L1}, we have
    $$
    \| T_t^{\delta}  f_n \|_\infty \leq C' t_0^{-d/2} \| T_{t- t_0/2}^{\delta}   f_n \|_1 \leq C' t_0^{-d/2} \| f_n \|_1 \leq C' t_0^{-d/2} \| f \|_1.
    $$
    Since $t \leq T$, we have $t_0^{-d/2} \leq (\tfrac{T}{t_0})^{d/2} t^{-d/2},$ and thus we obtain
    $$
    \| T_{t}^{\delta} f_n \|_\infty \leq C' t^{-d/2} \| f \|_1, \qquad t \in (0,T],
    $$
    for a new constant $C' = C'(\| b\|_\infty, \| \eta \|_\infty,d, T)>0.$ Together with \eqref{eq: Linfty bound up to t0}, this concludes Step $1.$

 \textit{Step 2: $ \|  T_{t}^{\delta} f \|_\infty \leq C' t^{-d/2} \| f \|_1$. } Using the $L^1$ contraction from Lemma \ref{lemma: contraction on L1}, $$
 \| u_n(t) - u(t) \|_1 =\|  T_{t}^{\delta} (f_n -f) \|_1 \leq \| f_n - f \|_1  \to 0, \qquad t \in (0,T].
 $$
Hence $u_n(t) \to u(t)$ in $L^1(\T^d)$. Up to a subsequence, $u_n (t) \to u(t)$ a.e. in $\T^d$. But by Step $1$, $\|u_n \|_\infty \leq C' t^{-d/2} \| f \|_1$. Hence, also $\|u(t) \|_\infty \leq C' t^{-d/2} \| f \|_1$, thanks to the almost sure convergence and the fact that $C'$ is independent of $n$.

    \textit{Step $3$:  $\|   T_{t}^{\delta}\delta_x \|_\infty \leq C' t^{-d/2} $.} Let $\mathcal{M}(\mathbb{T}^d)$ be the space of finite signed Borel measure $\mu$ with the total variation norm $\| \cdot \|_{TV}$, and recall that 
    $$
    \mathcal{M}(\mathbb{T}^d)= (C(\mathbb{T}^d))^*,
    $$
    where $C(\mathbb{T}^d):= \{f \colon \mathbb{T}^d \to \mathbb{R} \; \mid  \; f \text{ is continuous }\}$. Further, we recall that $\mu_n \xrightharpoonup{*} \mu$ in $\mathcal{M}(\mathbb{T}^d)$ if
    $$
    \int_{\mathbb{T}^d} \phi (x) \mu_n   (dx) \xrightarrow{n \to \infty } \int_{\mathbb{T}^d} \phi (x) \mu (dx), \qquad \forall \phi \in C(\mathbb{T}^d).
    $$
    Given $x \in \mathbb{T}^d$, we also recall that it is possible to approximate $\delta_x$ by mollification as follows. Let $\rho \colon \mathbb{R}^d \to [0,\infty)$ such that $\rho \in C^\infty_c(\mathbb{R}^d)$, $\mathrm{Supp}(\rho) \subset B(0,1)$ and $\int_{\mathbb{R}^d} \rho(x) dx = 1$. For any $\varepsilon>0$, set
    $$
    \rho^\varepsilon (z):= \varepsilon^{-d} \rho \left(\frac{z}{\varepsilon}
\right)    .$$
Define
$$
f^\varepsilon (y) = \rho^\varepsilon (y-x).
$$
It holds
$$
f^\varepsilon \in C^\infty(\mathbb{T}^d), \qquad \| f^\varepsilon \|_1 =1, \qquad  f^\varepsilon \xrightharpoonup{*} \delta_x \quad \text{ in } \mathcal{M}(\mathbb{T}^d).
$$
Further, the adjoint of $T_t^\delta $ is the Markov operator $P_t^\delta$ introduced in \eqref{eq: Markov operator appendix}. In particular, it holds
\begin{equation}
    \langle T_t^{\delta } \delta_x , \phi \rangle = ( P_t^{\delta}  \phi )  (x).
    \label{eq: semigroup applied to dirac delta}
\end{equation}

Define $u^\varepsilon(t):=T_t^{\delta}f^\varepsilon$. By \eqref{eq:Linfty_bound},
\begin{equation}
\|u^\varepsilon(t)\|_\infty \le C' t^{-d/2}\qquad \forall\,t\in(0,T],\ \forall\,\varepsilon>0.
\label{eq:uniform_Linfty_mollified}
\end{equation}
Hence, for every $t\in(0,T]$ the family $\{u^\varepsilon(t)\}_{\varepsilon>0}$ is bounded in
$L^\infty(\T^d)$ by \eqref{eq:uniform_Linfty_mollified}. By Banach-Alaoglu there exist
$\varepsilon_n\downarrow0$ and $u(t)\in L^\infty(\T^d)$ such that
\begin{equation}
u^{\varepsilon_n}(t)\xrightharpoonup{*} u(t)\qquad \text{in }L^\infty(\T^d).
\label{eq:Linfty_weakstar_limit}
\end{equation}
In particular, for every $\phi\in L^1(\T^d)$,
\begin{equation}
\int_{\T^d}\phi(y)\,u^{\varepsilon_n}(t,y)\,dy \longrightarrow \int_{\T^d}\phi(y)\,u(t,y)\,dy.
\label{eq:pairing_limit_L1}
\end{equation}
On the other hand, for $\phi\in C(\T^d)$ we have, by duality,
$$
\int_{\T^d}\phi(y) u^{\varepsilon}(t,y)\,dy
=\langle f^\varepsilon, P_t^{\delta}\phi\rangle.
$$
Since $f^\varepsilon \xrightharpoonup{*}\delta_x$ in $\mathcal{M}(\T^d)$ and
$ P_t^{\delta}\phi\in C(\T^d)$, letting $\varepsilon\to0$ yields
$$
\lim_{\varepsilon\to0}\int_{\T^d}\phi(y)\,u^{\varepsilon}(t,y)\,dy
=(P_t^{\delta}\phi)(x) =\langle T_t^{\delta}\delta_x,\phi\rangle.
$$
Combining this
with \eqref{eq:pairing_limit_L1} (along the subsequence $\varepsilon_n$) gives
$$
\int_{\T^d}\phi(y) u(t,y)\,dy = \langle T_t^{\delta}\delta_x,\phi \rangle,
\qquad \forall\,\phi\in C(\T^d).
$$
Therefore the measure $T_t^{\delta}\delta_x$ is absolutely continuous with respect to
Lebesgue measure and has density $u(t)\in L^\infty(\T^d)$, i.e.
$$
T_t^{\delta}\delta_x = u(t,\cdot)\,dy.
$$
Finally, by weak-$*$ lower semicontinuity of the $L^\infty$-norm and
\eqref{eq:uniform_Linfty_mollified},
$$
\| T_t^{\delta}\delta_x\|_\infty = \|u(t)\|_\infty
\leq \liminf_{n\to\infty}\|u^{\varepsilon_n}(t)\|_\infty
\leq C' t^{-d/2}.
$$

\end{proof}

We now derive, by interpolation, a smoothing estimate for $\|T_t^{\delta} f \|_2$.
\begin{proposition}
\label{prop: properties semigroup}
    Let $b, \eta \colon \mathbb{T}^d \to \mathbb{R}^d$, $T>0$ and assume $b, \eta \in L^\infty(\T^d; \mathbb{R}^d).$ Consider the Fokker-Planck semigroup $(T_t^{\delta})_t$ introduced in \eqref{eq: FPE semigroup appendix}. Then there exists $c = c( \|b \|_{\infty}, \| \eta \|_{\infty}, d, T)>0$ such that for any $t \in (0,T]$ and $\delta \in [0,1)$
    $$
    \| T_t^{\delta}  f\|_2\le c\,t^{-d/4}\|f\|_1.
    $$
\end{proposition}
\begin{proof}
By H\"older inequality and using Proposition \ref{prop:Linfty_bound_density}
$$
\|  T_t^{\delta} f \|_2 \leq \|T_t^{\delta} f \|_\infty ^{1/2}\| T_t^{\delta} f \|_1 ^{1/2}  \leq (C' t^{-d/2} \| f \|_1)^{1/2} \| f \|_1^{1/2} = c t^{-d/4} \| f \|_1,
$$
for a constant $c = c( \|b \|_{\infty}, \| \eta \|_{\infty}, d, T)>0$.

\end{proof}
The estimate above is a bound on the transition density associated with $T_t^\delta$ starting from an initial condition $f \in L^1(\T^d).$ The next result states the previous estimate for the case of deterministic initial condition $x \in \mathbb{T}^d,$ which is equivalent in having $\delta_x$ as initial condition to the Fokker-Planck equation.
\begin{corollary}
\label{cor: semigroup estimates for dirac delta initial condition}
    Under the assumptions of Proposition \ref{prop: properties semigroup}, there exist
    $$
    c = c( \| b \|_{{\infty}},  \| \eta  \|_{{\infty}}, d, T)>0,
    $$
    such that for any $x \in \mathbb{T}^d$, $\delta \in [0,1)$ and $t \in (0,T)$
    it holds
       $$
        \| T_t^{\delta} \delta_x \|_2 \leq c t^{-d/4},
    $$   
\end{corollary}
\begin{proof}
    Let $\mathcal{M}(\mathbb{T}^d)$ be the space of finite signed Borel measure $\mu$ with the total variation norm $\| \cdot \|_{TV}$, and recall that 
    $$
    \mathcal{M}(\mathbb{T}^d)= (C(\mathbb{T}^d))^*,
    $$
    where $C(\mathbb{T}^d):= \{f \colon \mathbb{T}^d \to \mathbb{R} \; \mid  \; f \text{ is continuous }\}$. Further, we recall that $\mu_n \xrightharpoonup{*} \mu$ in $\mathcal{M}(\mathbb{T}^d)$ if
    $$
    \int_{\mathbb{T}^d} \phi (x) \mu_n   (dx) \xrightarrow{n \to \infty } \int_{\mathbb{T}^d} \phi (x) \mu (dx), \qquad \forall \phi \in C(\mathbb{T}^d).
    $$
    Given $x \in \mathbb{T}^d$, we also recall that it is possible to approximate $\delta_x$ by mollification as follows. Let $\rho \colon \mathbb{R}^d \to [0,\infty)$ such that $\rho \in C^\infty_c(\mathbb{R}^d)$, $\mathrm{Supp}(\rho) \subset B(0,1)$ and $\int_{\mathbb{R}^d} \rho(x) dx = 1$. For any $\varepsilon>0$, set
    $$
    \rho^\varepsilon (z):= \varepsilon^{-d} \rho \left(\frac{z}{\varepsilon}
\right)    .$$
Define
$
f^\varepsilon (y) = \rho^\varepsilon (y-x).$ It holds
$$
f^\varepsilon \in C^\infty(\mathbb{T}^d), \qquad \| f^\varepsilon \|_1 =1, \qquad  f^\varepsilon \xrightharpoonup{*} \delta_x \quad \text{ in } \mathcal{M}(\mathbb{T}^d).
$$
Further, the adjoint of $T_t^{\delta} $ is the Markov operator $P_t^\delta$ introduced in \eqref{eq: Markov operator appendix}. In particular, for $ \mu = \delta_x \in \mathcal{M}(\mathbb{T}^d)$, it holds
\begin{equation}
    \langle T_t^{\delta}\delta_x , \phi \rangle =(P_t^{\delta}  \phi )  (x).
    \label{eq: semigroup applied to dirac delta2}
\end{equation}

 Let $u^\varepsilon (t) :=  T_t^{\delta} f^\varepsilon$. By Proposition \ref{prop: properties semigroup}, it satisfies
\begin{equation}
    \| u^\varepsilon(t) \|_2 = \| T_t^{\delta} f^\varepsilon \|_2 \leq c t^{-d/4} \| f^\varepsilon \|_1 = c t^{-d/4},
    \label{eq: uniform bound of u epsilon in L2}
\end{equation}
where $c = c( \| b \|_{{\infty}},  \| \eta  \|_{{\infty}}, d, T)>0$ is the constant given by Proposition \ref{prop: properties semigroup}. By \eqref{eq: uniform bound of u epsilon in L2}, we get that there exists $u(t) \in L^2(\mathbb{T}^d)$ such that
$$
u^\varepsilon (t) \xrightharpoonup{} u(t) \quad \text{ in } L^2(\mathbb{T}^d).
$$

We now verify that $ u(t) = T_t^{\delta} \delta_x$. Indeed, using the adjoint operator,
$$
\int_{\mathbb{T}^d} u^\varepsilon (t,y) \phi (y) dy = \int_{\mathbb{T}^d} f^\varepsilon (y) (P_t^{\delta } \phi ) (y) dy \qquad \forall \phi \in C (\mathbb{T}^d).
$$
Taking in both sides the limit for $\varepsilon \to 0$, and using $f^\varepsilon \xrightharpoonup{ *} \delta_x$, we get
$$
\int_{\mathbb{T}^d} u(t,y) \phi (y) dy = ( P_t^{\delta} \phi )(x).
$$
Hence $u(t) = T_t^{\delta} \delta_x$ thanks to \eqref{eq: semigroup applied to dirac delta2}. In conclusion, since the norm is lower-semicontinuous with respect to the weak convergence in $L^2$, it yields
$$
\| T_t^{\delta} \delta_x \|_2 = \| u(t) \|_2 \leq \liminf_{ \varepsilon \to 0 } \| u^\varepsilon (t) \|_2 \leq C t^{-d/4},
$$
where in the last inequality we have used \eqref{eq: uniform bound of u epsilon in L2}.
\end{proof}

\subsubsection{Positivity of the kernel}
We now show that, for every fixed positive time, the density $p^{\delta}(t,x,y)$ associated to the SDE \eqref{eq: SDE in Td} is uniformly strictly positive for $(x,y) \in \T^d\times\T^d$, uniformly with respect to $\delta\in[0,1)$.
\begin{proposition}
    Let $T>0$ and $t  \in (0,T]$. Assume $b, \eta \in L^\infty (\T^d; \mathbb{R}^d).$ Then there exists $c = c(t, \| b \|_\infty, \| \eta \|_\infty, d)>0$ such that
    $$
    \inf_{(x,y) \in \T^d \times \T^d} p^\delta (t,x,y) \geq c >0, \qquad \delta \in[0,1).
    $$
    \label{prop: kernel positivity}
\end{proposition}
\begin{proof}
    Let $x \in \T^d$. As in the proof of Corollary \ref{cor: semigroup estimates for dirac delta initial condition}, let, for $\varepsilon>0$, $f^\varepsilon (z) := \varepsilon^{-d}\rho  (\frac{z-x}{\varepsilon} )$, with $\rho \colon \mathbb{R}^d \to [0,\infty)$ such that $\rho \in C^\infty_c (\mathbb{R}^d)$, $\mathrm{Supp}(\rho) \subset B(0,1)$ and $\int_{\mathbb{R}^d} \rho (z) dz = 1$. Then
    $$
    f^\varepsilon \in C^\infty (\T^d), \qquad \| f^\varepsilon \|_1 = 1, \qquad f^\varepsilon \xrightharpoonup{*} \delta_x \quad \text{ in } \mathcal{M}(\T^d).
    $$
    Let $p^{\delta,x,\varepsilon}(t,y) = p^{\delta, \varepsilon}(t,x,y)$ be the solution of the regularised FPE
    \begin{equation}
    \left \lbrace
    \begin{aligned}
                \partial_t p^{\delta, x, \varepsilon} &=  \frac{1}{2}\Delta_y  p^{\delta, x, \varepsilon} - \mathrm{div}_y(g^{\delta} p^{\delta,x,\varepsilon}), \qquad &\text{ in } (0,T) \times \T^d,\\
                p^{\delta, x, \varepsilon}|_{t = 0} & = f^{\varepsilon},  \qquad & \text{ in } (0,T) \times \T^d.
    \end{aligned}
    \right.
        \label{eq: mollified FPE}
    \end{equation}
    First, observe that, since $f^\varepsilon \in C^\infty$, $f^\varepsilon \geq 0$, and $b, \eta \in L^{\infty }(\T^d;\mathbb{R}^d)$, then there exists a unique solution (see \cite[Theorem 9.4.8]{Bogachev2022})  $p^{\delta,x, \varepsilon}$, it is locally H\"older continuous on $(0,T) \times \T^d$, see \cite[Corollary 6.4.3]{Bogachev2022}, and $p^{\delta,x, \varepsilon} \geq 0$ on $(0,T)\times \T^d$, see \cite[Theorem 6.6.4]{Bogachev2022}.
    
    Second, by Harnack inequality, see \cite[Corollary 7.42]{Lieberman}, there exists $\alpha = \alpha(t, \| b \|_\infty, \| \eta \|_\infty) >0$ such that
    \begin{equation}
    \sup_{ y \in \T^d} p^{\delta, x, \varepsilon} ( \frac{t}{2}, y) \leq \alpha  \inf_{y \in \T^d} p^{\delta, x, \varepsilon}(t, y), \qquad t \in (0,T).
        \label{eq: inequality consequence Harnack}
    \end{equation}
    Since $y \mapsto p^{\delta, x, \varepsilon}( t, y)$ is continuous thanks to the regularity of $p^{\delta, x, \varepsilon}$, then $$\sup_{ y \in \T^d} p^{\delta, x, \varepsilon} (\frac{t}{2}, y) = \max_{ y \in \T^d} p^{\delta, x, \varepsilon} (\frac{t}{2}, y)=: M_t.
    $$
    We already know that $M_t \geq 0$, since $p^{\delta,x, \varepsilon} \geq 0$. But further, it also holds that $M_t \geq 1$. Indeed, if by contradiction $M_t <1$, then
    $$\int_{\T^d} p^{\delta, x, \varepsilon}( \frac{t}{2},y) \, dy < M_t \int_{\T^d} 1 \, dy = M_t,
$$    
     which again is a contradiction of the preservation of the probability mass. Hence, we have obtained
     \begin{equation*}
     1 \leq M_t \leq  \alpha \inf_{y \in \T^d} p^{\delta, x, \varepsilon}(t,y), \qquad t \in (0,T),
     \end{equation*}
     which gives
     \begin{equation}
     0<\alpha^{-1} \leq \inf_{y \in \T^d} p^{\delta, x, \varepsilon}(t,y), \qquad t \in (0,T).
     \label{eq: positive inf density}
     \end{equation}
     It remains to pass to the limit as  $\varepsilon \to 0$. We do this weakly. Let  $ \varphi \in C(\T^d)$, $\varphi \geq 0$. From
\eqref{eq: positive inf density}, we have
$$
\int_{\T^d}\varphi(y)p^{\delta,x,\varepsilon}(t,y) \, dy \geq
c\int_{\T^d}\varphi(y)\,dy,
$$
where $c>0$ is independent of \(x\), \(\varepsilon\), and
$\delta\in[0,1)$.

On the other hand, by duality between the Fokker-Planck semigroup and the Kolmogorov semigroup
$$
\int_{\T^d}\varphi(y)p^{\delta,x,\varepsilon}(t,y) \, dy =
\int_{\T^d} P_t^\delta\varphi(z) f^\varepsilon(z) \, dz,
$$
where $ P_t^\delta\varphi(z):= \int_{\T^d}\varphi(y)p^\delta(t,z,y)\,dy .$
Since $f^\varepsilon\xrightharpoonup{*}\delta_x$ in
$\mathcal M(\T^d)$ and since $z\mapsto P_t^\delta\varphi(z)$ is
continuous\footnote{The continuity of  $z \mapsto P_t^\delta\varphi(z)$ follows from the
strong Feller property of uniformly elliptic diffusions with bounded drift.} for every $t>0$, we obtain
$$
\int_{\T^d} P_t^\delta\varphi(z) f^\varepsilon(z)\,dz \longrightarrow
P_t^\delta\varphi(x) = \int_{\T^d} \varphi(y) p^\delta(t,x,y) \, dy .
$$
Passing to the limit in the previous inequality gives
$$
\int_{\T^d}\varphi(y)p^\delta(t,x,y)\,dy \geq c\int_{\T^d}\varphi(y)\,dy
$$
for every nonnegative $\varphi\in C(\T^d)$. Therefore
$$
p^\delta(t,x,\cdot)\geq c \qquad\text{a.e. on }\T^d .
$$
Finally, since $y\mapsto p^\delta(t,x,y)$ is continuous\footnote{Indeed, by \cite[Theorem 6.6.4 and Corollary 6.4.3]{Bogachev2022}, $p^\delta$ is a locally H\"older continuous density on $(0,T) \times \T^d$.} for every
\(t>0\), the inequality holds for every \(y\in\T^d\). Hence
$$
\inf_{(x,y)\in\T^d\times\T^d}p^\delta(t,x,y) \geq c,
$$
with $c$ independent of $\delta\in[0,1)$.
\end{proof}

\subsubsection{Regularity for the Fokker-Planck equation}
\label{subsection: where the main prop on SDE is proved}
We now apply the semigroup bounds to the Fokker-Planck equation with singular initial datum $\delta_x$.
The first goal is to quantify the spatial regularity of $p^\delta(t,\cdot)$ for fixed $t>0$.
The second goal is perturbative: we compare $p^\delta$ to $p^0$ and obtain explicit rates in $\delta$, which later serve as a quantitative differentiability statement with respect to the parameter.

As in the previous section, let $b, \eta \colon \mathbb{T}^d \to \mathbb{R}^d$. For $\delta \in [0,1)$,  set $g^{\delta}:= b+ \delta \eta$. Given $x \in \mathbb{T}^d$, we consider the Fokker-Planck equation
\begin{equation}
\left \lbrace
\begin{aligned}
    \partial_t p^\delta &= \frac{1}{2}\Delta_y p^\delta - \mathrm{div}(p^\delta g^\delta), \quad &\text{ in } (0,T) \times \mathbb{T}^d,\\
    p^\delta|_{t = 0} & = \delta_x, &\text{ in }  \mathbb{T}^d
\end{aligned}
\right.
\label{eq: FPE p delta}
\end{equation}
Consider then $w^\delta := p^\delta - p^0$, which solves
\begin{equation*}
\left \lbrace
\begin{aligned}
        \partial_t w^\delta &= \frac{1}{2}\Delta w^\delta - \mathrm{div}(b w^\delta) - \delta \mathrm{div} (\eta p^\delta), \quad & \text{ in } (0,T)\times \mathbb{T}^d, \\
        w^\delta_{|t= 0} & = 0 , \quad &\text{ in } \mathbb{T}^d.
\end{aligned}
\right.
\end{equation*}
The difference $w^\delta$ isolates the effect of the perturbation $\delta\eta$ in the drift.
Its equation has zero initial condition and a forcing term proportional to $\delta$. This structure is crucial: it allows us to estimate $w^\delta$ by combining the heat-semigroup smoothing estimates with the integral formulation of the equation and the uniform bounds on $p^\delta$ proved in the previous section.

We first prove an $L^1$ estimate for $w^\delta$, which already exhibits the linear dependence on $\delta$ and will serve as an input for the stronger $L^2$ bound proved later.
\begin{lemma}
\label{lemma: L1 estimate w delta}
Assume $b,\eta \in L^\infty(\T^d;\mathbb{R}^d)$ and let $T>0$. Fix $x\in \T^d$, and for $\delta\in[0,1)$ let
$$
p^\delta(t):=p^\delta(t,x,\cdot), \qquad w^\delta(t):=p^\delta(t)-p^0(t), \qquad t>0.
$$
Then there exists $ C=C(\|b\|_\infty,\|\eta\|_\infty,d,T)>0
$
such that for every $t\in(0,T]$ and every $\delta\in[0,1)$ it holds
$$ \|w^\delta(t)\|_1\leq C\,\delta\, t^{1/2}.
$$
\end{lemma}
\begin{proof}
Set $ g^\delta := b+\delta\eta$ and $K:=\|b\|_\infty+\|\eta\|_\infty.$ Observe that
$$
\|g^\delta\|_\infty\leq K, \qquad \forall\,\delta\in[0,1).
$$
We denote by $S(t)=e^{\frac{t}{2}\Delta}$ the heat semigroup on $\T^d$.

Interpreting in distributional sense the following identity thanks to Lemma \ref{lemma: heat semigroup divergence}, we have
\begin{equation*}
    \begin{split}
         w^\delta(t) &= - \int_0^t S(t-s) \mathrm{div}(b w^\delta(s))\, ds - \delta \int_0^t S(t-s) \mathrm{div}(\eta p^\delta(s)) \, ds \\
         &= - \int_0^t \nabla S(t-s)\cdot (b w^\delta(s))\,ds
         - \delta \int_0^t \nabla S(t-s)\cdot (\eta p^\delta(s))\,ds .
    \end{split}
\end{equation*}
Applying the $L^1$-norm, using the triangle inequality and Lemma \ref{lemma: heat semigroup divergence}, gives
\begin{equation*}
    \begin{split}
        \|w^\delta(t)\|_1
        &\leq \int_0^t \|\nabla S(t-s)\cdot (b\,w^\delta(s))\|_1 \, ds
        + \delta \int_0^t \|\nabla S(t-s)\cdot (\eta p^\delta(s))\|_1 \,ds \\
        &\leq C\int_0^t (t-s)^{-1/2}\|b\,w^\delta(s)\|_1\,ds
        + C\delta \int_0^t (t-s)^{-1/2}\|\eta p^\delta(s)\|_1\,ds \\
        & \leq C \| b \|_\infty \int_0^t (t-s)^{-1/2} \| w^\delta(s) \|_1   \,ds
        + C \| \eta \|_\infty \delta \int_0^t (t-s)^{-1/2}\|p^\delta(s)\|_1 \, ds.
    \end{split}
\end{equation*}
Since $\|p^\delta(s)\|_1\leq 1$ by Lemma \ref{lemma: contraction on L1}, up to renaming $C$ from line to line we obtain
$$
\|w^\delta(t)\|_1  \leq C\int_0^t (t-s)^{-1/2}\|w^\delta(s)\|_1 \,  ds + C\delta \int_0^t (t-s)^{-1/2} \, ds \leq C\int_0^t (t-s)^{-1/2}\|w^\delta(s)\|_1 \, ds + C\delta t^{1/2}.
$$
Set  $ v^\delta(t):=t^{-1/2}\|w^\delta(t)\|_1.$ Then the previous inequality can be rewritten as
$$
t^{1/2}v^\delta(t) \leq C\int_0^t (t-s)^{-1/2}s^{1/2}v^\delta(s) \, ds + C\delta t^{1/2}.
$$
Dividing by $t^{1/2}$, we obtain
$$
v^\delta(t)\leq Ct^{-1/2}\int_0^t (t-s)^{-1/2}s^{1/2}v^\delta(s)\,ds + C\delta.
$$
Since $t^{-1/2}s^{1/2}\leq 1$ for $s\in[0,t]$, we have
$$
v^\delta(t)\leq C\int_0^t (t-s)^{-1/2}v^\delta(s)\,ds + C\delta.
$$
By the fractional Gr\"onwall inequality in Lemma \ref{lemma: fractional gronwall lemma}, we conclude, up to renaming the constant $C$,
$$
v^\delta(t)\leq C\delta, \qquad t\in(0,T],
$$
which concludes the proof.
\end{proof}
The next proposition upgrades the previous $L^1$ control to an $L^2$ estimate by combining the smoothing effect of the heat semigroup with Lemma \ref{lemma: L1 estimate w delta}.
\begin{proposition}
\label{prop:new_prop}
Assume $b, \eta \in L^\infty (\T^d, \mathbb{R}^d).$ Then there exists $C = C( \|b\|_\infty, \| \eta \|_{\infty}, d,T)>0$ such that for any $\tau \in (0,T)$ and $\delta \in [0,1)$ it holds
$$
\| w^{\delta}(\tau) \|_2 \leq C \delta \tau^{1/2 - d/4}.
$$
\end{proposition}
\begin{proof}
We divide the proof into two steps. We first use the mild formulation and split the time integral at $\tau/2$; in the second step, we conclude by a fractional Gr\"onwall argument.

\textit{Step $1$: estimate of $\|w^\delta(\tau)\|_2$.}
Let $S(t)=e^{\frac{t}{2}\Delta}$ be the heat semigroup on $\T^d$. Interpreting in distributional sense the following identity thanks to Lemma \ref{lemma: heat semigroup divergence}, we have
\begin{equation*}
    \begin{split}
        w^\delta(\tau) &= -\int_0^\tau S(\tau-s) \mathrm{div}(b w^\delta(s)) \, ds -\delta \int_0^\tau S(\tau-s) \mathrm{div}(\eta p^\delta(s)) \, ds \\
        &= -\int_0^\tau \nabla S(\tau-s)\cdot (b w^\delta(s))\, ds -\delta\int_0^\tau \nabla S(\tau-s)\cdot (\eta\,p^\delta(s))  \,ds .
    \end{split}
\end{equation*}
Splitting the two integrals at $\tau/2$, we obtain
\begin{equation*}
    \begin{split}
        \|w^\delta(\tau)\|_2
        \leq I_1+I_2+I_3+I_4,
    \end{split}
\end{equation*}
where
$$
        I_1 := \int_0^{\tau/2}\|\nabla S(\tau-s)\cdot (b w^\delta(s))\|_2\,ds, \qquad I_2 := \int_{\tau/2}^{\tau}\|\nabla S(\tau-s)\cdot (b\,w^\delta(s))\|_2\,ds
        $$
and
        $$
I_3 := \delta\int_0^{\tau/2}\|\nabla S(\tau-s)\cdot (\eta p^\delta(s))\|_2  \,  ds, \qquad  I_4 := \delta\int_{\tau/2}^{\tau}\|\nabla S(\tau-s)\cdot (\eta p^\delta(s))\|_2  \, ds .
$$
We estimate separately the four terms. For $I_1$, using the $L^1\to L^2$ smoothing estimate for $\nabla S(t)$ from Lemma \ref{lemma: heat semigroup divergence}, the fact that $\tau-s\geq \tau/2$ for $s\in[0,\tau/2]$, and the bound $\|w^\delta(s)\|_1 \leq C\delta s^{1/2}$ from Lemma \ref{lemma: L1 estimate w delta},
we get
\begin{equation*}
    \begin{split}
        I_1
        &\leq C \int_0^{\tau/2} (\tau-s)^{-1/2-d/4}\|b\,w^\delta(s)\|_1\,ds \leq C\|b\|_\infty \int_0^{\tau/2} (\tau-s)^{-1/2-d/4}\|w^\delta(s)\|_1\,ds \\
        &\leq C\delta \tau^{-1/2-d/4}\int_0^{\tau/2}s^{1/2}\,ds
        \leq C\delta \tau^{1-d/4}
        \leq C\delta \tau^{1/2-d/4},
    \end{split}
\end{equation*}
where, in the last inequality, we used that $\tau\leq T$ and absorbed the additional factor into the constant. For $I_2$, using the $L^2\to L^2$ estimate for $\nabla S(t)$ in Lemma \ref{lemma: smoothing estimates heat semigroup} and H\"older inequality, we obtain
$$
I_2 \leq C\int_{\tau/2}^{\tau} (\tau-s)^{-1/2}\|b\,w^\delta(s)\|_2 \, ds \leq C\|b\|_\infty \int_{\tau/2}^{\tau} (\tau-s)^{-1/2}\|w^\delta(s)\|_2 \, ds .
$$
For $I_3$, using again the $L^1\to L^2$ smoothing estimate for $\nabla S(t)$ from Lemma \ref{lemma: heat semigroup divergence}, $\tau-s\geq \tau/2$ for $s\in[0,\tau/2]$, and the bound $\|p^\delta(s)\|_1\leq 1$ from Lemma \ref{lemma: contraction on L1}, we get
\begin{equation*}
    \begin{split}
        I_3
        &\leq C\delta \int_0^{\tau/2}(\tau-s)^{-1/2-d/4}\|\eta p^\delta(s)\|_1\,ds \leq C\delta \|\eta\|_\infty \int_0^{\tau/2}(\tau-s)^{-1/2-d/4}\|p^\delta(s)\|_1\,ds \\
        &\leq C\delta \tau^{-1/2-d/4}\int_0^{\tau/2}ds
        \leq C\delta \tau^{1/2-d/4}.
    \end{split}
\end{equation*}
For $I_4$, again by Lemma \ref{lemma: heat semigroup divergence}, and using the bound $\|p^\delta(s)\|_2 \leq C s^{-d/4}$ from Corollary \ref{cor: semigroup estimates for dirac delta initial condition}, we have
\begin{equation*}
    \begin{split}
        I_4
        &\leq C\delta \int_{\tau/2}^{\tau}(\tau-s)^{-1/2}\|\eta    p^\delta(s)  \|_2  \,  ds \leq C\delta \|\eta\|_\infty \int_{\tau/2}^{\tau}(\tau-s)^{-1/2}    \|     p^\delta(s)     \|_2  \,   ds \\
        &  \leq C \delta \int_{\tau/2}^{\tau}    (\tau-s)^{-1/2}s^{-d/4}   \,   ds \leq C  \delta \tau^{-d/4}  \int_{\tau/2}^{\tau}    (\tau-s)^{-1/2}    \,  ds \leq C\delta \tau^{1/2-d/4}.
    \end{split}
\end{equation*}
Combining the estimates for $I_1$, $I_2$, $I_3$ and $I_4$, we deduce
\begin{equation}
    \|w^\delta(\tau)\|_2
    \leq C\delta \tau^{1/2-d/4}
    + C\int_{\tau/2}^{\tau}(\tau-s)^{-1/2}\|w^\delta(s)\|_2\,ds,
    \qquad 0<\tau\leq T,
    \label{eq: aux inequality wdelta L2}
\end{equation}
for a constant $C=C(\|b\|_\infty,\|\eta\|_\infty,d,T)>0$.

\textit{Step $2$: conclusion by fractional Gr\"onwall.}
Define
\[
h^\delta(t):= t^{-1/2+d/4}\|w^\delta(t)\|_2, \qquad t\in(0,T].
\]
Then \eqref{eq: aux inequality wdelta L2} becomes
\begin{equation*}
    \begin{split}
        h^\delta(\tau)
        &\leq C\delta
        + C\tau^{-1/2+d/4}\int_{\tau/2}^{\tau}(\tau-s)^{-1/2}s^{1/2-d/4}h^\delta(s)\,ds .
    \end{split}
\end{equation*}
Since for $s\in[\tau/2,\tau]$ one has $
\left(\frac{s}{\tau}\right)^{1/2-d/4}\leq c(d),$ we have that
$$
        h^\delta(\tau)
        \leq C\delta + c(d)C\int_{\tau/2}^{\tau}(\tau-s)^{-1/2}h^\delta(s)\,ds \leq C\delta + c(d)C\int_0^{\tau}(\tau-s)^{-1/2}h^\delta(s)\,ds .
$$
Thus, up to renaming the constant $C$, the fractional Gr\"onwall inequality in Lemma \ref{lemma: fractional gronwall lemma} gives
$$
h^\delta(\tau)\leq C\delta, \qquad 0<\tau\leq T.
$$
This is the thesis by recalling the definition of $h^\delta$.
\end{proof}

\subsection{$L^2$ initial data for the Fokker-Planck equation}

\subsubsection{Proof of Proposition \ref{prop: convergence density pdelta regular p0} }\label{proof: main result convergence p delta regular ic}
    Fix $p_0 \in L^2(\T^d)$. For each $\delta \in [0,1)$, let $p^{\delta,p_0}$ be the solution of the FPE , i.e.
    \begin{equation}
        \partial_t p^{\delta,p_0} = \frac{1}{2}\Delta p^{\delta, p_0} - \mathrm{div}(p^{\delta,p_0} (b + \delta \eta)), \qquad p^{\delta,p_0}|_{ t = 0} = p_0.
        \label{eq: FPE for p delta p0 appendix}
    \end{equation}
    Since $b, \eta \in L^{\infty}(\T^d; \mathbb{R}^d)$, the theory of linear parabolic equations guarantees the existence of a unique variational solution
    $$
    p^{\delta, p_0} \in   L^2(0,T; H^{1}(\T^d)) \cap H^1(0,T; H^{-1}(\T^d)) \cap  C([0,T];L^2(\T^d)). 
    $$
    Indeed, the bilinear form
$$
a^{\delta}: H^1(\T^d) \times H^1 (\T^d) \to \mathbb{R},\qquad a^\delta(u,v):=\frac12\int_{\T^d}\nabla u\cdot\nabla v\,dx
-\int_{\T^d} g^\delta u\cdot\nabla v\,dx
$$
is bounded on $H^1(\T^d)\times H^1(\T^d)$ and satisfies, thanks to H\"older and Young's inequality
$$
a^{\delta}(u,u) = \frac{1}{2} \| \nabla u \|_2^2 - \int_{\T^d} g^{\delta} u \cdot \nabla u \, dx \geq \lambda_1 \| u \|_{H^1(\T^d)}^2 - \lambda_2 \| u \|_2^2,
$$
with $\lambda_i >0$. Hence, by the classical theory of Lions-Magens (see \cite[Chapter 4]{Lions_MagensI}), for any $p_0 \in L^2(\T^d)$ there exists a unique solution
$$
p^{\delta,p_0} \in L^2(0,T; H^1(\T^d))\cap H^1(0,T; H^{-1}(\T^d)) \cap C([0,T]; L^2(\T^d)).
$$
such that $p^{\delta,p_0}$ solves the FPE \eqref{eq: FPE for p delta p0 appendix} in variational sense. By the classical identification between variational and distributional formulations, this solution is also the unique weak/distributional solution, and admits the mild representation recalled in \eqref{eq:mild-representation-appendix}.

    The difference quotient 
    $$
    r^{\delta,p_0} (t):= \frac{p^{\delta,p_0} (t) - p^{0, p_0} (t)}{\delta }, \qquad \delta \in (0,1),
    $$
    satisfies
    \begin{equation}
    \partial_t r^{\delta,p_0} = \frac{1}{2}\Delta r^{\delta,p_0} - \mathrm{div}( b r^{\delta,p_0}) - \mathrm{div}(\eta  p^{\delta ,p_0}), \qquad r^{\delta,p_0}|_{t = 0 } = 0.
        \label{eq:FPE_r_delta_app}
    \end{equation}
    We define the candidate derivative $\dot r ^{p_0}$ as the unique solution of
    \begin{equation}
    \partial_t \dot r ^{p_0} = \frac{1}{2}\Delta \dot r ^{p_0} - \mathrm{div}(b \dot r ^{p_0} ) - \mathrm{div}( \eta  p ^{0,p_0}), \qquad \dot r^{p_0}|_{ t = 0} = 0.
        \label{eq:FPE_der_app}
    \end{equation}
    Lastly, the error $z^{\delta, p_0}:= r^{\delta, p_0} - \dot r ^{p_0} $ satisfies 
    \begin{equation}
    \partial_t z^{\delta, p_0} = \frac{1}{2}\Delta z^{\delta, p_0} - \mathrm{div}( b  z^{\delta, p_0}) - \delta \mathrm{div} (\eta r^{\delta, p_0}), \qquad z^{\delta, p_0} |_{ t = 0} = 0.
        \label{eq:FPE_z_delta_app}
    \end{equation}
    We now proceed by performing $L^2$ bounds for $p^{\delta,p_0} $, $r^{\delta,p_0}$ and $z^{\delta, p_0}$.

    \textit{Step $1$: $L^2$ bound for $p^\delta$.} 
    Let $\tau \in [0,T]$. Denote by $S(t) = e^{\frac{t}{2}\Delta}$ the heat semigroup. The solution of \eqref{eq: FPE for p delta p0 appendix} is
    \begin{equation*}
        \begin{split}
             p^{\delta,p_0} (\tau ) &= S(\tau) p_0 - \int_0^\tau  S(\tau -s) \mathrm{div} (b p^{\delta, p_0} (s)) \, ds - \delta\int_0^\tau  S(\tau-s) \mathrm{div}(\eta p^{\delta,p_0} (s)) \, ds \\
             & = S(\tau ) p_0 - \int_0^\tau \nabla S(\tau -s)\cdot (b p^{\delta, p_0} (s)) \, ds -\delta  \int_0^\tau \nabla S(\tau -s) \cdot (\eta p^{\delta,p_0} (s)) \, ds,
        \end{split}
    \end{equation*}
    where in the second equality we have used the commuting property of the divergence with the heat semigroup. Applying the $L^2$-norm to the previous identity, using that $\|\nabla S(t) g\|_2 \lesssim t^{-1/2} \| g \|_2 $ thanks to Lemma \ref{lemma: heat semigroup divergence}, $\| S(\tau) p_0 \|_2 \leq \| p_0\|_2$ thanks to Lemma \ref{lemma: smoothing estimates heat semigroup}(i) and applying H\"older inequality, we get
    \begin{equation*}
        \begin{split}
            \| p^{\delta,p_0} (\tau ) \|_2 &\leq  \| S(\tau ) p_0\|_2  +\int_0^\tau \| \nabla S(\tau -s)\cdot (b p^{\delta, p_0} (s))\|_2 \, ds + \int_0^\tau \| \nabla S(\tau -s) \cdot (\eta p^{\delta,p_0} (s))\|_2 \, ds \\
            &\leq \| p_0 \|_2 + c \int_0^\tau (\tau -s)^{-1/2 } \| b p^{\delta, p_0} (s) \|_2 \, ds + c \int_0^{\tau} (\tau -s)^{-1/2} \| \eta p^{\delta,p_0} (s) \|_2 \, ds \\
            & \leq \| p_0 \|_2 + c \| b \|_\infty \int_0^{\tau } (\tau -s)^{-1/2} \| p^{\delta,p_0}(s) \|_2 \, ds + c \| \eta \|_\infty \int_0^{\tau} (\tau -s)^{-1/2} \| p^{\delta, p_0} (s) \|_2 \, ds\\
            & \leq \| p_0 \|_2 + c (\| b\|_\infty + \| \eta \|_\infty )\int_0^\tau (\tau -s)^{-1/2} \| p^{\delta,p_0} (s) \|_2\,ds, \qquad 0 \leq \tau \leq T,\; \delta \in [0,1),
        \end{split}
    \end{equation*}
    for a constant $c = c(T, d)>0$. Applying the fractional Gr\"onwall inequality from Lemma \ref{lemma: fractional gronwall lemma}, we get that
    \begin{equation}
         \| p^{\delta, p_0} (\tau ) \|_2 \leq C \| p_0 \|_2, \qquad 0 \leq \tau \leq T,\; 0 \leq \delta < 1, 
        \label{eq:L2_unif_bound_p_delta_app}
    \end{equation}
    for a constant $C = C(\| b \|_\infty, \| \eta \|_\infty, T, d)>0$.

    \textit{Step $2$: $L^2$ bound for $r^\delta$.} 
    Let $\tau \in [0,T]$. Using the variation-of-constants formula in the sense of distributions, together with Lemma \ref{lemma: heat semigroup divergence}, \eqref{eq:FPE_r_delta_app} takes the form
    $$
      r^{\delta,p_0} (\tau) = -\int_0 ^\tau \nabla S (\tau -s) \cdot (b r^{\delta, p_0} (s) ) \, ds- \int_0^{\tau } \nabla S(\tau -s) \cdot (\eta p^{\delta,p_0} (s)) \, ds.
    $$
    Reasoning as before, hence applying the $L^2$-norm to the previous identity, using that $\|\nabla S(t) g\|_2 \lesssim t^{-1/2} \| g \|_2 $ thanks to Lemma \ref{lemma: heat semigroup divergence}, and applying H\"older inequality, we get
    \begin{equation*}
        \begin{split}
            \|  r^{\delta,p_0} (\tau) \|_2 &\leq  \int_0 ^\tau \| \nabla S (\tau -s) \cdot (b r^{\delta, p_0} (s) )\|_2 \, ds +\int_0^{\tau } \| \nabla S(\tau -s) \cdot (\eta p^{\delta ,p_0} (s))\|_2 \, ds \\
            & \leq c \| b \|_\infty \int_0^{\tau} (\tau -s)^{-1/2} \| r^{\delta,p_0} (s) \|_2 \, ds + c \| \eta \|_\infty \int_0^\tau (\tau -s)^{-1/2} \|p^{\delta , p_0} (s) \|_2 \, ds,
        \end{split}
    \end{equation*}
    for a constant $c =c(T,d)>0.$ Using the bound in \eqref{eq:L2_unif_bound_p_delta_app} and $\int_0^{\tau} (\tau -s)^{-1/2}\, ds =2 \tau ^{1/2}$, we obtain
    $$
    \| r^{\delta,p_0} (\tau) \|_2 \leq C \int_0^\tau (\tau -s)^{-1/2} \| r^{\delta, p_0} (s) \|_2 \, ds +C \| p_0 \|_2 \tau^{1/2},
    $$
    for a constant $C = C(\| b \|_\infty, \| \eta \|_\infty , T, d)>0$. Then $v^\delta (t):= t^{-1/2}\| r^{\delta,p_0} (t) \|_2$ satisfies
    $$
    v^\delta (\tau) \leq C \tau^{-1/2}  \int_0^\tau (\tau -s)^{-1/2} s^{1/2}v^{\delta}(s) \, ds +C \| p_0 \|_2  \leq C \int_0^{\tau }(\tau -s)^{-1/2} v^\delta (s) + C \| p_0 \|_2,
    $$
    where in the second inequality we have used $\tau ^{-1/2} s^{1/2} \leq 1$ for $s \in [0, \tau]$. Applying the fractional Gr\"onwall inequality from Lemma \ref{lemma: fractional gronwall lemma}, we obtain
    $$
    v^\delta (\tau ) \leq C\| p_0 \|_2, \qquad 0 \leq \tau \leq T, \; \delta  \in [0,1),
    $$
    for a (not renamed) new constant $C = C(\| b \|_\infty, \| \eta \|_\infty , T, d)>0$. Hence,
    \begin{equation}
    \|r^{\delta, p_0}(\tau) \|_2 \leq C \tau^{1/2}\| p_0 \|_2 \qquad 0 \leq \tau \leq T, \; \delta  \in [0,1),
        \label{eq:L2_bound_r_delta_app}
    \end{equation}
    for $C = C(\| b \|_\infty, \| \eta \|_\infty , T, d)>0$.

\textit{Step $3$: $L^2$ bound for $z^{\delta, p_0}$.}  For any $\tau \in [0,T]$, the solution of \eqref{eq:FPE_z_delta_app} is
$$
z^{\delta, p_0} (\tau) = - \int_0^{\tau } \nabla S(\tau -s) \cdot ( b z^{\delta,p_0} (s)) \, ds  -\delta \int_0^\tau  \nabla S(\tau -s)  \cdot ( \eta r^{\delta, p_0} (s) ) \, ds
$$
Applying the $L^2$-norm, using the triangle inequality, $\| \nabla S(t) g\|_2 \lesssim t^{-1/2} \| g \|_2$ thanks to Lemma \ref{lemma: heat semigroup divergence}, and  H\"older inequality, we obtain
\begin{equation*}
    \begin{split}
        \|z^{\delta, p_0} (\tau)\|_2  &\leq   \int_0^{\tau } \| \nabla S(\tau -s) \cdot ( b z^{\delta,p_0} (s))\|_2 \, ds + \delta \int_0^\tau  \| \nabla S(\tau -s)  \cdot ( \eta r^{\delta, p_0} (s) )\|_2 \, ds\\
        & \leq c \int_0^\tau (\tau -s)^{-1/2} \| b z^{\delta, p_0} (s) \|_2 \, ds+ c \delta \int_0^{\tau } (\tau -s)^{-1/2} \|\eta r^{\delta, p_0} (s) \|_2 \, ds \\
         & \leq c \| b \|_\infty \int_0^\tau  (\tau -s)^{-1/2} \| z^{\delta, p_0} (s) \|_2 \, ds + c \delta \| \eta \|_\infty \int_0^\tau (\tau -s)^{-1/2} \| r^{\delta, p_0} (s) \|_2 \, ds,
    \end{split}
\end{equation*}
    for a constant $c = c(T,d)>0$. Using \eqref{eq:L2_bound_r_delta_app} on the last term on the right-hand side of the previous expression, we get
\begin{equation*}
    \begin{split}
        \| z^{\delta, p_0} (\tau ) \|_2 
        &\leq C \int_0^\tau (\tau -s)^{-1/2} \| z^{\delta, p_0} (s) \|_2 \, ds 
        + C \delta \int_0^\tau (\tau -s)^{-1/2} \| r^{\delta, p_0} (s) \|_2 \, ds \\
        &\leq C \int_0^\tau (\tau -s)^{-1/2} \| z^{\delta, p_0} (s) \|_2 \, ds 
        + C \delta \| p_0 \|_2 \int_0^\tau (\tau -s)^{-1/2} s^{1/2} \, ds \\
        &\leq C \int_0^\tau (\tau -s)^{-1/2} \| z^{\delta, p_0} (s) \|_2 \, ds 
        + C \delta \tau \| p_0 \|_2,
    \end{split}
\end{equation*}
where we used that $\int_0^\tau (\tau -s)^{-1/2} s^{1/2}\, ds = \frac{\pi}{2}\tau$ and $C = C( \|  b\|_\infty , \| \eta \|_{\infty }, T,d)>0$ is a constant absorbing the other constants. For $\tau \in (0,T]$, define
$$
w^\delta(\tau) := \tau^{-1}\| z^{\delta, p_0}(\tau)\|_2 .
$$
Dividing the previous inequality by $\tau$ and using that $s/\tau \leq 1$ for $s \in [0,\tau]$, we obtain
$$
w^\delta(\tau) \leq C \int_0^\tau (\tau -s)^{-1/2} w^\delta(s)\, ds + C \delta \| p_0 \|_2 .
$$
 Lastly, the fractional Gr\"onwall inequality from Lemma \ref{lemma: fractional gronwall lemma} leads to
    $$
    \| z^{\delta,p_0} (t) \|_2 \leq C \delta t \| p_0 \|_2, \qquad t \in [0,T], \; \delta \in (0,1),
    $$
    with $C = C( \|  b\|_\infty , \| \eta \|_{\infty }, T,d)>0$. \qed

Fix $\tau \in (0,T)$. Recall that in Section \ref{sec: SDEs}, given an initial condition $p_0 \in L^2(\T^d)$ for the FPE \eqref{eq: FPE with regular initial condition}, we have introduced the map
$$
    {D}_{p_0,\tau }^{L^\infty } :L^{\infty}(\mathbb{T}^d; \mathbb{R}^d)\to L^2(\mathbb{T}^d), \qquad \eta \mapsto \dot{r}^{p_0,\eta} (\tau):= \dot r^{p_0},
$$
where $\dot r^{p_0}$ denotes the solution of \eqref{eq: eq for dot r regular initial condition}, i.e.
$$
\partial_t \dot { {r}}^{p_0}  = \frac{1}{2}\Delta \dot{ r }^{p_0}  - \mathrm{div} (b \dot{ r}^{p_0} ) - \mathrm{div} ( \eta  p ^{0,p_0} ), \qquad \dot r^{p_0} |_{t = 0} = 0.
$$
The map $D_{p_0,\tau }^{L^\infty }$ represents the derivative with respect to $\delta $ in direction $\eta$ of the solution $p^{\delta, p_0}$ of the FPE \eqref{eq: FPE with regular initial condition} that has $p^0$ as initial condition. We are now ready to prove Proposition \ref{cont}, that asserts the linearity and continuity of $D_{p_0,\tau }^{L^\infty }$.

\subsubsection{Proof of Proposition \ref{cont}}
\label{proof: derivative linear and bounded}
    We denote for simplicity by $\dot r = \dot r^{p_0}$ the solution of \eqref{eq: eq for dot r regular initial condition}, and by $p^\delta = p^{\delta, p_0}$ the solution of \eqref{eq: FPE with regular initial condition}. Let $S(t) = e^{\frac{t}{2} \Delta}$ be the heat semigroup on $L^2 (\mathbb{T}^d).$ The solution of \eqref{eq: eq for dot r regular initial condition} is given by
    \begin{equation}
        \begin{split}
            \dot r (\tau) &= -\int_0^{\tau } S(\tau -s) \mathrm{div}(b \dot r (s)) \, ds - \int_0^\tau S (\tau -s) \mathrm{div}(\eta p^0 (s)) \, ds \\
            &
             = -\int_0^{\tau } \nabla S(\tau -s) \cdot (b \dot r (s)) \, ds - \int_0^\tau \nabla S (\tau -s) \cdot (\eta p^0 (s)) \, ds,
        \end{split}
        \label{eq: mild formulation dot r}
    \end{equation}
    where in the previous equality we have used Lemma \ref{lemma: heat semigroup divergence} to interpret it in distributional sense. The mild formulation \eqref{eq: mild formulation dot r} shows that $D_{p_0,\tau }^{L^\infty} $ is linear.

    To prove that $D_{p_0,\tau }^{L^\infty }$ is continuous, we show that there exists $K = K (\| b \|_{{\infty}},T,d)>0$ such that for any $\eta \in L^\infty (\mathbb{T}^d; \mathbb{R}^d)$, it holds
    \begin{equation}
        \| \dot r (\tau) \|_2 \leq K \tau ^{1/2} \| p_0 \|_2 \| \eta \|_\infty , \qquad 0 <  \tau \leq  T.
        \label{eq: bound dot r}
    \end{equation}
    Considering the $L^2$-norm in \eqref{eq: mild formulation dot r}, applying the triangle inequality and using that $\| \nabla S(t) g \|_2 \leq Ct^{-1/2}\| g \|_2$ thanks to Lemma \ref{lemma: heat semigroup divergence}, we get
    \begin{equation}
        \begin{split}
            \| \dot r (\tau) \|_2 & \leq C\int_0^\tau  \| \nabla  S(\tau -s) \cdot (b \dot r (s)) \|_2\, ds + \int_0^\tau \| S(\tau -s ) \mathrm{div}( \eta p^0(s))\|_2 \, ds \\
            &\leq C\int_0^\tau (\tau-s)^{-1/2} \| b \dot r (s) \|_2 \, ds + \int_0^\tau  (\tau -s)^{-1/2} \| \eta p^0 (s) \|_2\, ds  \\
            & \leq  C \| b \|_\infty \int_0^\tau (\tau -s)^{-1/2} \| \dot r (s) \|_2 \, ds + C \| \eta \|_\infty \int_0^\tau (\tau -s)^{-1/2} \| p^0 (s) \|_2 \, ds,
        \end{split}
        \label{eq: L2 estimate dot r aux}
    \end{equation}
    where in the last inequality we have used H\"older inequality. Using that $\sup_{0 \leq t \leq T}\| p^\delta (t) \|_2 \leq c \| p_0 \|_2 $ for a constant $c = c(\|b \|_{\infty}, \| \eta \|_{\infty}, d,T)>0$ thanks to \eqref{eq:L2_unif_bound_p_delta_app}, we deduce 
    $$
    \| \dot r (\tau) \|_2 \leq C \int_0^\tau (\tau -s )^{-1/2} \| \dot r (s) \|_2 \, ds + C \| p_0 \| \int_0^\tau (\tau -s)^{-1/2} \, ds \leq C \int_0^\tau (\tau -s )^{-1/2} \| \dot r (s) \|_2 \, ds + \tau^{1/2} C \| p_0\| \| \eta \|_\infty ,
    $$
    for a new constant $ C = C( \| b \|_{\infty},d, T)>0$. Then, the bound \eqref{eq: bound dot r} follows by applying the fractional Gr\"onwall inequality from Lemma \ref{lemma: fractional gronwall lemma} to $v(\tau):= \tau^{-1/2} \| \dot r(\tau) \|_2$, as at the end of the proof of Proposition \ref{prop:new_prop}.

    We conclude the proof by showing that $D_{p_0, \tau}^{L^\infty }(\eta)$ has zero average. Since $\dot r=\dot r^{p_0,\eta}$ solves
    $$
    \partial_t \dot r
    = \frac12 \Delta \dot r - \mathrm{div}(b\dot r)-\mathrm{div}(\eta p^0),
    \qquad
    \dot r|_{t=0}=0,
    $$
    integrating over $\T^d$ and using periodicity gives
    $$
    \frac{d}{dt}\int_{\T^d}\dot r(t,x)\,dx
    =
    \frac12\int_{\T^d}\Delta \dot r(t,x)\,dx
    -
    \int_{\T^d}\mathrm{div}(b\dot r)(t,x)\,dx
    -
    \int_{\T^d}\mathrm{div}(\eta p^0)(t,x)\,dx
    =0.
    $$
    Since $\dot r(0,\cdot)=0$, we obtain
    $$
    \int_{\T^d}\dot r(t,x)\,dx=0
    \qquad\text{for all }t\in[0,T].
    $$
    In particular,
    $$
    \int_{\T^d} D_{p_0,\tau}^{L^\infty }(\eta)\,dx =
    \int_{\T^d}\dot r(\tau,x)\,dx
    =0,
    $$
    hence the thesis.
    
\qed

{\bf Acknowledgements. }{\small Gianmarco Del Sarto acknowledges the support from the DFG project FOR~5528.  The research of Franco Flandoli is funded by the PRIN-PNRR project “Some mathematical
approaches to climate change and its impacts” n. P20225SP98 and the European Union (ERC, NoisyFluid, No. 101053472). Views and opinions expressed are however those of the authors only and do not necessarily reflect those of the European Union or the European Research Council. Neither the European Union nor the granting authority can be held responsible for them. 

Stefano Galatolo acknowledges the MIUR Excellence Department Project awarded to the Department of Mathematics, University of Pisa, CUP I57G22000700001.

Part of this project was conducted while Sakshi Jain was affiliated with Monash University whose support is gratefully acknowledged.

}

\bibliographystyle{plain}
\bibliography{biblio}

\end{document}